\documentclass[12pt]{amsart}
\usepackage{latexsym}
\usepackage{amssymb}
\usepackage{inputenc}
\usepackage{epsfig}
\usepackage{color}
\usepackage{amsthm}
\usepackage{amscd}
\usepackage{amsmath}
\usepackage{amsfonts}
\usepackage{graphics}

\usepackage[all]{xy}
\usepackage{enumerate}

\newtheorem{theorem}{Theorem}[section] \newtheorem{philosophy}{Statement}[theorem]
\newtheorem{lemma}[theorem]{Lemma}
\newtheorem{corollary}[theorem]{Corollary} 
\newtheorem{prop}[theorem]{Proposition}
\newtheorem{claim}[theorem]{Claim}
\newtheorem*{claim*}{Claim}
\newtheorem{fact}[theorem]{Fact}
\newtheorem*{theorem*}{Theorem} 
\newtheorem*{corollary*}{Corollary}
\theoremstyle{definition}
\newtheorem{example}[theorem]{Example}
\newtheorem{remark}[theorem]{Remark}
\newtheorem{definition}[theorem]{Definition}
\newtheorem*{remark*}{Remark}

\newtheorem*{definition*}{Definition}
\newtheorem*{example*}{Example}

\newtheorem*{namedtheorem}{\theoremname}
\newcommand{\theoremname}{testing}

\theoremstyle{remark}

\newcommand{\BC}{\mathbb C} \newcommand{\BH}{\mathbb H}
\newcommand{\BR}{\mathbb R} 
\newcommand{\BN}{\mathbb N} 
 \newcommand{\BZ}{\mathbb Z}

\newcommand{\CA}{\mathcal A} 

 \newcommand{\CT}{\mathcal T}

\newcommand{\actson}{\curvearrowright}

\DeclareMathOperator{\PSL}{PSL} 

\DeclareMathOperator{\length}{length}

\DeclareMathOperator{\dist}{d}

\DeclareMathOperator{\Mod}{Mod}
\newcommand{\note}[1]{\marginpar{\tiny \raggedright \color{red} #1 }}
\newcommand{\PML}{{\mathcal P \mathcal M \mathcal L}}

\newcommand{\into}{\hookrightarrow}

\newcommand{\comment}[1]{}

\title{Homoclinic leaves, Hausdorff limits, and homeomorphisms}

\author{Ian Biringer}
\address{\hskip-\parindent
        Department of Mathematics \\
         Boston College\\
       140  Commonwealth Avenue \\
        Chestnut Hill, MA 02467 \\
        USA}
\email{biringer@bc.edu}

\author{Cyril Lecuire}
\address{\hskip-\parindent
        Toulouse Mathematics Institute \\
        Universit\'e Paul Sabatier\\
		118, route de Narbonne\\
		F-31062 Toulouse Cedex 9\\
        France}
\email{cyril.lecuire@math.univ-toulouse.fr}

\begin{document}

\begin{abstract}
We show that except for one exceptional case, a lamination on the boundary of a $3$-dimensional handlebody $H$ is a Hausdorff limit of meridians if and only if it is commensurable to a lamination with a `homoclinic leaf'. This is a precise version of a philosophy called Casson's Criterion, which appeared in unpublished notes of A.\ Casson. Applications include a characterization of when a non-minimal lamination is a Hausdorff limit of meridians, in terms of properties of its minimal components, and a related characterization of which reducible self-homeomorphisms of $\partial H$ have powers that extend to subcompressionbodies of $H$. \end{abstract}

\maketitle


\section{Introduction}

Let $H$ be a $3$-dimensional handlebody\footnote{The body of this paper is written in greater generality, with the pair $(H,S)$ replaced by a compact, orientable, $3$-manifold $M$ with hyperbolizable interior, together with an essential connected subsurface $S\subset \partial M$ such that the multi-curve $\partial S$ is incompressible in $M$. However, everything we do is just as interesting in the handlebody case.} with genus $g \geq 2$ and let $S:=\partial H$. A simple closed curve $m $ on $S$ is called a \emph {meridian} if it bounds an embedded disk in $H$ but not in $S$. Equip $S$ with an arbitrary hyperbolic metric, and consider the set of geodesic laminations on $S$ with the Hausdorff topology. We refer the reader to \cite{Bleilerautomorphisms} for more information on laminations.

\subsection{Homoclinic leaves} In J.P.\ Otal's thesis \cite{Otalcourants} the following is stated; it is  attributed to an unpublished manuscript of A. Casson. 

\begin{philosophy}['Casson's Criterion']\label{ccr}
	A geodesic lamination on $S$ is a Hausdorff limit of meridians if and only if it has a homoclinic leaf.
\end{philosophy}

We call it a `statement' here instead of a theorem because it is not true as written, as we'll see later on. However, the connection between homoclinic leaves and meridians has been well studied, partly motivated by this statement, see for example the papers \cite{Otalcourants,Lecuireplissage,Lecuireextension,Kleineidamalgebraic,namazi2012non-realizability,jeon2014primitive} and Long's earlier paper \cite{long1986discs}.

To define homoclinic, let $\tilde H$ be the universal cover of $H$, which is homeomorphic to a thickened infinite tree. A path $\ell : \BR \longrightarrow \tilde H$ is called \emph {homoclinic} if there are sequences $s_i,t_i\in \BR$ such that $$|s_i-t_i|\to \infty, \ \text{and} \  \sup_i d_{\tilde H}(\tilde \ell(s_i),\tilde \ell(t_i))<\infty.$$
Here, distance in $\tilde H$ is measured using the lift of any Riemannian metric on $H$. Since $H$ is compact, the choice of metric does not matter. A path $\ell : \BR \longrightarrow S$ is called \emph{homoclinic} if it has a lift to $\partial \tilde H$ that is homoclinic as above, and a complete geodesic on $S$ is called \emph{homoclinic} if it has a (possibly periodic, if the geodesic is closed) arc length parametrization that is homoclinic.

As an example, any geodesic meridian $m$ on $S$ is homoclinic, for as $m$ lifts to a simple closed curve in $\partial \tilde H$, any lift of a periodic parameterization of $m$ is also periodic, and therefore homoclinic. On the other hand, if $\gamma$ is an essential simple closed curve that is not a meridian, any periodic parameterization of $\gamma$ lifts to a properly embedded biinfinite path in $\partial \tilde H$ that is invariant under a nontrivial deck transformation, and is readily seen to be non-homoclinic.

We should mention that the definition introduced above is not quite what Casson and Otal called `homoclinic' (in French, `homoclinique'), but rather what Otal calls `faiblement homoclinique', or `weakly homoclinic'. However, the definition above has been adopted in most subsequent papers.  While some of the discussion below is incorrect if we use Casson's definition, it can all be modified to apply. In particular, the same counterexamples show that Statement \ref{ccr} is still false using Casson's original definition. See \S \ref{alternate}.

\medskip

\begin{figure*}
	\centering
	\includegraphics{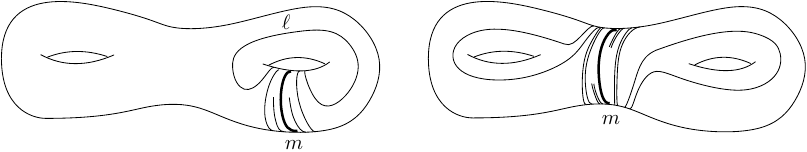}
	\caption{Laminations on the boundary of a genus two handlebody that have a meridian $m$ as a leaf, but are not Hausdorff limits of meridians.}\label {hlim}
\end{figure*}

  In his thesis \cite{Otalcourants}, Otal showed that any Hausdorff limit of meridians has a homoclinic leaf. (This statement was later extended by Lecuire \cite{lecuire2004structures} from handlebodies to more general $3$-manifolds.)  However, the converse is not true. First of all, any Hausdorff limit of meridians is connected, and there are disconnected laminations on $S$ that have homoclinic leaves, e.g.\ the union of two disjoint simple closed curves, one of which is a meridian. There are also connected laminations with homoclinic leaves that are not Hausdorff limits of meridians. For example, let $\lambda = m \cup \ell$ be a lamination with two leaves, where $m$ is a nonseparating meridian and the two ends of $\ell$ spiral around $m$ in the same direction, but from opposite sides, as on the left in Figure \ref{hlim}. Then $\lambda$ has a homoclinic leaf, but it is not a Hausdorff limit of simple closed curves: any simple geodesic that approximates $\ell$ closely is trapped and forced to spiral forever around $m$, so cannot be closed. 
  
  As another example, let $\lambda$ be the lamination on the right in Figure \ref{hlim}, which has three leaves, a meridian $m$ and two leaves spiraling onto it. Then $\lambda$ is a Hausdorff limit of simple closed curves, but it is not a limit of meridians. Indeed, given a simple closed curve $\mu$ on $S$, an arc $\alpha \subset \mu$ is called an \emph{$m$-wave disjoint from $m$} if it is homotopic rel endpoints in $H$ to an arc of $m$ and $int(\alpha) \cap m=\emptyset$. Any meridian $\mu$ that intersects $m$ has an $m$-wave disjoint from $m$: one looks for an `outermost' arc of intersection with a disk bounded by $\mu$ in a disk bounded by $m$. However, no simple closed curve that is Hausdorff-close to $\lambda$ has an $m$-wave disjoint from $m$, so $\lambda$ is not a limit of meridians. See the discussion after the statement of Theorem \ref{hlimits} for more details. 
  
  \begin{figure}[t]
  	\centering
  	\includegraphics{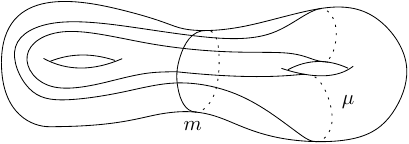}
  	\caption{Two meridians on a genus two handlebody. The two arcs of $\mu$ that lie on the right side of $m$ are $m$-waves.} 
  \end{figure}

In both of these examples, the problem lies with the spiraling isolated leaves. One way to address this is as follows. We say that two laminations $\mu_1,\mu_2$ on $S$ are \emph {commensurable} if they contain a common sublamination $\nu$ such that  for both $i$,  the difference $\mu_i \setminus \nu$  is the union of finitely many isolated leaves.  We say $\mu_1,\mu_2$ are \emph {strongly commensurable} if they contain a common $\nu$ such that for both $i$, the difference $\mu_i \setminus \nu$  is the union of finitely many isolated leaves, none of which are simple closed curves. 

So, is Casson's Criterion at least true up to strong commensurability? It turns out the answer is still no: the lamination $\lambda$ in Figure \ref{counterex} contains a meridian as a leaf, but it is not strongly commensurable to a limit of meridians. Indeed, suppose that $\lambda'$ is a Hausdorff limit of meridians that contains $\lambda$. In each component $T \subset S\setminus m$, there is a \emph{unique} homotopy class rel $m$ of $m$-waves in $T$. So $\lambda'$ contains a leaf $\ell$ that either intersects $T$ in an arc in this homotopy class, or is contained in $T$ and is obtained by spinning an arc in this homotopy class around $\mu$. In either case, $\ell$ intersects transversely the component of $\lambda$ contained in $T$, a contradiction. One can also make similar examples of laminations that contain a meridian but are not even \emph{commensurable} to a Hausdorff limit of meridians, by replacing the two curves on either side of $m$ in Figure \ref{counterex} with minimal laminations that fill the two components of $S \setminus m$. 
\begin{figure}[b]
	\centering
	\includegraphics{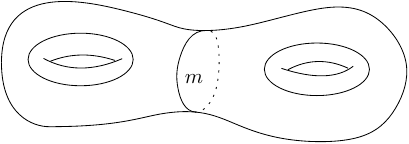}
	\caption{A lamination that contains a meridian as a leaf, but is not strongly commensurable to a Hausdorff limit of meridians.}\label {counterex}
\end{figure}

It turns out, though, that this is basically the only counterexample. Let's say that a lamination $\lambda$ on $S$ is \emph{exceptional} if $S$ has genus $2$, there is a separating meridian $m$ on $S$ that is either disjoint from $\lambda$ or is a leaf of $\lambda$, and there are minimal sublaminations of $\lambda$ that fill the two components of $S \setminus m$.

\begin{theorem}[A weak Casson Criterion, see Theorem \ref{hlimits}]\label{hlimitsintro}
If $\lambda$ is a geodesic lamination on $S$ that is not exceptional, then $\lambda$ is strongly commensurable to a Hausdorff limit of meridians if and only if it is strongly commensurable to a lamination with a homoclinic leaf.
\end{theorem}

See Theorem \ref{hlimits} for a more general statement and for the proof. In our view, this is the strongest version of Casson's criterion that is likely to be true for arbitrary geodesic laminations. It may be, though, that the original Casson Criterion is true for \emph{minimal} laminations. Our methods only work up to strong commensurability, though. For instance, if $\lambda$ is minimal filling on $S$ and contains a homoclinic leaf, to prove that $\lambda$ is a limit of meridians one would have to ensure that the meridians produced in Lemma \ref{tightcuts} do not run across any diagonals of the ideal polygons that are components of $S \setminus \lambda$.

The main tool in the proof of Theorem \ref{hlimitsintro} is a complete characterization of the minimal laminations onto which the two ends of a  homoclinic simple geodesic on $S$ can accumulate.

\begin{theorem}[Limits of homoclinic geodesics, see Corollary \ref{travelers-biinfinite}]\label{travelersintro}
	Suppose that $h$ is a homoclinic simple biinfinite geodesic on $S$ and that the two ends of $h$ limit onto minimal laminations $\lambda_-,\lambda_+ \subset S$. Then either
	\begin{enumerate}
		\item the two ends of $h$ are asymptotic on $S$,
		\item one of $\lambda_-,\lambda_+$ is an intrinsic limit of meridians, or 
		\item $\lambda_-,\lambda_+$ are contained in incompressible subsurfaces $S_-,S_+ \subset S$ that bound an essential interval bundle $B\subset H$ through which $\lambda_-$ and $\lambda_+$ are homotopic.
	\end{enumerate}
\end{theorem}

\begin{figure}
	\centering
	\includegraphics{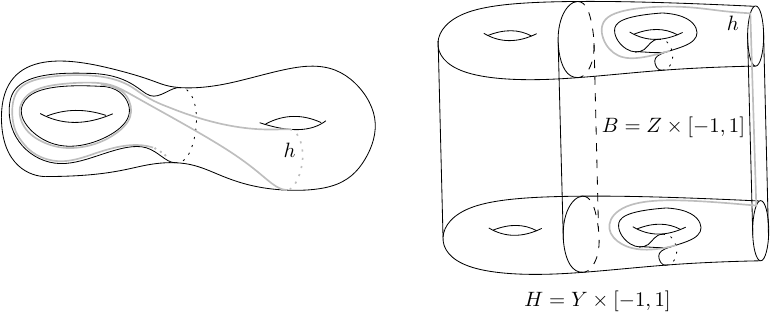}
	\caption{Examples of homoclinic geodesics (in grey) satisfying (3) in Theorem \ref{travelersintro}.}
	\label{condition3}
\end{figure}

Here, a minimal lamination  $\lambda \subset S$ is \emph{an intrinsic limit of meridians} if it is strongly commensurable to the Hausdorff limit of a sequence of meridians that are contained in the smallest essential subsurface $S(\lambda) \subset S$ containing $\lambda$, see Proposition \ref{intrinsiclimits} for a number of equivalent definitions. We refer the reader to Theorem \ref{travelers} and Corollary \ref{travelers-biinfinite} for more precise and more general versions of the above that apply both to homoclinic biinfinite geodesics, and also to pairs of `mutually homoclinic' geodesic rays on $S$. 

Examples of (3) are shows in Figure \ref{condition3}. On the left, $\lambda_-,\lambda_+$ are simple closed curves that bound an embedded annulus $A$ in $H$ and $B$ is a regular neighborhood of $A$. The geodesic $h$ is homoclinic since the annulus $A$ lifts to an embedded infinite strip $\BR\times [-1,1]\subset \tilde H$ and the two ends of a lift of $h$ are asymptotic to $\BR_+ \times \{-1\}$ and $\BR_+ \times \{1\}$, respectively. On the right, we write $H=Y\times [-1,1]$ where $Y$ is a genus two surface with one boundary component. The laminations $\lambda_\pm$ are minimal (in the picture they are drawn as `train tracks') and fill $Z\times \{\pm 1\}$, where $Z\subset Y$ is a torus with two boundary components. Here, $B=Z\times [-1,1]$.
Interval bundles are essential to the study of meridians on handlebodies, and it is no surprise that they appear in Theorem \ref{travelersintro}. For example, subsurfaces bounding such interval bundles are the `incompressible holes' studied by Masur-Schleimer \cite{masur2013geometry}, and interval bundles appear frequently in Hamenst\"adt's work on the disk set, see e.g.\ \cite{hamenstadt2016hyperbolic,hamenstadt2019asymptoticI,hamenstadt2019asymptoticII}. We note that the interval bundles $B$ appearing in Theorem \ref{travelersintro} may be twisted interval bundles over non-orientable surfaces, in which case $\lambda_-=\lambda_+$ and $S_-=S_+$. See \S\ref{intervalbundle} for background on interval bundles.

\subsection{Hausdorff limits via their minimal sublaminations} The previous two theorems suggest that if a lamination $\lambda$ that is a Hausdorff limits of meridians, one might expect to see minimal sublaminations $\lambda$ that are intrinsic limits of meridians, or pairs of components that are homotopic through essential interval bundles in $H$. In fact, we show the following.

\begin{theorem}[see Theorem \ref{hlimits}]\label{hlimits2intro}
Suppose that $\lambda \subset S$ is a nonexceptional geodesic lamination that is a finite union of minimal components. Then $\lambda $ is strongly commensurable to a Hausdorff limit of meridians if and only if either
\begin{enumerate}
\item $\lambda$ is disjoint from a meridian on $S$,
\item some component of $\lambda $ is an intrinsic limit of meridians, or
\item there are components $\lambda_\pm \subset \lambda$ that fill incompressible subsurfaces $S_\pm \subset S$, such that $S_\pm$ bound an essential interval bundle $B\subset H$, the laminations $\lambda_\pm$ are essentially homotopic through $B$, and there is a compression arc $\alpha$ for $B$ that is disjoint from $\lambda $.
\end{enumerate}
\end{theorem}

In (3), a \emph {compression arc} for $B$ is an arc from $\partial S_-$ to $\partial S_+$ that is homotopic in $H$, keeping its endpoints in $\partial S_\pm$, to a fiber of the interval bundle $B$. See \S \ref{compressionsec} and Figure \ref{spanningfig} for more explanation.

Note that Theorem \ref{hlimits2intro} does not say anything interesting about which {minimal filling} laminations $\lambda$ on $S$ are strongly commensurable to Hausdorff limits of meridians, just that that happens if and only if (2) holds, which is trivial. Indeed, for minimal filling laminations, it is not clear that there {should} be an easy way to `identify' Hausdorff limits of meridians. The point of Theorem \ref{hlimits2intro}, though, is that it reduces the characterization of Hausdorff limits of meridians to the minimal filling case. We note that it should be possible to replace the part of the proof of Theorem \ref{hlimits2intro} that references homoclinic geodesics with arguments similar to those used in Masur-Schleimer's paper \cite{masur2013geometry}.

\subsection{Extension of reducible maps to compression bodies}
As another application of our techniques, we consider extension properties of  homeomorphisms $f : S \longrightarrow S$. To motivate this, recall that a  \emph{subcompression body} of $H$ is a $3$-dimensional submanifold $C \subset H$ with $S \subset \partial C$ that is obtained by choosing a finite collection $\Gamma $ of disjoint meridians on $S$, taking a regular neighborhood of $S$ and a collection of discs in $H$ with boundary $\Gamma$, and adding in any complementary components that are topological $3$-balls. We say $C$ is obtained by \emph{compressing} $\Gamma$. We usually consider subcompression bodies only up to isotopy, and we allow the case that $\Gamma=\emptyset$, in which case we recover the \emph{trivial subcompression body}, which is just a regular neighborhood of $S$. 

Two examples where $\Gamma$ contains a single meridian are drawn in Figure \ref{cbodyfig}. On the left, we compress a separating meridian $m_1$ and obtain a subcompression body of $H$ that has two `interior' boundary components contained in $int(H)$; these are the tori drawn in gray. On the right, we compress a nonseparating meridian $m_2$ and obtain a subcompressionbody with a single torus interior boundary component. Note that compressing $\Gamma=\{m_1,m_2\}$ gives the same subcompression body as compressing $m_2$, because we fill in complementary components that are balls. See \S \ref{cbodysec} for details.

\begin{figure}
	\centering
	\includegraphics{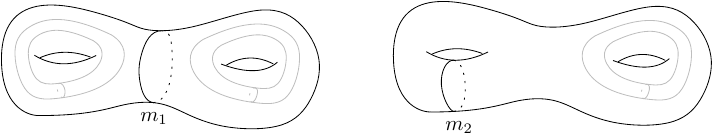}
	\caption{Compression bodies inside a genus $2$ handlebody.}
	\label{cbodyfig}
\end{figure}

Biringer-Johnson-Minsky \cite{Biringerextending} showed\footnote{Their condition was really that the measured lamination $\lambda_+ $ lies in the `limit set' of the handlebody, i.e.\ that it is a limit of meridians in $\mathcal{PML}(S)$, but that is equivalent to being strongly commensurable to a Hausdorff limit of meridians. See \S \ref{laminationssec} for more information about the limit set.} that the attracting lamination $\lambda_+$ of a pseudo-Anosov map $f : S \longrightarrow S$ is strongly commensurable to a Hausdorff limit of meridians if and only if $f$ has a nonzero power that extends to a homeomorphism of some subcompression body of $H$. Both statements are `genericity' conditions on $f$ with respect to the structure of $H$ and both were previously studied in the literature, see e.g.\ \cite{Namaziheegaard} and \cite{Lackenbyattaching}.

Here, we show that extension of powers of a homeomorphism $f : S \longrightarrow S$ to subcompression bodies can be detected by looking at extension of powers of its components in the Nielsen-Thurston decomposition. More precisely, recall that $f$ is \emph{pure} if there are disjoint essential subsurfaces $S_i \subset S$,  such that $f=id$ on $S_{id}:=S \setminus \cup_i S_i$, and where for each $i$, if we set  $f_i: = f | _{S_i}$, then either
\begin {enumerate}
\item $S_i $ is an annulus and $f_i $ is a power of a Dehn twist, or
\item $f_i $ is a pseudo-Anosov map on $S_i$.
\end {enumerate}
It follows from the Nielsen-Thurston classification  \cite{Farbprimer}, that every homeomorphism of $S$ has a power that is isotopic to a pure homeomorphism.

\begin{theorem}[Partial extension of reducible maps, see Theorem \ref{extension}]\label{extensionintro}
Let $f : S \longrightarrow S$ be a pure homeomorphism. Then $f$ has a power that extends to a nontrivial subcompressionbody of $H$ if and only if either: 
\begin{enumerate}
	\item there is a meridian in $S_{id}$, 
	\item for some $i$, the map $f_i : S_i \longrightarrow S_i$ has a power that extends to a nontrivial subcompressionbody of $H$ that is obtained by compressing a set of meridians in $S_i$, or
	\item there are (possibly equal) indices $i,j$ such that $S_i,S_j$ bound an essential interval bundle $B$ in $H$, such that some power of $f|_{S_i \cup S_j}$ extends to $B$, and  there is a compression arc $\alpha$ for $B$ whose interior lies in $S_{id}$. \end{enumerate}
\end{theorem}

In (3), note that if (for simplicity) $B$ is a trivial interval bundle, then $f|_{S_i\cup S_j}$ extends to $B$ exactly when $f_i,f_j$ become isotopic maps when $S_i,S_j$ are identified through  $B$. More generally, a power of $f|_{S_i\cup S_j}$ extends to $B$ when $f_j$ is obtained from $f_i$ by multiplying by a periodic map that commutes with $f_i$.

For the proof of Theorem \ref{extensionintro}, it is necessary to extend the theorem of Biringer-Johnson-Minsky \cite{Biringerextending} referenced above to the case of pseudo-Anosovs on essential subsurfaces of $S$. More precisely, we show that (2) above holds exactly when the attracting lamination of $f_i$ is an intrinsic limit of meridians. This is done in Theorem \ref{BJM}; the proof is basically the same as theirs, although we reorganize it into separate topological and geometric arguments in a way that makes it clearer than the original version. 

In contrast to the pseudo-Anosov case, one cannot determine when a pure homeomorphism $f: S \longrightarrow S $ has a power that extends by looking at its attracting lamination. Here, the `attracting lamination' of $f$ is the union of all its twisting curves and all the attracting laminations of its pseudo-Anosov components. Indeed, set $H=Y \times [-1,1]$, where $Y$ is a surface with boundary, and let $f : Y \longrightarrow Y$ be a pseudo-Anosov map such that $f=id$ on $\partial Y$. Then $$S = Y \times \{-1\} \cup \partial Y \times [-1,1] \cup Y \times \{+1\}$$
and if we let $F : S\longrightarrow S$ be $f\times id$ on $Y\times \{\pm 1\}$ and the identity map on the rest of $S$, while we let $G : S \longrightarrow S$ be $f\times id$ on $Y \times \{1\}$, and $f^2\times id$ on $Y \times \{-1\}$, and the identity map on the rest of $S$, then $F,G$ have the same attracting lamination, but $F$ extends to a homeomorphism of $H$, while $G$ does not. In fact, it follows from Theorem \ref{extensionintro} that no power of $G$ extends to a nontrivial subcompression body of $H$. 

\subsection{Other results of interest}

There are two other theorems in this paper that we should mention in the introduction.  

In \S\ref{largesmallsec} we study the \emph{disk set} $\mathcal D(S,M)$ of all isotopy classes of meridians in an essential subsurface $S \subset \partial M$ with $\partial S$ incompressible, where here $M$ is a compact, irreducible $3$-manifold with boundary. We show in Proposition~\ref{discsetcurve} that either $\mathcal D(S,M)$ is \emph{small}, meaning that it is either empty, has a single element, or has a single non-separating element and infinitely many separating elements that one can explicitly describe, or $\mathcal D(S,M)$ is \emph{large}, meaning that it has infinite diameter in the curve complex $\mathcal C(S)$. This result will probably not surprise any experts, but we have never seen it in the literature.

In \S\ref{windowssec} we show how essential interval bundles in a compact $3$-manifold with boundary $M$ can be seen in the limit sets in $\partial \BH^3$ associated to hyperbolic metrics on $int(M)$. This picture was originally known to Thurston \cite{Thurstonhyperbolic3}, and was studied previously under more restrictive assumptions by Walsh \cite{walsh2014bumping} and Lecuire \cite{lecuire2004structures}. We need a more general theorem in the proof of Theorem \ref{travelers}: in particular, we need a version that allows accidental parabolics. This is Theorem \ref{windowsthm}. Our proof is also more direct and more elementary than those of \cite{walsh2014bumping,lecuire2004structures}. See \S\ref{windowssec}  for more context and details.

\subsection{Outline of the paper}

Section \ref{prelims} contains all the necessary background for the rest of the paper. We discuss the curve complex, the disc set, compression bodies, interval bundles, the Jaco-Shalen and Johannson characteristic submanifold theory, compression arcs, and geodesic laminations. \S 3 and \S 4 are described in the previous subsection. \S 5 contains a discussion of homoclinic geodesics, intrinsic limits of meridians, and some of their basic properties. The main point of \S 6 is Theorem  \ref{travelers}, which is the more precise and general version of Theorem \ref{travelersintro} above. \S 7 is devoted to Theorem \ref{hlimits}, which characterizes Hausdorff limits of meridians and combines Theorems \ref{hlimitsintro} and \ref{hlimits2intro} above. \S 8 contains our extension of Biringer-Johnson-Minksy \cite{Biringerextending} to partial pseudo-Anosovs, and \S 9 contains the proof of Theorem \ref{extension}, which generalizes Theorem \ref{extensionintro} above.

  \vspace{1mm}
  
    \subsection{Acknowledgements}
  
The authors would like to thank Jeff Brock, Juan Souto and Sebastian Hensel for useful conversations.  The first author was partially supported by NSF grant DMS-1308678.

%


\section {Preliminaries}
\label{prelims}
\subsection{Subsurfaces with geodesic boundary}
\label{subsurfacessec}
Suppose $S $ is a  finite type hyperbolic surface with geodesic boundary. A \emph{connected subsurface with geodesic boundary} in $S$ is by definition either
\begin{enumerate}
	\item a simple closed geodesic $X$ on $S$, which is the degenerate case, or 
	\item an immersed surface $X\longrightarrow S$ such that the restriction to $int(X)$ and to each component of $\partial X$ is an embedding, and where each component of $\partial X$ maps to a simple closed geodesic on $S $.
\end{enumerate} 
In (2), the point is that our surface is basically an embedding, except that we allow two boundary components of $X$ to map to the same geodesic in $S$. We will usually suppress the immersion and write $X \subset S$, abusing notation. We consider  $X,Y\subset S$ to be \emph{equal} if they are either the same simple closed geodesic, or if they are both immersions as in (2) and the interiors of their domains have the same images. We say $X,Y$ are \emph{essentially disjoint} if either:
\begin{itemize}
\item $X,Y$ are disjoint simple closed geodesics, 
\item $X$ is a simple closed geodesic, $Y$ is not, and $X$ is disjoint from $int(Y)$, or vice versa with $X,Y$ exchanged, or
\item $X,Y$ have nonempty disjoint interiors.
\end{itemize} 
More generally, we define a (possibly disconnected) \emph{subsurface with geodesic boundary} in $S$ to be a finite union of essentially disjoint connected subsurfaces with geodesic boundary.

\medskip

Any connected essential subsurface $T \subset S $ that is not an annulus homotopic into a cusp of $S$ determines a unique connected subsurface with geodesic boundary $X$ such that the images of $\pi_1 T$ and $\pi_1 X$ in $\pi_1 S$ are conjugate. Here, we say that $X$ is obtained by \emph{tightening} $T$. More generally, we can tighten a  disconnected $T$ to a disconnected $X$ by tightening all its components. 

Tightening is performed as follows. If $T$ is an annulus, then we let $X$ be the unique simple closed geodesic homotopic to the core curve of $T$. Otherwise, we obtain $X$ by homotoping $T$ so that every component of $\partial T$ is either geodesic or bounds a cusp in $S \setminus T$, and then adding in any components of $S \setminus T$ that are cusp neighborhoods.
 Alternatively, let $\tilde T$ be a component of the pre-image of $T$ in the universal cover $\tilde S$, which is isometric to a convex subset of $\BH^2$, let $\Lambda_T \subset \partial \BH^3$ be the set of limit points of $\tilde T$ on $\partial_\infty \BH^2$, and let $\tilde X$ be the convex hull of $\Lambda_T$ within $\tilde S$. Then $\tilde X$ projects to an $X$ as desired.

\medskip

Conversely, suppose $X$ is a subsurface with geodesic boundary in $S$. Then there is a compact essential subsurface $T \hookrightarrow S$, unique up to isotopy and called a \emph{resolution} of $X$, that tightens to $X$. When $X$ is a simple closed geodesic, we take $T$ to be a regular neighborhood of $X$. Otherwise, construct $T$ by deleting half-open collar neighborhoods of all boundary components of $X$, and deleting open neighborhoods of all cusps of $T$. 

Note that subsurfaces with geodesic boundary $X,Y$ are essentially disjoint if and only if they admit disjoint resolutions.

\subsection {The curve complex}
\label {sec:cc}

Let $S $ be a compact orientable surface, possibly with boundary, and assume that $S$ is not an annulus.

\begin {definition}The \emph {curve complex} of $S $, written $\mathcal C (S )$, is the graph whose vertices are homotopy classes of nonperipheral, essential simple closed curves on $S $ and whose edges connect homotopy classes that intersect minimally.  \end {definition}

When $S$ is a 4-holed sphere, minimally intersecting simple closed curves intersect twice, while on a punctured torus they intersect once.  Otherwise, edges in $\mathcal C (S )$ connect homotopy classes that admit disjoint representatives.  

Masur-Minsky \cite {Masurgeometry1} have shown that the curve complex is Gromov hyperbolic, when considered with the path metric in which all edges have unit length.  Klarreich \cite {Klarreichboundary} (see also \cite {Hamenstadttrain}) showed that the Gromov boundary $\partial_\infty \mathcal C (S)$ is homeomorphic to the space of \emph {ending laminations} of $S$: i.e.\ filling, measurable geodesic laminations on $S$ with the topology of Hausdorff super-convergence.

\subsection{The disc set}

Suppose that $S\subset \partial M$ is an essential subsurface of the boundary of a compact, irreducible $3$-manifold $M $, and that $\partial S$ is incompressible in $M$.  An essential simple closed curve $\gamma $ on $M$ is called a \emph {meridian} if it bounds an embedded disc in $M $.  By the loop theorem, $\gamma $ is a meridian if and only if it is homotopically trivial in $M $.

\begin {definition}
The \emph {disc set} of $S$ in $M$, written $\mathcal D (S,M) $, is the (full) subgraph of  $\mathcal C (S) $ whose vertices are the meridians of $S$ in $M$.
\end {definition}

When convenient, we will sometimes regard $\mathcal D (S,M) $ as a subset of the space of projective measured laminations $\PML(S)$, instead of as a graph. 

\vspace{2mm}

The following is an extension of a theorem of Masur-Minsky \cite [Theorem 1.1]{Masurquasi-convexity}, which they prove in the case that $S $ is an entire component of $\partial M$.

\begin {theorem}[Masur-Minsky]\label {quasi-convexity}
The subset $\mathcal D (S,M) $ of $\mathcal C (S) $ is quasiconvex.
\end {theorem}

To prove Theorem \ref{quasi-convexity} as stated above, one follows the outline of \cite {Masurquasi-convexity}: given $a,b \in \mathcal D (S,M) $, the goal is to construct a \emph {well-nested curve replacement sequence} from $a=a_1,\ldots,a_n=b $ consisting of meridians, which must be a quasi-geodesic by their Theorem 1.2.  The sequence $(a_i)$ is created by successive surgeries along innermost discs, and the only difference here is that one needs to ensure that none of the surgeries create peripheral curves.  However, the surgeries create meridians and $S$ has incompressible boundary.

\medskip

\subsection{Compression bodies}\label{cbodysec}
We refer the reader to \S 2 of \cite{Biringerautomorphisms} for a more detailed discussion of compression bodies, and state here only a few definitions that will be used later on.

A \textit{compression body} is a compact, orientable, irreducible $3$-manifold $C$ with a $\pi_1$-surjective boundary component $\partial_+ C$, called the \emph {exterior boundary} of $C$. The complement $\partial C \smallsetminus \partial_+ C$ is called the \emph {interior boundary}, and is written $\partial_- C$. Note that the interior boundary is incompressible.  For if an essential simple closed curve on $\partial_- C$ bounds a disk $D \subset C$, then $C \smallsetminus  D$ has either one or two components, and in both cases, Van Kampen's Theorem implies that $\partial_+ C$, which is disjoint from $D$, cannot $\pi_1$-surject.

Suppose $M$ is a compact irreducible $3$-manifold with boundary, let $\Sigma$ be a component of $\partial M$ and let $S \subset \Sigma$ is an essential subsurface. A \emph{subcompression body of $(M,S)$} is a compression body $C\subset M$ with exterior boundary $\Sigma$ that can be constructed as follows. Choose a set $\Gamma$ of disjoint, pairwise nonhomotopic simple closed curves on $S$ that are all meridians in $M$. Let $C'\subset M$ be the union of $\Sigma$ with a set of disjoint disks in $M$ whose boundaries are the components of $\Gamma$, and define $C\subset M$ to be the union of a regular neighborhood of $C' \subset M$ together with any components of the complement of this neighborhood that are topological $3$-balls. Here, we say that $C\subset M$ is obtained \emph{by compressing $\Gamma$}. Note that the irreducibility of $M$ implies that no component of $\partial C$ is a $2$-sphere, and hence that $C$ is irreducible, and therefore a compression body. See \cite[\S 2]{Biringerautomorphisms}  for details about constructing compression bodies via compressions.

When the set $\Gamma$ above is empty, we obtain the \emph {trivial subcompression body} of $(M,S)$, which is just a regular neighborhood of $\Sigma \subset \partial M$. At the other extreme, we can compress a maximal $\Gamma$, which gives the `characteristic compression body' of $(M,S)$, defined via the following fact.

\begin{fact}[The characteristic compression body]\label{charcomp}
Suppose $M$  is an  irreducible  compact $3$-manifold, that $\Sigma$ is a component of $\partial M$ and that $S \subset \Sigma$ is an essential subsurface such that the multicurve $\partial S$ is incompressible in $M$. Then  there is a unique (up to isotopy) subcompression body $$C:=C(S,M) \subset M$$ of $(M,S)$, called the \emph {characteristic compression body} of $(M,S)$, such that a curve $\gamma$ in $S$ is a meridian in $C$ if and only if it is a meridian in $M$.

Moreover, $C$ can be constructed by compressing any  maximal set of disjoint, pairwise nonhomotopic meridians in $S$.\end{fact}

This is a version of a construction of Bonahon \cite{Bonahoncobordism}, except that he only defines the characteristic compression body when $S$ is an entire boundary component of $M$. In that case, the interior boundary components of $M$ are incompressible in $M$, so Bonahon's construction can be used to reduce problems about $3$-manifolds to problems about compression bodies and about  $3$-manifolds with incompressible boundary. 

The reader can also compare the fact to Lemma 2.1 in \cite{Biringerautomorphisms}, which is the special case of the fact where $M$ is a compression body and $S$ is its exterior boundary, so that  $C=M$ is obtained by compressing any maximal set of disjoint, nonhomotopic meridians in $M$. 

\begin{proof} Let $\Gamma$ be a maximal set of disjoint, pairwise nonhomotopic $M$-meridians on $S$, and define $C$ by compressing $\Gamma$. We have to check that any curve in $S$ that is an $M$-meridian is also a $C$-meridian. Suppose not, and take an $M$-meridian $m\subset S$ that is not a $C$-meridian, and that intersects $\Gamma$ minimally. Since $\Gamma$ is maximal, $m$ intersects some component $\gamma\subset \Gamma$. Then there is an arc $\alpha \subset \gamma$ with endpoints on $m$ and interior disjoint from $m$, that is homotopic rel endpoints in $M$ to the arcs $\beta',\beta'' \subset m$ with the same endpoints. (Here $\alpha$ is an `outermost' arc of intersection on a disk bounded by $\gamma$, where the intersection is with the disk bounded by $m$, see e.g.\ Lemma 2.8 in \cite{Biringerautomorphisms}.) Since $m$ is in minimal position with respect to $\Gamma$, the curves $m' = \alpha \cup \beta'$ and $m''=\alpha\cup \beta''$ are both essential, and are $M$-meridians in $S$ that intersects $\Gamma$ fewer times than $m$. So by minimality of $m$, both $m',m''$ are $C$-meridians, implying that $\alpha$ is homotopic rel endpoints to $\beta'$ and $\beta''$ in $C$. This implies $m$ is a $C$-meridian, contrary to assumption. 

For uniqueness, suppose we have two subcompression bodies $C_1,C_2$ of $(M,S)$ in which all curves in $S$ that are meridians in $M$ are also meridians in $C_1,C_2$. Since $C_1,C_2$ are subcompression bodies of $(M,S)$, the kernels of the maps
$$\pi_1 \Sigma \longrightarrow \pi_1 C_i$$
induced by inclusion are both normally generated by the set of all elements of $\pi_1 \Sigma$ that represent simple closed curves in $S$ that are merdians in $M$. Hence, the disk sets $\mathcal D(\Sigma,C_i)$ are the same for $i=1,2$. It follows that $C_1,C_2$ are isotopic in $M$, say by Corollary 2.2 of \cite{Biringerautomorphisms}.
\end{proof}

\subsection{Interval bundles}\label{sec:ibundle}

In this paper, an \emph{interval bundle} always means a fiber bundle $B \longrightarrow Y$, where $Y $ is a compact surface with boundary, and where all fibers are closed intervals $I$. Regarding the fibers as `vertical', we call the associated $\partial I$-bundle over $Y$ the \emph {horizontal boundary} of $B$, written $\partial_H B$. An interval bundle that is isomorphic to $Y \times[-1,1]$ is called \emph{trivial}, and we often call nontrivial interval bundles \emph{twisted}.

All $3$-manifolds in this paper are assumed to be orientable, but even when the total space $B$ of an interval bundle is orientable, the base surface $Y$ may not be. Indeed, let $Y$ be a compact non-orientable surface and let $\pi: \hat Y \longrightarrow Y$ be its orientation cover. Then the mapping cylinder $$B := \hat Y \times [0,1] / \sim, \ \ (x,1)\sim(x',1) \iff \pi(x)=\pi(x')$$ 
is orientable, and is a twisted interval bundle over $Y$, where the fiber over $y\in Y$ is obtained by gluing together the two intervals $\{x\}\times [0,1]$ and $\{x'\}\times [0,1]$ along $(x,1)$ and $(x',1)$, where $\pi^{-1}(y)=\{x,x'\}$. The horizontal boundary $\partial_H B$ here is $\hat Y \times \{0\}$, which is homeomorphic to the orientable surface $\hat Y$. Note that $B$ is double covered by the trivial interval bundle $\hat Y \times [-1,1]$.

\begin{fact}
Suppose that $B\longrightarrow Y$ is an interval bundle and $B$ is orientable. If $Y$ is orientable, then $B$ is a trivial interval bundle. If $Y$ is nonorientable, then $B$ is isomorphic to the mapping cylinder of the orientation cover of $Y$.
\end{fact}
\begin{proof}
If $Y$ and $B$ are orientable, so is the line bundle, so the bundle is trivial. If $Y$ is nonorientable, the horizontal boundary $\partial_H B\subset \partial B$ is an orientable surface that double covers $Y$, and from there it's easy to construct the desired isomorphism to the mapping cylinder of the projection $\partial_H B \longrightarrow Y$.
\end{proof}

An interval bundle $B\longrightarrow Y$ comes with a \emph{canonical involution} $\sigma$, which is well defined up to isotopy, and which is defined as follows. If $B\cong Y \times [0,1]$ is a trivial interval bundle, we define $$\sigma : Y \times [-1,1] \longrightarrow Y \times [-1,1], \ \  \sigma(y,t)=(y,-t).$$
And if $B$ is the twisted interval bundle $B \cong \hat Y \times [0,1] / \sim$ above, we define $$\sigma : \hat Y \times [0,1] / \sim \longrightarrow \hat Y \times [0,1] / \sim, \ \ \sigma(\hat y,t) = (\iota(\hat y),t)$$ 
where $\iota$ is the nontrivial deck transformation of the orientation cover. Note that $\sigma$ is always an orientation reversing involution of $B$, so in particular, when we give the surface $\partial_H B$ its boundary orientation, the restriction $\sigma|_{\partial_HB}$ is also orientation reversing.

\medskip

We also recall the following well-known fact.

\begin{fact}
If $S$ is a compact, orientable surface with nonempty boundary, The trivial interval bundle $S \times [-1,1]$ is homeomorphic to a handlebody.
\end{fact}

It's a nice topology exercise to visualize the homeomorphism. Regard $S$ as the union of a polygon and a collection of bands (long, skinny rectangles), each of which is glued along its short sides to two sides of the polygon. Thickening, the picture becomes a ball with $1$-handles attached. 

Note that if $S=S_{g,b}$ has genus $g$ and $b$ boundary components, then the handlebody $S\times [-1,1]$ has genus $2g+b-1$, since that is the rank of the free group $\pi_1(S \times [-1,1])\cong \pi_1 S$. 

Finally, suppose $\pi : B \longrightarrow Y$ is an interval bundle and $f : \partial_HB \longrightarrow \partial_HB$ is a homeomorphism. We say that $f$ \emph{extends to $B$} if there is a homeomorphism $F:B \longrightarrow B$ such that $F|_{\partial_HB}=f$. We leave the following to the reader. 

\begin{fact}\label{extendfact}
The following are equivalent:
	\begin{enumerate}
		\item $f$ extends to $B$,
		\item $f \circ \sigma$ is isotopic to $f$ on $\partial_HB$, 
		\item after isotoping $f$, there is a homeomorphism $\bar f : Y \longrightarrow Y$ such that $\pi \circ f = \bar f \circ \pi$,
		\item there is a homeomorphism from $B$ to either $$Y \times [-1,1] \ \ \text{or} \ \ \hat Y \times [0,1]/\sim,$$ taking horizontal boundary to horizontal boundary, such that $f=F|_{\partial_HB}$, and where either
		$$F : Y \times [-1,1] \longrightarrow Y \times [-1,1], \ \ F(y,t)=(\bar f(y),t),$$
		for some homeomorphism $\bar f :Y\longrightarrow Y$, or 
			$$F : \hat Y \times [0,1]/\sim \ \longrightarrow \hat Y \times [0,1]/\sim, \ \ F(y,t)=(\bar f(y),t),$$
			for some homeomorphism $\bar f :\hat Y\longrightarrow \hat Y$ commuting with the deck group of $\hat Y\longrightarrow Y$, and hence covering a homeomorphism of $Y$.
	\end{enumerate}
\end{fact}

\subsection{The characteristic submanifold of a pair} 
\label {sec:characteristic} Suppose that $M $ is a compact, orientable $3$-manifold and that $S \subset\partial M$  is an incompressible subsurface. In the late 1970s, Jaco-Shalen \cite{Jacoseifert} and Johannson \cite{Johannsonhomotopy} described a `characteristic' 
submanifold of $(M,S)$ that contains the images of all nondegenerate maps from interval bundles and Seifert fibered spaces. 

\begin{theorem}[see pg 138 of Jaco-Shalen \cite{Jacolectures}]\label{charsubthm}
There is a perfectly embedded Seifert pair $(X,\Sigma) \subset (M,S)$, unique up to isotopy and called the \emph {characteristic submanifold} of $(M,S)$,  such that any nondegenerate map $(B,F)\longrightarrow (M,S)$ from a Seifert pair $(B,F)$  is homotopic as a map of pairs into $(X,\Sigma)$. 
\end{theorem}

A \emph{Seifert pair} is $3$-manifold pair that is a finite disjoint union of interval bundle pairs $(B,\partial_H B)$ and $S^1$-bundle pairs. Here, an \emph{$S^1$-bundle pair} $(B,F)$ is a $3$-manifold $B$ fibered by circles, where $F \subset \partial B$ is a compact subsurface saturated by fibers. A Seifert pair $(X,\Sigma) \subset (M,S)$ is \emph{well embedded} if $X \cap \partial M =\Sigma \subset S$ and the frontier of $X$ in $M$ is a $\pi_1$-injective surface, and is \emph{perfectly embedded} if it is well embedded, no component of the frontier of $X$ in $M$ is homotopic into $S$, and no component of $X$ is homotopic into another component.

When $(B,F)$ is a connected Seifert pair, a map $f: (B,F)\longrightarrow (M,S)$ is \emph{essential} if it is not homotopic as a map of pairs into $S$. Notice that this only depends on the image of $f$ and not on $f$ itself. One says $f$ is \emph{nondegenerate} if it is essential, its $\pi_1$-image is nontrivial, its $\pi_1$-image is non-cyclic when $F=\emptyset$, and no fiber of $B$ is nullhomotopic in $(M,S)$. For disconnected $(B,F)$, one says $f$ is nondegenerate if its restriction to every component is nondegenerate.

\medskip

The following is very well known.

\begin{fact} If $int(M)$ is hyperbolizable and $(B,F)$ is an $S^1$-bundle pair that is perfectly embedded in $(M,S)$, then either 
\begin{enumerate}
\item $(B,F)$ is a `fibered solid torus', i.e.\ $B$ is an $S^1$-bundle over a disk, and $F \subset \partial B \cong T^2$ is a collection of fibered parallel annuli, or
\item $(B,F)$ is a `thickened torus', i.e.\ $B$ is an $S^1$-bundle over an annulus, so is homeomorphic to $T^2 \times [0,1]$, and each component of $F$ is either a torus or a fibered annulus.
\end{enumerate}\label{comps of char}
\end{fact}

So in particular, the components of the characteristic submanifold of $(M,S)$ are either interval bundles, solid tori, or thickened tori.

\begin{proof}
Suppose that  $(B,F)$ is a perfectly embedded $S^1$-bundle pair in $M$. Then $B \longrightarrow Y$ is an $S^1$-bundle, where $Y$ is a compact $2$-orbifold, and the cyclic subgroup $Z \subset \pi_1 B$ corresponding to a regular fiber is normal in $\pi_1 B$. In a hyperbolic $3$-manifold, any subgroup of $\pi_1$ that has a cyclic normal subgroup is elementary, say by a fixed point analysis on $\partial_\infty \BH^3$. So, $\pi_1 B $ is either cyclic or isomorphic to $\BZ^2$. It follows that $Y$ is a disc, in which case $B$ is a fibered solid torus, or $Y$ is an annulus, in which case $B$ is a thickened torus.
\end{proof}

In this paper we will mostly be interested in interval bundles. For brevity, we'll use the following terminology, which differs slightly from the terminology above used by Jaco-Shalen.

\begin{definition}
An \emph{essential interval bundle} in $(M,S)$ is an essential, well-embedded interval bundle pair $(B,\partial_H B)\hookrightarrow (M,S)$. 
\end{definition}

Note that the horizontal boundary of any essential interval bundle is an incompressible subsurface of $S$.

The definition above differs from a well embedded interval bundle pair in that we are excluding boundary-parallel interval bundles over annuli, and differs from a perfectly embedded interval bundle pair in that we are allowing components of the frontier of an interval bundle over a surface that is not an annulus to be boundary parallel. For instance, if $Y$ is a surface with boundary and $Y'\subset Y$ is obtained by deleting collar neighborhoods of the boundary components, and we set $M=Y\times [-1,1]$, which is a handlebody, then $(Y'\times [-1,1],Y'\times \{-1,1\})$ is an essential interval bundle in $(M,\partial M)$, but is not perfectly embedded. However, note that any essential interval bundle $(B,\partial_H B)\hookrightarrow (M,S)$ is perfectly embedded in $(M,\partial_H B)$.

 \subsection{Compression arcs}
\label{compressionsec}
Suppose $(B,\partial_H B) \subset (M,S)$ is an essential interval bundle. An arc $\alpha \subset S $ with endpoints on $\partial (\partial_H B)$ and interior disjoint from $\partial_H B$ is called  a \emph{compression arc} if it is homotopic in $M$ to a fiber of $B$, while keeping its endpoints on  $\partial(\partial_HB)$. See Figure~\ref{spanningfig}. 
To link this definition with more classical ones, it is easy to see that there is a compression arc for $B$ if and only if $\overline{{\mathrm Fr}(B)}$ is boundary compressible, see \cite[pp.36--37]{Jacolectures} for a definition.

Write our interval bundle as $\pi : B \longrightarrow Y$. Let $\alpha$ be a compression arc for $B$. After isotoping the bundle map $\pi$, we can assume that $\alpha$ is homotopic rel endpoints to a fiber $\pi^{-1}(y),$ where $y\in Y.$ Suppose $c$ is an oriented, two-sided, essential, simple closed loop $Y$ based at $y$, and suppose that either $c$ is nonperipheral in $Y$, or that $Y$ is an annulus or M\"obius band. Write $\pi^{-1}(c)=c_-\cup c_+$, where $c_\pm$ are disjoint simple closed oriented loops in $X$ based at $y_\pm$, and where the orientations of $c_\pm$ project to that of $c$.

\begin{claim}\label{mc}
The concatenation $m(c):=c_- \cdot \alpha \cdot c_+^{-1} \cdot \alpha^{-1}$ is homotopic to a meridian on $S$.
\end{claim}

So, a compression arc $\alpha$ allows one to make compressible curves on $S$ from essential curves on $Y$. See Figure \ref{spanningfig}.

\begin{proof}
Since $\alpha$ is homotopic rel endpoints to the fiber $\pi^{-1}(y)$, the curve $m(c)$ is homotopic in $M$ to a curve in $B$ that projects under $\pi$ to $c \cdot c^{-1}$, and hence $m(c)$ is nullhomotopic in $M$. Checking orientations, one can see that $m(c)$ is homotopic to a simple closed curve on $S$.  So, we only have to prove that $m(c)$ is homotopically essential on $S$.

Suppose that $c_-,c_+$ are freely homotopic on $\partial_H B$ as oriented curves. (This happens exactly when the curve $c \subset Y$ bounds a M\"obius band in $Y$.) Then $m(c)$ is homotopic to the commutator of two essential simple closed curves on $S$ that intersect once, and hence is essential since $S$ is not a torus.

We can now assume that that $c_\pm$ are not freely homotopic in $\partial_H B$ as oriented curves. If $m(c)$ is inessential, then $c_\pm$ \emph{are} freely homotopic on $S$, so $c_\pm$ are homotopic in $\partial_H B$ to boundary components $c_\pm' \subset \partial_H B$ that bound an annulus in $S \setminus \partial_HB$. In this case $c_\pm$ are peripheral, so we may assume that $Y$ is either an annulus or a M\"obius band. If $Y$ is a M\"obius band, we are in the situation of the previous paragraph and are done. So, $Y$ is an annulus, and $\partial_HB$ is a pair of disjoint annuli on $S$, where $c_\pm'$ lie in different components of $\partial_HB$. Since $c_\pm'$ bound an annulus in $S\setminus \partial_HB$, the interval bundle $B$ is inessential, contrary to our assumption.
\end{proof}

In fact, more is true.

\begin{fact}[Arcs that produce meridians]\label{compressionfact}
Suppose $(B,\partial_H B) \subset (M,S)$ is an essential interval bundle and let $\alpha \subset S $ be an arc with endpoints on $\partial_H B$ and interior disjoint from $\partial_H B$. Let $X \subset S$ be a regular neighborhood of $\alpha \cup \partial_H B$ within $S$. Then there is a meridian in $X$ if and only if we have either:
\begin{enumerate}
	\item the endpoints of $\alpha$ lie on the same component $c$ of $\partial (\partial_HB)$, and there is an arc $\beta\subset c$ such that $\alpha\cup\beta$ is a meridian, or
	\item $\alpha$ is a compression arc.
\end{enumerate}
\end{fact}

Note that in the second case the endpoints of $\alpha$ lie on distinct components of $\partial (\partial_HB)$, so in particular the two cases are mutually exclusive.

The reason we say $X$ `contains a meridian' instead of `is compressible' is that $X$ may not be an essential subsurface of $S$, and we want to emphasize that the essential curve in $X$ that is compressible in $M$ is actually essential in $S$. For example, let  $Y$ be a compact surface with boundary, $Y' \subset Y$ be obtained by deleting a collar neighborhood of $\partial Y$, set $B=Y' \times [-1,1]$ and $M=Y \times [-1,1]$, and let $\alpha$ be a spanning arc of $B$ in $\partial M$.

\begin{figure}
	\centering
	\includegraphics{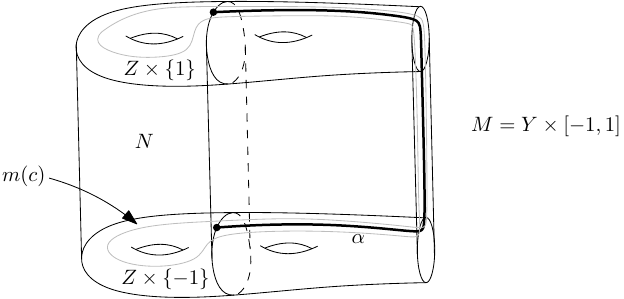}
	\caption{ $Z \subset Y$ is a compact surface, the interval bundle $B=Z \times [-1,1]$ embeds in $M=Y \times [-1,1]$, and $\alpha$ above is a compression arc. Also pictured in light gray is a meridian as described in Claim \ref{mc}.}\label{spanningfig}
\end{figure}

\begin{proof}
The `if' direction is immediate: in case (1) we are essentially given a meridian in $X$, and in case (2) we can appeal to Claim \ref{mc}.

We now work on the `only if' direction. Write our regular neighborhood of $\partial_H B  \cup \alpha$ as $X=\partial_H B \cup R$ where $R$ is a rectangle with two opposite `short' sides on the boundary of $\partial_H B$. Let $D \subset M$ be an essential disc whose boundary is contained in $X$, and where $D$ intersects the frontier $\mathrm{Fr}(B)\subset M$ in a minimal number of components. Let $a \subset D \cap \mathrm{Fr}(B)$ be an arc that is `outermost' in $D$, i.e.\ there is some arc $a'\subset \partial D$ with the same endpoints as $a$ such that $a,a'$ bound an open disk in $D$ that does not intersect $\mathrm{Fr}(B)$. 
 
We claim that $a' \subset R$. If not, then $a'\subset \partial_H B$, and bounds a disk in $B$ with the arc $a \subset \partial B$. Writing the interval bundle as $\pi: B\longrightarrow Y$, the projection $\pi(a \cup a')$ in $Y$ is then also nullhomotopic, so $\pi(a')$ is homotopic rel endpoints into $\partial Y$.  Lifting this homotopy through the covering map $\partial_H B \longrightarrow Y$ we get that $a'$ is inessential in $\partial_H B$, i.e.\ is homotopic in $\partial_H B$ rel endpoints into $\partial(\partial_H B)$. Lifting this homotopy through the covering map $\partial_H Y \longrightarrow Y$ $\pi(a)\subset \partial Y$, it follows that $\pi(a')$ is an inessential arc in $Y$. We can then decrease the number of components of $D \cap \mathrm{Fr}(B)$, contradicting that this number is minimal. 

So, $a' \subset R$. Again by minimality of the intersection, the endpoints of $a'$ lie on opposite short sides of $R$, so $\alpha$ is homotopic to $a'$ through arcs in $R$ with endpoints on $\mathrm{Fr}(B)$. Since $a'$ is homotopic rel endpoints to $a \subset \mathrm{Fr}(B)$, it follows that $\alpha$ is homotopic rel endpoints into $\mathrm{Fr}(B)$. If the two endpoints of $\alpha $ lie on the same component of $\partial(\partial_H B)$, we are in case (1), and otherwise we are in case (2).
  \end{proof}

\subsection{Laminations} We assume the reader is familiar with geodesic and measured laminations on finite type hyperbolic surfaces. See e.g.\ \cite{Bleilerautomorphisms,Kapovichhyperbolic}.
\label{laminationssec}

Suppose $\lambda $ is a connected geodesic lamination on a finite type hyperbolic surface $S $ with geodesic boundary. We say that $\lambda$ \emph{fills} an essential subsurface $T \subset S$ if $\lambda \subset T$ and $\lambda$ intersects every essential, non-peripheral simple closed curve in $T$.

\begin{fact}
For every connected $\lambda$, there is a unique subsurface with geodesic boundary (as in \S\ref{subsurfacessec}) that is filled by $\lambda$, which we denote by $S(\lambda)$. It is the minimal subsurface with geodesic boundary in $S$ that contains $\lambda$.
\end{fact}

Here, $S(\lambda)$ can be constructed by taking a component $\tilde \lambda \subset \tilde S \subset \BH^2$ of the preimage of $\lambda$, letting $C\subset \BH^2$ be the convex hull of the set of endpoints of leaves of $\tilde \lambda$ in $\partial \BH^2$, and projecting $C$ into $S$.

Suppose that $M $ is a compact, orientable irreducible $3$-manifold let $S \subset \partial M$ be an essential subsurface. The \emph{limit set} of $(S,M)$ is the closure 
$$\Lambda(S,M)=\overline {\{\text {meridians } \gamma \subset S\}} \subset \PML(S),$$ where $\PML(S)$ is the space of projective measured laminations on $S$.
The limits set was first studied by Masur \cite{Masurmeasured} in the case that $M$ is a handlebody, with $S$ its entire boundary. In this case, Kerckhoff \cite{Kerckhoffmeasure} later proved that the limit set has measure zero in $\PML(S)$, although a mistake in his argument was later found and fixed by Gadre \cite{gadre2011limit}. 

In some ways, $\Lambda(S,M)$ acts as a dynamical limit set. For instance, let $\mathrm{Map}(S)$ be the mapping class group of $S$, and let $\mathrm{Map}(S,M) \subset \mathrm{Map}(S)$ be the subgroup consisting of mapping classes represented by restrictions of homeomorphisms of $M$. Then we have:

\begin{fact}
	\begin{enumerate}
		\item If $\Lambda(S,M)$ is nonempty, it is the smallest nonempty closed subset of $\PML(S) $ that is invariant under $\Mod (S,H) $.
		\item If $\mathrm{Map}(S,M)$ contains a pseudo-Anosov map on $S$, then $\Lambda(S,M)$ is the closure of the set of the attracting and repelling laminations of pseudo-Anosov elements of $\mathrm{Map}(S,M)$.
	\end{enumerate}
\end{fact}

Note that $\mathrm{Map}(S,M)$ contains a pseudo-Anosov map on $S$ if and only if the disk set $\mathcal D(S,M)$ has infinite diameter in the curve complex $\mathcal C(S)$, where the latter condition was discussed earlier in Proposition \ref{discsetcurve}. See also \cite{Biringerextending,Lecuireextension}.

\begin{proof}
For the first part just note that Dehn twist $T_m$ around meridians $m\subset S$ are in $\mathrm{Map}(M,S)$, so if $A \subset\PML(S)$ is nonempty and invariant, $\lambda\in A$ and  $m$ is a meridian, then $m=\lim_i T^i_m(\lambda)$ is also in $A$, implying $\Lambda(M,S)\subset A$.

For the second part, take a pseudo-Anosov $f\in \mathrm{Map}(M,S)$ with attracting lamination $\lambda_+,$ say. If $m$ is a meridian in $S$, then $T_m^i \circ f \circ T^{-i}_m$ are pseudo-Anosov maps on $S$ and their attracting laminations converge to $m$, and then the argument finishes as before.
\end{proof}  

\subsection{Laminations on interval bundles}
\label{laminterval}
Suppose that $Y$ is a compact hyperbolizable surface with boundary, and that $B\longrightarrow Y$ is an interval bundle over $Y$. Endow $Y$ and the horizontal boundary $\partial_H B$ with arbitrary hyperbolic metrics such that the boundary components are all geodesic.

Suppose we have two geodesic laminations $\lambda_\pm$ on $\partial_H B$.

\begin{definition} 
 We say that $\lambda_\pm$ are \emph{essentially homotopic through $B$} if there is a lamination $\lambda$ and a homotopy
$h_t : \lambda  \longrightarrow B, \ t\in [-1,1]$
such that $h_{\pm 1}$ is a homeomorphism onto $\lambda_\pm$, and where $(h_t)$ is not homotopic into $\partial_H B$.
\end{definition}

When $B$ is a trivial interval bundle, $\lambda_\pm$ are essentially homotopic through $B$ if and only if we can write $B\cong Y \times [0,1]$ in such a way that $\lambda_\pm=\lambda \times\{\pm 1\}$ for some  geodesic lamination on $Y$. This is an easy consequence of the fact that on a surface, homotopic laminations are isotopic. In general:

\begin{fact}\label{homotopy involutoin}
Suppose that $\lambda_\pm$ are disjoint or equal geodesic laminations on $\partial_HB $. Then the following are equivalent.
	\begin{enumerate}
		\item $\lambda_\pm$ are essentially homotopic through $B$.
		\item $\lambda_\pm$ is isotopic on $\partial_H B$ to $\sigma(\lambda_\mp)$, where $\sigma$ is the canonical involution of $B$ discussed in \S\ref{sec:ibundle}.
	\end{enumerate}
Moreover, (1) and (2) imply
\begin{enumerate}
	\item[$(3)$] There is a geodesic lamination $\bar \lambda$ on $Y$ such that $\lambda_-\cup \lambda_+$ is isotopic on $\partial_HB$ to the preimage $ (\pi|_{\partial_HB})^{-1}(\bar \lambda)$.
\end{enumerate}
\end{fact}

Here, (3) does not always imply (1,2), since it could be that $\bar \lambda$ has two components, $ (\pi|_{\partial_HB})^{-1}(\bar \lambda)$ has four, and these components are incorrectly partitioned into the two laminations $\lambda_\pm$. However, that's the only problem, so for instance if $\lambda_\pm$ are minimal then (1) - (3) are equivalent.

While we have phrased things more generally in the section, we can always assume in proofs that our hyperbolic metrics have been chosen so that the covering map $\pi|_{\partial_HB}: \partial_HB \longrightarrow Y$ is locally isometric. Here, we're using the fact that given two hyperbolic metrics with geodesic boundary on a compact surface, a geodesic lamination with respect to one metric is isotopic to a unique geodesic lamination with respect to the other hyperbolic metric. In this case, we can remove the word `isotopic' from (2) and (3).

\begin {proof}
The fact is trivial when $B$ is a trivial interval bundle. When $B$ is nontrivial, lift the homotopy to the trivial interval bundle $B'\longrightarrow B$ that double covers $B$, giving homotopic laminations $\lambda_\pm' \subset \partial_HB'$. (1) $\iff$ (2) follows since the canonical involution on $B'$ covers that of $B$. For (2) $\implies$ (3), note that since $\lambda_\pm$ are disjoint or equal and differ by $\sigma$, their projections $\pi(\lambda_\pm)\subset Y$ are the same, and are a geodesic lamination $\bar\lambda$ on $Y$.
\end{proof}

%
%

\section{Large and small disk sets and compression bodies}
\label{largesmallsec}

Suppose that $S\subset \partial M$ is an essential subsurface of the boundary of a compact, irreducible $3$-manifold $M $, and that $\partial S$ is incompressible in $M$. The following is probably known to some experts, but we don't think it appears anywhere in the literature, so we give a complete proof.

\begin{prop}[Diameters of disk sets]\label {discsetcurve} With $M,S$ as above, either
	\begin {enumerate}
	\item $\mathcal D(S,M) $  has infinite diameter in $ \mathcal C(S)$,
	\item $S$ has one nonseparating meridian $\delta$, and every other meridian is a band sum of $\delta$,
	\item $S$ has a single meridian, which is separating, or
	\item $\mathcal D(S,M) =\emptyset.$ 
	\end {enumerate}
\end{prop}

In case (1), we will say that $\mathcal D(S,M) $ is \emph{large}, and in cases (2)--(4), we will say that $\mathcal D(S,M)$ is \emph{small}. Similarly, if $C(S,M)$ is the characteristic compression body defined in Fact \ref{charcomp}, then $C(S,M)$ is said to be large or small depending on whether $\mathcal D(S,M)$ is large or small. See also the discussion of small compression bodies in \S 3 of \cite{Biringerautomorphisms}.

\medskip


Here,  recall that a \emph{band sum} of a meridian $\delta$ is  the  boundary of a regular neighborhood of $\delta \cup \beta$, where  $\beta$ is a simple closed curve on $S $ that intersects  $\delta $ once. Any such band sum must be a meridian:  for instance,  as an element of $\pi_1 M$ it is a commutator with a trivial element.  Also, (3)  includes the case when $M$ is a solid torus and $S=\partial M$, in which case  there is  only one (nonseparating) meridian.  When $M$  is not a solid torus, though, every nonseparating curve has  infinitely many band sums.

Before beginning the proof, we first establish the following:
%

\begin {claim}\label{partialpaclaim}
Suppose $S$ is not a torus, $\gamma \subset S$  is a meridian on $S$ and $\delta$ is a meridian that lies in a component $T \subset S \setminus\gamma$. If $\gamma$ is not a band sum of $\delta $, there is a pseudo-Anosov $f: T\longrightarrow T$ that extends to a homeomorphism of $M $.
\end {claim}

The condition that $\gamma $ is not a band sum is necessary. For if $M$ is a handlebody with $S=\partial M$, and $\gamma$ is a separating meridian that bounds a compressible punctured torus $T \subset S$, then $T$  has only a single meridian $\delta$. This $\delta$ is non-separating and $\gamma$  is a band sum of $\delta$. Any map $T\longrightarrow T$ that extends to a homeomorphism of $M$
must then fix $\delta$,  so cannot be pseudo-Anosov.

Similarly, if $S $ is a torus and $\gamma$ is a meridian, the  complement of $\gamma$ is an annulus, which does not admit any pseudo-Anosov map.

\begin{proof}[Proof of Claim \ref{partialpaclaim}]Suppose first that $\gamma$ is not separating. Any  simple closed curve that intersects $\gamma$  once can be used  to create a band sum.  Now $S$ is not a torus, and cannot be a punctured torus either, since then its  boundary would be compressible.  So, there are a pair   $\alpha ,\beta$ of band sums of $\gamma$ that fill $S \smallsetminus \gamma$.  By  a theorem of Thurston \cite[III.3 in 13]{Fathitravaux}, the composition of twists $T_\alpha \circ T_\beta^{-1}$  is pseudo-Anosov. Each twist extends to $M $, because  twist about meridians can be extended to twists along the disks they bound.

Now suppose $\gamma$ separates $S $, and suppose that $R$ is the component of $T \setminus (\gamma \cup \delta)$ adjacent to $\gamma$ and $\delta$.   Any curve in $R$ that bounds a pair of pants with $\gamma $ and $\delta$  is also a meridian.  Such curves are constructed as the boundary of a neighborhood of the union of $\gamma ,\delta$ and any arc in $R$  joining the two.  Therefore,  there is a pair $\alpha,\beta$ of such curves that fills $R$.  As before,  $f=T_\alpha \circ T_\beta^{-1}$ is a pseudo-Anosov on $R$  that extends to $M$. 

\begin{figure}
	\centering
	\includegraphics{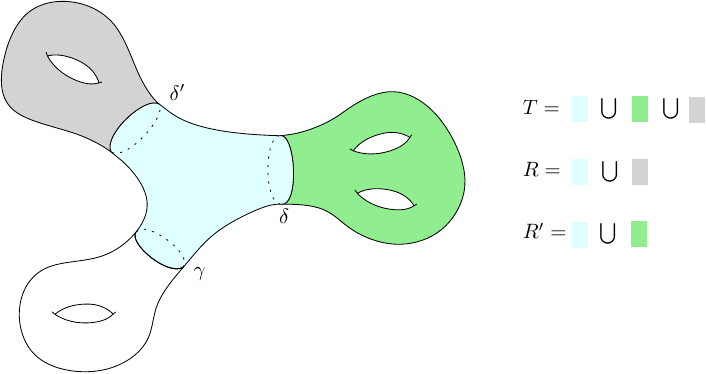}
	\caption{The surfaces $R$ and $R'$ fill $T$.}\label{curves}
\end{figure}

However, there was nothing special about $\delta$ in the above construction. So if there is some (non-peripheral) meridian $\delta' \subset T$ with $\delta\neq \delta'$, there is also a pseudo-Anosov $f'$ on the corresponding surface $R'$, such that $f'$ extends to $M$.  Since $R$ and $R'$ fill $T$,  \cite[Theorem 6.1]{Claygeometry} says  that for  large $i$  the composition $f^i (f')^i$ is a  pseudo-Anosov on $T$. See Figure \ref{curves}.

The  only case left to consider is when  $\delta $ is the only (non-peripheral) meridian in $T$. Since new meridians usually can be created by joining $\delta$ and $\gamma$  with an arc and taking a regular neighborhood, the only possibility here is that $T$  is a punctured torus, in which case this construction always just produces $\gamma $ again.  But then $\gamma$  is a band sum of $\delta$. \end{proof}

\begin{proof}[Proof of  Proposition \ref{discsetcurve}] When $S$  is a torus, distinct curves have nonzero  algebraic intersection number, so either there are no meridians or there is a single meridian.  So, we assume $S\neq T^2$  below.

We first claim that if there are two meridians in $S $, neither of which is a band sum of the other, then  $\mathcal D(S,M)$  has infinite diameter in the curve complex. To see this, suppose  $\gamma_1,\gamma_2$ are such meridians. Claim \ref{partialpaclaim} gives two pseudo-Anosov maps $f_1,f_2$,  each defined on the component of $S \setminus\gamma_i$ that contains $\gamma_j$, where $i\neq j$. Since the component of $S\setminus \gamma_1$ containing $\gamma_2$ and the component of $S\setminus \gamma_2$ containing $\gamma_1$ together fill $S$, for large $k$  the composition $f^kg^k$ is a pseudo-Anosov map on  the entire surface $S$, by \cite[Theorem 6.1]{Claygeometry}.  Any such composition extends to $M$, so maps meridians to meridians. As pseudo-Anosovs act  with unbounded orbits  on the curve complex \cite{Masurgeometry1}, this implies  that the set of meridians has infinite diameter in $\mathcal C(S)$. 

Starting now with the proof of the proposition, suppose there are \emph{no nonseparating meridians} in $S$. If $\gamma,\delta$ are distinct (separating) meridians, then an innermost disk surgery produces another separating meridian $\gamma_2$ disjoint from $\gamma_1 $, see \cite[Lemma 2.8]{Biringerautomorphisms}. By the  previous  paragraph, $\mathcal D(S,M)$ has  infinite diameter in the curve complex. So, the only other options are if  $\mathcal D(S,M) = \emptyset$, or if  the only meridian is a single separating curve.

Suppose now that there is a non-separating meridian $\gamma$ in $S$.  By Claim \ref{partialpaclaim}, unless the disc set has infinite diameter in the curve complex, any meridian  disjoint from $\gamma$ must be  a band sum of $\gamma$. So,  either we are in case (2) of the proposition, or there is some meridian $\delta$  that intersects $\gamma$. Any  innermost disk surgery of  $\delta $ along $\gamma$  must produce a band sum $\beta$ of $\gamma$. However, this $\beta$ must then bound a punctured torus $T$ containing $\gamma$, and $\delta$  is then forced to lie inside $T$,  which gives a contradiction, see Figure \ref{snake}.
\end {proof}

\begin{figure}
	\centering
	\includegraphics{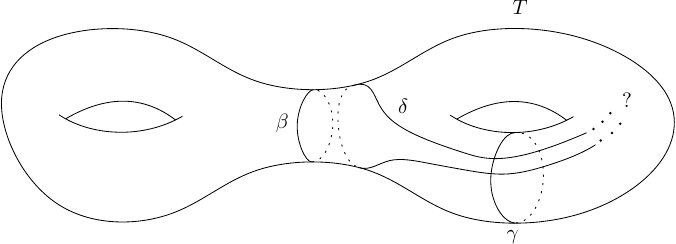}
	\caption{A surgery of a curve $\delta$ along a non-separating $\gamma$  cannot produce a curve $\beta$ that is a band sum of $\gamma$.}
	\label{snake}
\end{figure}

\section{Windows from limit sets}
\label{windowssec}
Let $N=\Gamma \backslash \BH^3$ be an orientable, geometrically finite hyperbolic $3$-manifold, let $\Lambda \subset \partial \BH^3$ be the limit set of $\Gamma$, and let $$CC(N) := \Gamma \backslash CH(\Lambda) \subset N$$ be the convex core of $N$. Equip $ \partial CC(N)$  with  its intrinsic length metric, which is hyperbolic, see for instance \cite [Prop 8.5.1]{Thurstongeometry}.

Let $S_\pm$ be (possibly degenerate) incompressible subsurfaces with geodesic boundary in $\partial CC(N)$ that are either equal or are essentially disjoint, as in \S \ref{subsurfacessec}. Let $$\tilde S_\pm \subset \partial CH(\Lambda) \subset \BH^3$$ be lifts of $S_\pm$, where if $S_-=S_+$, we require that $\tilde S_-\neq \tilde S_+$. Let $\Gamma_\pm \subset \Gamma$ be the stabilizers of $\tilde S_\pm$, let $\Lambda_\pm \subset \partial \BH^3$ be their limit sets and $\Delta=\Gamma_+\cap\Gamma_-$.

The lift $\tilde S_\pm$ is isometric to a convex subset of $\BH^2$. Let $\partial_\infty S_\pm \subset \partial \BH^2$ be the boundary of $\tilde S_\pm$. By \cite[Theorem 5.6]{mj2017cannon}, say, the inclusion $\tilde S_\pm \hookrightarrow \BH^3$ extends continuously to a $\Gamma_\pm$-equivariant quotient map $$\iota_\pm : \partial_\infty \tilde S_\pm \longrightarrow \Lambda_\pm \subset  \partial \BH^3.$$

\begin{theorem}[Windows from limit sets]\label{windowsthm} 
$\Lambda_- \cap \Lambda_+ =\Lambda_\Delta$. Next, suppose $\Delta$ is nonempty and is not a cyclic group acting parabolically on either $\tilde S_-$ or $\tilde S_+$, and let $\tilde C_\pm \subset \tilde S_\pm$ be the convex hulls of the subsets $\iota_\pm^{-1}(\Lambda_\Delta) \subset \partial_\infty \tilde S_\pm.$ Then $\tilde C_\pm$ are $\Delta$-invariant, the quotients $C_\pm :=\Delta \backslash \tilde C_\pm $ are (possibly degenerate) subsurfaces with geodesic boundary in $S_\pm$, and there is an essential homotopy from $C_-$ to $C_+$ in $CC(N)$ that is the projection of a homotopy from $\tilde C_- $ to $\tilde C_+$.
\end{theorem}

Above, $C_\pm$ are (possibly degenerate) subsurfaces with geodesic boundary in $S_\pm$, as defined in \S \ref{subsurfacessec}, but it follows from the above and Theorem \ref{charsubthm} that there are `resolutions' (see \S \ref{subsurfacessec}) $C_\pm'\subset S_\pm$ such that $C_\pm'$ bound an interval bundle in $CC(M)$. So informally, the theorem says that the intersection $\Lambda_- \cap \Lambda_+$ is exactly the limit set of the fundamental group of some essential interval bundle in $(CC(M),S_-'\cup S_+')$. The term `window' comes from Thurston \cite{Thurstonhyperbolic3} and refers to interval bundles; for example, one can `see through' a trivial interval bundle from one horizontal boundary component to the other.

The assumption that $\Delta$ is not cyclic and acting parabolically on either $\tilde S_\pm$ is just for convenience in the statement of the theorem. (Just to be clear, note that an element $\gamma\in \Delta$ can act parabolically as an isometry of $\BH^3$, but hyperbolically on the convex subsets $\tilde S_\pm \subset \BH^2$.) If $\Delta$ is cyclic and acts parabolically on $\tilde S_+$ the subset $\tilde C_+$ in the statement of the theorem will be empty. However, using the same proof one can construct a homotopy from a simple closed curve on $S_+$ bounding a cusp of $S_+$ to some simple closed curve on $S_-$.

As mentioned in the introduction, a version of Theorem \ref{windowsthm} was known to Thurston, see his discussion of the Only Windows Break Theorem in \cite{Thurstonhyperbolic3}. Precise statements for geometrically finite $N$ without accidental parabolics were worked out in Lecuire's thesis~\cite{lecuire2004structures} and by Walsh~\cite{walsh2014bumping}; note that Walsh uses the conformal boundary instead of the convex core boundary, but the two points of view are equivalent. However, for our applications in this paper, we need to allow accidental parabolics in $\tilde S_\pm$, which are not allowed in those theorems. Also, our proof is more direct and natural\footnote{In both  \cite{walsh2014bumping,lecuire2004structures}, the authors focus on proving that the boundary components of $\tilde C_\pm$ project to simple closed curves in $S_\pm$, but that isn't sufficient to say that $\tilde C_\pm$ projects to a subsurface with geodesic boundary in $S_\pm$, which is what they then claim. E.g.\ in \cite{walsh2014bumping} it is stated that under a covering map, the boundary of a subset goes to the boundary of the image, but this isn't true.}  than those in \cite{walsh2014bumping,lecuire2004structures}, despite the extra complication coming from parabolics. 

Finally, the assumption that $N$ is geometrically finite is not really essential for the theorem statement. With a bit more work dealing with degenerate ends, one can prove the theorem for all finitely generated $\Gamma$. Essentially, the point is to use Canary's covering theorem \cite{Canarycovering} to show that degenerate NP-ends in the covers $N_\pm := \Gamma_\pm \backslash \BH^3$ have neighborhoods that embed in $N$, and then to use this to prove that geodesic rays in $\BH^3$ that converge to points in $\Lambda_- \cap \Lambda_+$ cannot exit degenerate ends in $N_\pm$. After showing this, the proof of Claim~\ref{intersectlimit} extends to the general case. However, we don't have an application for that theorem in mind, so we'll spare the reader the details. 

\subsection*{Proof of Theorem \ref{windowsthm}.} 
We first focus on proving that $\Lambda_-\cap \Lambda_+ = \Lambda_\Delta$. For each $\xi \in \partial \BH^3$, let $\Gamma_\pm(\xi) \subset \Gamma_\pm$ be the stabilizer of $\xi$.

\begin{claim}\label{bothnontrivial}
	Let $\xi \in \partial \BH^3$ and suppose that $\Gamma_-(\xi) $ and $\Gamma_+(\xi)$ are both nontrivial. Then they are equal.
\end{claim}
\begin{proof}
 By the Tameness Theorem \cite{Agoltameness,Calegarishrinkwrapping}, we can identify $CC(N)$ topologically with a subset of a $3$-compact manifold with boundary $M$, where \begin{equation}
CC(N) \supset int(M), \ \ CC(N) \cap \partial M= \partial CC(N), \label{boundeq}
\end{equation} and where $\partial CC(N)$ is a collection of essential subsurfaces of $\partial \bar N$. Let $\partial_{\chi=0} M$ be the union of all torus boundary components of $M$, and let $(X,\Sigma)$ be the characteristic submanifold of the pair $(M,S_-\cup S_+ \cup \partial_{\chi=0} M),$ as in \S \ref{sec:characteristic}. 
	
Since $\Gamma_\pm$ are both contained in a discrete group $\Gamma$, both $\Gamma_\pm(\xi)$ are contained in the stabilizer $\Gamma(\xi)$, which is either infinite cyclic, or rank $2$  parabolic.

Suppose first that $\Gamma(\xi)$ is rank $2$ parabolic. The groups $\Gamma_\pm(\xi)$ are both cyclic, since $S_\pm$ are incompressible hyperbolic surfaces, so their fundamental groups do not contain $\BZ^2$ subgroups. So, we can write $\Gamma_\pm(\xi)=\langle \gamma_\pm \rangle$ for closed curves $\gamma_\pm$ on $S_\pm$.  Both $\gamma_\pm$ are homotopic into some fixed component $T \subset \partial_{\chi=0} M$, the component whose fundamental group can be conjugated to stabilize $\xi$. So, there is a component $(X_0,\Sigma_0) \subset (X,\Sigma)$ of the characteristic submanifold such that $\Sigma_0$ intersects $T$ and both $\gamma_\pm$ are homotopic on $S_\pm$ into $\Sigma_0$. Since $M \not \cong T^2 \times [0,x]$, the component $(X_0,\Sigma_0)$ is either an interval bundle over an annulus (so, a fibered solid torus), or an $S^1$-bundle pair, so by Fact \ref{comps of char}, $X_0$ is either a fibered solid torus or a thickened torus. In either case, $\Sigma_0 $ intersects each of $S_\pm$ in a fibered annulus, and these annuli are disjoint, so they are parallel on a torus boundary component of $X_0$, implying that $\gamma_\pm$ are homotopic in $M$, and hence $\Gamma_\pm(\xi)$ are conjugate in $\Gamma$. But since $\Gamma_\pm(\xi)$ have the same fixed point at infinity, the conjugating element must fix $\xi$, and therefore commute with the two groups, implying $\Gamma_-(\xi)=\Gamma_+(\xi)$.

Now assume $\Gamma(\xi)$ is cyclic. Pick a basepoint $p\in S_-$, say, and let $\gamma_- \subset S_-$  be a loop based at $p$ representing a generator of $\Gamma_-(\xi)$. Represent a generator of $\Gamma_+(\xi)$ as $\alpha \cdot \gamma_+ \cdot \alpha^{-1}$, where $\alpha $ is an arc from $p\in S_-$ to a point in $S_+$, and $\gamma_+$ is a loop in $S_+$.  Since $\Gamma_\pm$ stabilize distinct components $\tilde S_\pm$, the arc $\alpha$ is not homotopic into $S_-\cup S_+$. So, $\alpha$ is a spanning arc of an essential map from an annulus, where the boundary components of the annulus map to powers of $\gamma_\pm$. It follows that the loops $\gamma_\pm$ are homotopic on $ S_\pm$ into $\Sigma_0$ for some component $(X_0,\Sigma_0) \subset (X,\Sigma)$.
	
	If $X_0$ is an $I$-bundle with horizontal boundary $\Sigma_0$, then as $\gamma_\pm$ are not proper powers in $\pi_1 S_\pm$, they are both primitive in $\pi_1 X_0$, and hence $\gamma_\pm$ (rather than their powers) are homotopic in $X_0 \subset M$. Similarly, if $(X_0,\Sigma_0)$ is a fibered solid torus, $\Sigma_0$ is a collection of parallel annuli on $\partial X_0$, so since $\gamma_\pm$ are primitive in $\pi_1 S_\pm$, they are homotopic on $S_\pm$ to \emph{simple} closed curves in these annuli, and hence are homotopic to each other in $X_0$. 
	
	It follows that there are generators for $\Gamma_\pm(\xi)$ that are conjugate in $\Gamma$, but since these generators both fix $\xi$, they are equal.
\end{proof}

\begin{claim}\label{intersectlimit}
	For all $\xi \in \Lambda_-\cap \Lambda_+$, we have $\Gamma_-(\xi)=\Gamma_+(\xi)$. Moreover,
	$$\Lambda_\Delta = \Lambda_- \cap \Lambda_+.$$\end{claim}
\begin{proof}
	Let $N_\pm\subset \BH^3$ be the $1$-neighborhood of the convex hull of $\Lambda_\pm$, and for small $\epsilon>0$, let $T_\pm(\epsilon) \subset \BH^3$ be the set of all points that are translated less than $\epsilon$ by some parabolic element of $\Gamma_\pm$. If $\epsilon$ is at most the Margulis constant $\epsilon_0$, then $T_\pm(\epsilon)$ is a disjoint union of horoballs in $\BH^3$. 
	
	The sets $N_\pm $ and $T_\pm(\epsilon)$ are $\Gamma_\pm$ invariant. Since $\Gamma_\pm$ is a finitely generated subgroup of $
	\Gamma$, which is geometrically finite, $\Gamma_\pm$ is geometrically finite as well by \cite{Canarycovering}. So, the action of $\Gamma_\pm$ on $N_\pm \setminus T_\pm(\epsilon)$ is cocompact, see e.g.\ Theorem 3.7 in \cite{Matsuzakihyperbolic}, implying that either the function $$D_+ : \BH^3 \longrightarrow \BR_{>0}, \ \ D_+(x) = \min\{d(x,\gamma(x)) \ | \ \gamma \in \Gamma_+ \text{ loxodromic}\}$$
	is bounded above on $N_+ \setminus T_+(\epsilon)$ by some $B(\epsilon)>0$, or $\Gamma_+$ is elementary parabolic. A similar statement holds for $-$ instead of $+$.
	With $\epsilon_0$ the Margulis constant, the Margulis Lemma then implies that \emph{if $\epsilon>0$ is sufficiently small with respect to $B(\epsilon_0)$, and $\Gamma_+$ is not elementary parabolic, then} \begin{equation} \label{thin parts} T_-(\epsilon) \cap N_+\subset T_+(\epsilon_0),\end{equation}
	and similarly with $-,+$ exchanged.
	Indeed, if not then we have (say) a point $p \in \BH^3$ that is translated by less than $\epsilon$ by some parabolic $\gamma_- \in \Gamma_-$ and by at most $B$ by some loxodromic $\gamma_+ \in \Gamma_+$. If $\epsilon$ is small with respect to $B$, then both $\gamma_-$ and $[\gamma_+,\gamma_-]$ translates $p$ by at most $\epsilon_0$, so they generate an elementary discrete group by the Margulis lemma applied to $\Gamma$, implying that $\gamma_+$ fixes the fixed point of $\gamma_-$, which contradicts that they generate a discrete group.
	
	\medskip
	
	Fix $\xi \in \Lambda_+ \cap \Lambda_-$. We claim that $\Gamma_-(\xi)=\Gamma_+(\xi)$. By Claim \ref{bothnontrivial} it suffices to show that whenever $\Gamma_-(\xi)$ is nontrivial, say, so is $\Gamma_+(\xi)$.
	
	First, assume that $\Gamma_-(\xi) $ is elementary parabolic. We claim that $\Gamma_+(\xi)$ is elementary parabolic as well. Assume not, and let $\alpha$ be a geodesic ray in $\BH^3$ converging to $\xi$. Then $\alpha(t)$ lies in $T_-(\epsilon) \cap N_+$ for large $t$, and therefore in $T_+(\epsilon_0) $ for large $t$ by \eqref{thin parts}, which implies $\xi$ is a parabolic fixed point of $\Gamma_+$ as well, a contradiction.
	
	Next, suppose that $\Gamma_-(\xi)$ is elementary loxodromic. If $\xi$ is a parabolic fixed point of $\Gamma_+$, we are done, so let's assume this isn't the case. Let $\alpha$ be the axis of $\Gamma_-(\xi)$, parametrized so $\alpha(t)\to \xi$ as $t \to \infty$. Since $\xi$ is not a $\Gamma_+$ parabolic fixed point, there are $t_i\to\infty$ such that $\alpha(t_i) \not \in T_+(\epsilon)$ for all $i$. Since the action of $\Gamma_+$ on $N_+ \setminus T_+(\epsilon)$ is cocompact, if $p\in \BH^3$ is a fixed basepoint, there are elements $\gamma_i^+ \in \Gamma_+$ such that $\sup_i d(\gamma_i^+(p),\alpha(t_i)) < \infty$. Since the action of $\Gamma_-(\xi)$ on $\alpha$ is cocompact, there are then elements $\gamma_i^-\in \Gamma_-(\xi)$ with  $$\sup_i d(\gamma_i^+(p),\gamma_i^-(p))< \infty.$$
	By discreteness of $\Gamma$, after passing to a subsequence we can assume  $\gamma_i^+ =\gamma^-_i\circ g$ for some fixed $g\in \Gamma$. Hence, for all $i$ we have
	$$\gamma_{i}^+ \circ (\gamma_{1}^+)^{-1} =  \gamma_{i}^- \circ (\gamma_{1}^-)^{-1} \in \Gamma_+ \cap (\Gamma_-(\xi)) \subset \Gamma_+(\xi),$$
	so we are done.
	
	\medskip
	
	Finally, we want to show that $\Lambda_- \cap \Lambda_+ = \Lambda_\Delta$. The inclusion $\Lambda_\Delta \subset \Lambda_- \cap \Lambda_+$ is clear. So, take $\xi \in \Lambda_- \cap \Lambda_+$. We can assume that $\Gamma_\pm(\xi)=1$, since otherwise we're in the cases handled above. Let $\alpha$ be a geodesic ray in $\BH^3$ converging to $\xi$. As in the previous case, since $\xi$ is not a parabolic fixed point of $\Gamma_+$, there are $t_i\to\infty$ such that $\alpha(t_i) \not \in T_+(\epsilon_0)$ for all $i$.  Discarding finitely many $i$, we have $\alpha(t_i) \in N_+$, so it follows from \eqref{thin parts} that $\alpha(t_i) \not \in T_-(\epsilon)$. Fixing a base point $p\in \BH^3$, as  $\Gamma_-$ acts cocompactly on  $N_- \setminus T_-(\epsilon)$ and $\Gamma_+$ acts cocompactly on $N_+ \setminus T_+(\epsilon_0)$, there are elements $\gamma_i^\pm \in \Gamma_\pm$ such that $$\sup_i d(\gamma_i^\pm(p),\alpha(t_i))< \infty.$$ So passing to a subsequence, $\gamma_i^+ = \gamma_i^-\circ g$ for some fixed $g\in \Gamma$, and then $$\gamma_{i}^+ \circ (\gamma_{1}^+)^{-1} =  \gamma_{i}^- \circ (\gamma_{1}^-)^{-1} \in \Gamma_+ \cap \Gamma_-=\Delta$$
	for all $i$. But applying this sequence to $p$ and letting $i\to \infty$ gives a sequence of points in the orbit $\Delta(p)$ that converge to $\xi$, so $\xi \in \Lambda_\Delta$.\end{proof}


Now assume that $\Delta \neq 1$. We want to construct the interval bundle $W$ mentioned in the statement of the theorem. After an isotopy on $\partial CC(N)$, let's assume that $S_\pm$ is a subsurface of $\partial CC(N)$ with geodesic boundary. Consequently, we allow degenerate subsurfaces, where $S_\pm$ is a simple closed geodesic, as well as subsurfaces where only the interior is embedded and two boundary components can coincide. As $\partial CC(N)$ may have cusps, we also must allow $S_\pm$ to be noncompact with finite volume, rather than compact.

Recall that $\tilde S_\pm$ is isometric to a convex subset of $\BH^2$, and that if $\partial_\infty S_\pm \subset \partial \BH^2$ is the boundary of $\tilde S_\pm$ the inclusion $\tilde S_\pm \hookrightarrow \BH^3$ extends continuously to a $\Gamma_\pm$-equivariant quotient map $$\iota_\pm : \partial_\infty \tilde S_\pm \longrightarrow \Lambda_\pm \subset  \partial \BH^3.$$
Moreover, if $\xi,\xi' \in \partial_\infty \tilde S_+ $, say, we have $\iota_+(\xi)=\iota_+(\xi')$ if and only if there is an element $\gamma \in \Gamma_+$ that acts hyperbolically on $\tilde S_+ \cup \partial_\infty \tilde S_+$ with fixed points $\xi,\xi'\in \partial_\infty \tilde S_+$, but acts parabolically on $\BH^3$. By discreteness of the action $\Gamma_+ \actson \tilde S_+$, each $\xi \in \partial_\infty \tilde S_+$ has the same image under $\iota_+$ as \emph{at most one} other $\xi'$. Similar statements holds with $-$ instead of $+$. All this is a consequence (for instance) of Bowditch's theory of the boundary of a relatively hyperbolic group \cite{bowditch2012relatively}: since the action $\Gamma_\pm \actson \BH^3$ is geometrically finite, $\Lambda_\pm$ is a model for the Bowditch boundary of the group $\Gamma_\pm$ relatively to its maximal parabolic subgroups, so the statement above follows from Theorem 1.3 of \cite{manning2015bowditch}, say\footnote{See also Theorem 5.6 of \cite{mj2017cannon}, which says that there  is a continuous equivariant extension $\iota_\pm$ of the inclusion $\tilde S_\pm \hookrightarrow \BH^3$ as above. This theorem is stated in a much more general setting, though, and our statement is a trivial case.}.

Let $\iota_\pm^{-1}(\Lambda_\Delta) \subset \partial_\infty \tilde S_\pm$. 
Since $\Delta\neq 1$ and is not cyclic parabolic, $\iota_\pm^{-1}(\Lambda_\Delta)$ has at least two points, so it has a well-defined convex hull $\tilde C_\pm \subset S_\pm$.

\begin{claim}[Convex hulls]\label{precise} One of the following holds.
	\begin{enumerate}
		\item $\Delta$ is cyclic and acts hyperbolically on $\tilde S_+$. The convex hull $\tilde C_+$ is its geodesic axis, which is precisely invariant under $\Delta\subset \Gamma$, so that the quotient $C_+:=\Delta \backslash \tilde C_+$ embeds as a simple closed geodesic in $S_+$.
		\item $\tilde C_+$ is a subsurface of $\tilde S_+$ with geodesic boundary, the interior $int(\tilde C_+)$ is precisely invariant under $\Delta\subset \Gamma$, and the quotient $C_+ := \Delta \backslash \tilde C_+$ is a generalized subsurface of $S_+$ with compact geodesic boundary.
	\end{enumerate}
A similar statement holds with $-$ instead of $+$.
\end{claim}

\begin{proof}
	Let's work with $+$, for concreteness. If $g \in \Delta$, then $g(\Lambda_\Delta)=\Lambda_\Delta$, so $g$ leaves $\iota_+^{-1}(\Lambda_\Delta)$ invariant by equivariance of $\iota_+$. Hence $g$ leaves $\tilde C_+$ invariant.
	
	Let's suppose first that $\tilde C_+$ has nonempty interior, since that is the more interesting case. We'll address the case that $\tilde C_+$ is a biinfinite geodesic at the end of the proof. Let $g\in \Gamma_+\setminus \Delta$.  We want to show $$g(int(\tilde C_+)) \cap int(\tilde C_+)=\emptyset.$$ Assume this is not the case. By Claim \ref{intersectlimit}, the fixed points of $g$ in $\partial_\infty S_+$ lie outside $\iota_+^{-1}(\Lambda(\Delta))$. So, we cannot have $g(\tilde C_+)\subset \tilde C_+$, as then we'd have $g^n(\tilde C_+)\subset \tilde C_+$ for all $n$, contradicting that points of $\tilde S_+$ converge to the fixed points of $g$ under iteration. Considering backwards iterates, we also cannot have $\tilde C_+ \subset g(\tilde C_+)$. Therefore, $\partial \tilde C_+$ and $\partial g(\tilde C_+)$ intersect transversely. 
	
	Since $\tilde C_+$ has nonempty interior, $\Delta$ is nonelementary, and therefore the fixed points of loxodromic isometries of $\Delta$ are dense in $\Lambda_\Delta$. Loxodromic fixed points of $\Delta$ are in particular \emph{not} parabolic fixed points in $\Gamma_\pm$, so any biinfinite geodesic in $\tilde C_+$ is a limit of biinfinite geodesics in $\tilde C_+$ whose endpoints are \emph{not} fixed points of parabolic isometries of $\Gamma_\pm$. By the previous paragraph, there are then  biinfinite geodesics $\alpha_+,\beta_+$ in $\tilde C_+$ such that $g(\alpha_+)$ and $\beta_+$ intersect transversely, and where the endpoints of $\alpha_+,\beta_+$ project under $\iota_+$ to points $\xi_\alpha,\xi_\alpha',\xi_\beta,\xi_\beta'\in \Lambda_\Delta$ that are not parabolic fixed points in $\Gamma_\pm$.
	
	Let $\alpha_-$ be the geodesic in $\tilde S_-$ whose endpoints in $\partial_\infty \tilde S_-$ map to the points $\xi_\alpha,\xi_\alpha'$ under  $\iota_-$.  Define $\beta_-$ similarly. Then
	$$\alpha:= \alpha_+ \cup \{\xi_\alpha,\xi_\alpha'\} \cup \alpha_-, \ \ \beta:= \beta_+ \cup \{\xi_\beta,\xi_\beta'\} \cup \beta_-$$
	are two \emph{simple} closed curves on the closure $cl(\partial CH(\Lambda_\Gamma))\subset \BH^3 \cup \partial \BH^3$, which is homeomorphic to a sphere. For instance, the arcs $\alpha_\pm$ are disjoint and $\xi_\alpha\neq \xi_\alpha'$, since the endpoints of $\alpha_+$ are not parabolic fixed points.
	
	Now consider how the two simple closed curves $g(\alpha),\beta$ intersect. The arcs $\beta_-$ and $g(\alpha_+)$ are disjoint since $\tilde S_-\neq \tilde S_+$. The arcs $g(\alpha_-),\beta_-$ are disjoint since $g\not\in \Gamma_-$ and hence $g(\alpha_-)$ lies on a different translate of $\tilde S_-$ than $\beta_- \subset \tilde S_-$. Moreover, since $g(\alpha_+),\beta_+$ intersect transversely in $\tilde S_+$, the endpoints of $g(\alpha_+)$ and $\beta_+$ are distinct in $\partial_\infty \tilde S_+$, and since none of them are parabolic fixed points, the points $g(\xi_\alpha),g(\xi_\alpha'),\xi_\beta,\xi_\beta'$ are all distinct. But by assumption, $g(\alpha_+)$ intersects $\beta_+$ transversely in a single point! This shows that $g(\alpha)$ and $\beta$ intersect exactly once, transversely, which is a contradiction.
	
		By precise invariance of the action on the interior, the quotient $int(C_+)=\Delta / int(\tilde C_+)$ embeds in the finite volume surface $S_+$, so $C_+$ has finite volume itself.  So if $\partial C_+$ is non-compact, it must have two noncompact boundary components that are asymptotic. Lifting, we get two boundary components $\beta_1,\beta_2$ of $\tilde C_+$ that are asymptotic. Since $\tilde C_+$ is convex, it is contained in the subset of $\BH^2$ bounded by $\beta_1,\beta_2$, and hence the common endpoint of $\beta_1,\beta_2$ is an isolated point of $\Lambda_\Delta$, which is a contradiction since $\Delta$ is not elementary.
		
		\medskip

	The case when $\tilde C_+$ is a biinfinite geodesic is similar. Here, $\Delta$ must be cyclic, acting on $\tilde S_+$ with axis $\tilde C_+$, and acting either parabolically or loxodromically on $\BH^3$. In the parabolic case, $\tilde C_+$ compactifies to a simple closed curve on the sphere $cl(CH(\Lambda_\Gamma))\subset \BH^3 \cup \partial \BH^3$, so no translate $g(\tilde C_+), g\in \Gamma_+$, can intersect $\tilde C_+$ transversely, since if it did we'd get two simple closed curves on the sphere that intersect once. In the loxodromic case, we get a similar contradiction by looking at the simple closed curve $cl(\tilde C_+ \cup \tilde C_-) \subset cl(\partial \tilde M)$ and its $g$-image. So, $\tilde C_+$ is precisely invariant under $\Delta \subset \Gamma$. The quotient $C_+ := \Delta \backslash \tilde C_+$ is obviously compact, and is therefore a simple closed geodesic in $S_+$.
\end{proof}

We claim that $C_-$ and $C_+$ are homeomorphic. If $C_\pm$ are isotopic in $\partial CC(M)$ this is clear, and otherwise we argue as follows. The subgroups $\pi_1 C_\pm$ are both represented by $\Delta$, so are conjugate in $\pi_1 M$. The fact that every curve in $C_-$ is homotopic to a curve in $C_+$ (and vice versa) implies that $C_\pm$ are isotopic to subsurfaces $C_\pm' \subset \Sigma$ in the boundary $\Sigma$ of a component $(X,\Sigma)$ of the characteristic submanifold\footnote{Really, we need to be using resolutions of our subsurfaces with geodesic boundary, as discussed in \S \ref{subsurfacessec}.} of $(CC(M),S_-\cup S_+)$, see \S\ref{sec:characteristic}, and that even within $X$ every closed curve in $C_-'$ is homotopic to a closed curve in $C_+'$, and vice versa. When $X$ is a solid torus or thickened torus, $C_\pm'$ are annuli, while if $X$ is an interval bundle, $C_\pm'$ bound a vertical interval bundle in $X$, and are homeomorphic.

So, let $f : C_- \longrightarrow C_+$ be a homeomorphism, lift $f$ to a $\Delta$-equivariant homeomorphism $\tilde f : \tilde C_-\longrightarrow \tilde C_+$ and let  $$F : \tilde C_-\times [0,1] \longrightarrow CH(\Lambda)$$
where $F(x,\cdot)$ parametrizes the geodesic from $x$ to $f(x)$. Then $F$ is $\Delta$-equivariant, and projects to an essential homotopy from $C_-$ to $C_+$, as desired.

\subsection {An annulus theorem for laminations}
\label {annulussec}
Suppose $M $ is a compact, orientable, hyperbolizable $3$-manifold with nonempty  boundary and let $S = \partial_{\chi<0} M$ be the union of all non-torus boundary components of $M$. When $\alpha,\beta \subset S $ are disjoint simple closed curves that are essential and homotopic in $M $, but not homotopic in $S $, the Annulus Theorem says that there is an essential embedded annulus $A\subset M $ with $\partial A =\alpha \cup \beta $, see Scott \cite{Scottstrong}. 

More generally, equip $S$ with an arbitrary hyperbolic metric. An \emph{essential homotopy} between two geodesic laminations $\lambda_\pm$ on $S$ is a map $$H : (\lambda \times [-1,1],\lambda \times \{-1,1\}) \longrightarrow (M,S) $$ where $\lambda $ is a lamination, such that $H$ maps $\lambda \times \{\pm 1\}$ homeomorphically onto $\lambda_\pm$, and where $H$ is not homotopic rel $\lambda \times \{-1,1\}$ into $\partial M$. 

 Here is an `Annulus Theorem' for minimal laminations.

\begin{prop}[An annulus theorem for laminations]\label {intervalbundle}
Let $\lambda_-,\lambda_+ $ be two minimal geodesic laminations on $S $ that are either disjoint or equal, and assume that $S (\lambda_\pm) $ are incompressible in $ M $. If $\lambda_\pm $ are essentially homotopic in $(M,S)$, there is an essential interval bundle $(B,\partial_H B) \subset (M,S)$ such that $\lambda_\pm$ fill $\partial_H B$, and where $\lambda_\pm$ are essentially homotopic through $B$, as in \S \ref{laminterval}.
\end{prop}

 Here, $S (\lambda_\pm)$ are the subsurfaces with geodesic boundary filled by $\lambda_\pm$, as in \S \ref{laminationssec}. The assumption that they are incompressible  generalizes the  assumption that $\alpha,\beta$ are homotopically essential in $M$ in the Annulus Theorem.

\begin{proof}
Identify $M\setminus \partial_{\chi=0} M$ with the convex core of a geometrically finite hyperbolic $3$-manifold. Set $S_\pm := S(\lambda_\pm)$. Lift the essential homotopy from $\lambda_-$ to $\lambda_+$ to a homotopy from lifts $\tilde \lambda_-\subset \tilde S_-$ to $\tilde \lambda_+\subset \tilde S_+$ in $\BH^3$. Under the homotopy, which has bounded tracks, corresponding leaves of $\tilde \lambda_\pm$ have the same endpoints in $\partial \BH^3$. The endpoints of $\tilde \lambda_\pm $ are dense in $\partial_\infty \tilde S_\pm$, so this means that the subsurfaces $C_\pm \subset S_\pm$ constructed in Theorem \ref{windowsthm} are just $C_\pm=S_\pm$. Passing to disjoint or equal resolutions $S_\pm'$ of $S_\pm$ and applying Theorem \ref{charsubthm} gives an interval bundle $B$ where $\lambda_\pm$ fill $\partial_H B=S_-'\cup S_+'$. 

We claim that $\lambda_\pm$ are essentially homotopic through $B$. By Fact~\ref{homotopy involutoin}, it suffices to show that if $\sigma$ is the canonical involution of $B$, as described in  \S \ref{intervalbundle}, then $\sigma(\lambda_\pm) $ is isotopic to $\lambda_\mp$ on $S_\mp'$. Using the notation of Theorem \ref{windowsthm}, $\sigma$ lifts to a $\Delta$-equivariant involution $\tilde \sigma$ of $\tilde B$ that exchanges $\tilde S_-'$ and $\tilde S_+'$, where here $\Delta = \Gamma_-\cap \Gamma_+$. By equivariance, $\tilde \sigma$ extends continuously to the identity on $\Lambda_\Delta$, so $\tilde \sigma(\tilde \lambda_-)$ is a lamination on $\tilde S_+$ with all the same endpoints at infinity as $\tilde \lambda_+$, and hence equals $\tilde\lambda_+$. The claim follows.
\end{proof}

\section{Laminations on the boundary}  \label{lamsec}
Suppose that $M $ is a compact, orientable $3$-manifold with hyperbolizable interior and nonempty boundary $\partial M$.  Equip $M$ with an arbitrary Riemannian metric  and lift it to  a Riemannian metric on the universal cover $\tilde M$.  As in the introduction, a biinfinite path or ray $h$ on $\partial \tilde M$ is called \emph{homoclinic} if there are points $s^i,t^i$  with $|s^i-t^i|\to \infty$ such that $$\sup_i d_{\tilde M} (h(s^i),h(t^i)) < \infty.$$ Two rays $h_+,h_- $ on $\partial \tilde M$ are called \emph{mutually homoclinic} if there are  parameters $s_\pm^i \to \infty$  such that $$\sup_i d_{\tilde M} (h_+(s_+^i),h_-(s_-^i)) < \infty.$$
 Here, a \emph{ray}  is a continuous map  from an interval $ [a,\infty) $, and a \emph{biinfinite path} is  a continuous map from $\BR$.  We will also call rays and paths on $\partial M$ (mutually) homoclinic if they have lifts that are (mutually) homoclinic paths on $\partial \tilde M$. We refer the reader to \S \ref{alternate} for some comments on alternate definitions of homoclinic that exist in the literature.

Note that if we divide a biinfinite homoclinic path  into two rays, then  either one of the two rays is  itself homoclinic, or  the two rays are mutually homoclinic.  Also,  these definitions are metric independent: since $M$ is compact, any two Riemannian metrics on $M$ lift to quasi-isometric metrics on $\tilde M$, and a path is homoclinic or mutually homoclinic with respect to one metric if and only if it is with respect to the other metric.

Here are some examples.

\begin{enumerate}
	\item  Suppose that $D$ is a  properly embedded disc in $M$, and $h : \BR \longrightarrow \partial M$ is a path that covers $\partial D\subset\partial M $. Then $h$ is homoclinic: indeed, $D$ lifts homeomorphically to $\tilde M$, so $h$ lifts to a path in $\tilde M$ with compact image.
\item Suppose that $\phi : (S^1 \times [0,1], S^1 \times \{0,1\}) \longrightarrow (M,\partial M)$  is an essential embedded annulus. Then rays covering the two boundary components of the annulus are mutually homoclinic: indeed, $\phi$ lifts to  $$\tilde \phi : \BR \times [0,1] \longrightarrow \tilde M ,$$ and we have $\sup_{t\in \BR} d(\tilde \phi(t,0), \tilde \phi(t,1)) < \infty$, so restricting to $t\in [0,\infty)$ we get two mutually homoclinic rays in $\tilde M$.\end{enumerate}

It will be convenient  below to work with a  particular choice of metric on $M$.  

\begin{example}[An explicit metric on $M$]\label{explicitmetric}
Let $\partial_{\chi<0} M$  be the union of all  components of $\partial M$  that have negative Euler characteristic, i.e.\ are not tori. Thurston's Haken hyperbolization theorem, see \cite {Kapovichhyperbolic}, implies that  there is a hyperbolic $3$-manifold $N = \BH^3 / \Gamma$ homeomorphic to the interior of $M$, where every component of $\partial_{\chi<0} M$  corresponds to a convex cocompact end of $N$. A torus $T \subset \partial M$, on the other hand, determines a cusp of $N$. So, in other words, $N$ is `minimally parabolic': the only parabolics come from torus boundary components of $M$. For each $T$, pick an open neighborhood $N_T \subset N$ of the associated cusp   that is the quotient of a horoball in $\BH^3$ by a $\BZ^2$-action. Then 
\begin{equation}
		M \cong CC(N) \ \setminus \bigcup_{\text {tori } T\subset \partial M} N_T, \label{ident}
\end{equation}
and we will  identify $M$  with the right-hand side everywhere below. Then
\begin{itemize}
	\item $\tilde M \subset \BH^3$ is obtained from the convex hull $CH(\Gamma) \subset \BH^3 $ of the limit set of $\Gamma$ by deleting an equivariant collection of horoballs, and
\item the path metric induced on $\partial_{\chi<0} M $ is   hyperbolic \cite [Proposition 8.5.1]{Thurstongeometry},  and the  path metric induced on every torus $T \subset\partial M$  is Euclidean. 
\end{itemize} 
\end{example}

 We now specialize to the case of paths that are \emph{geodesics} on  $\partial M $.  Recall  from Example (1) above that one can make homoclinic paths by running around the boundaries of disks in $\partial M$. The following shows that discs are essential in such constructions.

\begin {fact}\label{incompressible-qg}
 Suppose that $S \subset \partial M$ is an essential subsurface. Then the inclusion of any lift $\tilde S \subset \partial \tilde M$ is a quasi-isometric embedding into $\tilde M$. Moreover if $S$ is incompressible then any pair of mutually homoclinic infinite rays on $S$ are asymptotic and no biinfinite geodesic $\gamma$ in $S$ is homoclinic.
\end {fact}
\begin {proof}
Think of $M$ as embedded in a complete hyperbolic $3$-manifold $N$ as in \eqref{ident}, write $N=\Gamma \backslash \BH^3$, and let $\tilde M \subset \BH^3$ be the preimage of $M$, so that $\tilde M$ is obtained from the convex hull $CH(\Gamma)$ be deleting an equivariant collection of horoballs. Fix a subgroup $\Delta < \Gamma$ that represents the conjugacy class associated to the image of the fundamental group of $S \subset M$. To show that $$\tilde S \hookrightarrow \partial \tilde M$$ is a quasi-isometric embedding, it suffices to show that $\Delta$ is undistorted in $\Gamma$. But since $M$ is geometrically finite and $\Delta$ is finitely generated, it follows from a result of Thurston (see Proposition 7.1 in \cite{morgan1984thurston}) that the group $\Delta$ is geometrically finite, and geometrically finite subgroups of (say, geometrically finite) hyperbolic $3$-manifold groups are undistorted, c.f.\ Corollary 1.6 in \cite{hruska2010relative}.

For the `moreover' statement, assume $S$ is incompressible, so that $\tilde S$ is simply connected, and consider a pair of infinite rays $$h^{\pm}:\BR^+\to \tilde S$$ that are geodesic for the induced hyperbolic metric and $t^{\pm}_n\to  +\infty$ such that $d_{\tilde M}(h^+(t_n^+),h^-(t_n^-))$ is bounded. Since $\tilde S \subset \partial \tilde M$ is a quasi-isometric embedding, $d_{\tilde S}(h^+(t_n^+),h^-(t_n^-))$ is also bounded. Since $\tilde S$ is simply connected and hyperbolic, this is possible only if $h^+$ and $h^-$ are asymptotic on $S$. Taking $h^+=h^-$ we get that a geodesic ray on $\tilde S$ can not be homoclinic. Taking $h^+\neq h^-$, we get that any pair of mutually homoclinic infinite rays on $S$ are asymptotic. In particular two disjoint geodesic rays in a homoclinic geodesic should be asymptotic. This is impossible for a geodesic in a simply connected hyperbolic surface.
\end {proof}

Example (2) above shows how embedded annuli in $M$ can be used to create mutually homoclinic rays. In analogy to Fact \ref{incompressible-qg}, one can show that annuli are essential in such a construction. For instance, suppose $M$ is acylindrical. Then work of Thurston, see \cite{Kapovichhyperbolic} and more generally \cite{Lecuireplissage},  says that we can  choose the hyperbolic manifold $N$ so that $\partial CC(N) \cong \partial_{\chi<0} M$  is totally geodesic. Hence,  the preimage of $\partial_{\chi<0} M$ in $\tilde M \subset \BH^3$ is a collection of hyperbolic planes.  Any geodesic ray on $\partial_{\chi<0} M$ then lifts to a geodesic in $\BH^3$, and two geodesic rays on $\partial_{\chi<0} M $ are mutually homoclinic if and only if their geodesic lifts are asymptotic in $\BH^3$, which implies that they were asymptotic on $\partial_{\chi<0} M $.

\medskip


\subsection{Alternate definitions of homoclinic}
\label{alternate} \note{Added this subsection}
Above, we defined a path $$h : I \longrightarrow \partial \tilde M$$ to be homoclinic if there is are $s^i,t^i \in I$  with $|s^i-t^i|\to \infty$ such that $$\sup_i d_{\tilde M} (h(s^i),h(t^i)) < \infty.$$
Some other papers use slight variants of this definition. For example, the definition of \emph{(faiblement) homoclinique} in Otal's thesis \cite{Otalcourants} is almost the same as what is written above, except that distances are computed \emph{in the intrinsic metric on $\partial \tilde M$} instead of in $\tilde M$. This is equivalent to our definition, though: the nonobvious direction follows from Fact \ref{incompressible-qg}, which says that boundary components of $\tilde M$ quasi-isometrically embed in $\tilde M$. And in the definition of \emph{homoclinique} in Lecuire's earlier work \cite{Lecuireplissage}, distances are computed not in $\tilde M$, but within $\BH^3$, with respect to a given identification of $\tilde M$ with the convex core of some minimally parabolic hyperbolic $3$-manifold, as discussed in Example~\ref{explicitmetric}. When $M$ has tori in its boundary, the inclusion $\tilde M \hookrightarrow \BH^3$ is not a quasi-isometric embedding, but the following lemma says that $d_{\BH^3}$ is bounded if and only if $d_{\tilde M}$ is bounded, so Lecuire's earlier definition is equivalent to ours.

\begin{lemma}
Whenever $x,y\in \tilde M$, we have $$d_{\BH^3}(x,y) \leq d_{\tilde M}(x,y)\leq e^{d_{\BH^3}(x,y)/2} d_{\BH^3}(x,y).$$
\end{lemma}
\begin{proof}
Set $N := \Gamma \backslash \BH^3$, so that $\tilde M$ is obtained from the convex hull $CH := CH(\Lambda(\Gamma))$ of the limit set of $\Gamma$ by deleting horoball neighborhoods around all rank two cusps. Take a $\BH^3$-geodesic $\gamma$ from $x$ to $y$. Then $\gamma$ lies inside $CH$, and it can only penetrate the deleted horoball neighborhoods to a depth of $d(x,y)/2$. Now, whenever $B\supset B'$ are horoballs in $\BH^3$ such that $d_{\BH^3}(\partial B, B') \leq d(x,y)/2$, the closest point projection $$\pi : B \setminus B' \longrightarrow \partial B$$ is well defined and $e^{d(x,y)/2}$-lipschitz. (Indeed, it suffices to take $B$ as the height $1$ horoball in the upper half space model and $B'$ as the height $e^{d(x,y)/2}$ horoball, and then the claim is obvious.) So, the parts of $\gamma$ above that penetrate the deleted horoballs can be projected back into $\partial \tilde M$, and if we do this the resulting path has length at most $e^{d(x,y)/2}d(x,y)$.
\end{proof}

We should mention the version of homoclinic defined in Casson's original unpublished notes. There, $M$ is a handlebody, and if we regard $\partial \tilde M \hookrightarrow \BH^3$ as above, then a simple geodesic $ h : I \longrightarrow \partial \tilde M$ is called \emph{homoclinic} if when we subdivide $ h$ into two rays $h_\pm$, these rays limit onto subsets $A_\pm \subset \BH^3 \cap \partial \BH^3$ such that $A_+\cap A_-\neq \emptyset$. This definition is stronger than all the ones mentioned above: if $A_+\cap  A_-$ contains a point on $\partial \tilde M$, rather than at infinity, then the definition of homoclinic above is obviously satisfied. Otherwise, $h_\pm$ have to have a common accumulation point in $\partial \BH^3$, which corresponds to an end $\xi$ of $\tilde M$, and one can use the treelike structure of the universal cover $\tilde M$ of the handlebody $M$ to say that $h_\pm$ have to both intersect a sequence of meridians $(m_i)$ on $\tilde M$ that cut off smaller and smaller neighborhoods of $\xi$. The times $t^i_pm$ when $h_\pm$ intersects $m_i$ then work in the definition of homoclinic above. In fact, Casson's definition is strictly stronger. For instance, if $h_\pm$ both spiral around disjoint simple closed curves $\gamma_\pm \subset \partial \tilde M$, then $h$ is homoclinic by our definition but not by Casson's. However, Statement \ref{ccr} still fails using Casson's original definition, due to the examples in Figure \ref{hlim}.

 \subsection{Waves, Tight position, and Intrinsic limits} \label{intrinsiclimitssec} 

  As in the previous section, let $M$  be a compact, orientable  hyperbolizable $3$-manifold with  nonempty boundary  $\partial M$,  which we  think of as the convex core  of a hyperbolic $3$-manifold with horoball neighborhoods of its rank $2$ cusps deleted.  

\begin {definition}[Waves  and tight position]\label{wavedef}
 Suppose that $m $ is a meridian multicurve on $\partial M $, and let $\gamma \subset \partial M $ be a simple geodesic ray or a simple biinfinite geodesic.  An \emph {$m$-wave}  is a segment $\beta \subset \gamma$ that has endpoints on $m$, and is homotopic rel endpoints in $M$ to an arc of $m$.  If $\gamma $ has no $m$-waves, and  every infinite length segment of $\gamma$ intersects $m$, then we say that $\gamma $ is in \emph {tight position} with respect to $m$.
\end {definition}
 
Waves and tight position were discussed previously in \cite{Kleineidamalgebraic,Lecuireplissage}, for instance. Note that in our definition, an $m$-wave $\beta$  can intersect $m$ in its interior.  More generally, an $m$-wave of a lamination is an $m$-wave of one of its leaves, and  a lamination is in \emph {tight position}  with respect to $m$ if all of its leaves are.  Note that from this perspective, if  a geodesic $\gamma$ is in tight position with respect to some multicurve $m$ (regarded as a lamination), then it is in tight position with respect to some component of $m$.


 As an example, a meridian $\gamma $ can never be in tight position with respect to  another meridian $m$: taking discs with boundaries $\gamma$ and $m$ that are transverse and intersect minimally, any arc of intersection of these disks terminates in a pair of intersection points of $\gamma$ and $m$  that bound a $m$-wave of  $\gamma$.  

More generally,  we have the following fact.

\begin{fact}[Tight position $\implies$ $\BH^3$ quasi-geodesic] \label {tightquasi} 
Let  $\gamma $  be a simple geodesic ray or biinfinite geodesic on $\partial M$. If $\gamma$  is in tight position with respect to some meridian $m$ then  any lift $\tilde \gamma \subset \partial \tilde M $ of $\gamma$ is a quasigeodesic in $\BH^3$. In particular, $\gamma$ is an $\tilde M$-quasigeodesic, and is not homoclinic.\end{fact}


\begin {proof}
Intersecting with $m$ breaks $\gamma$ into a union of  finite arcs. By simplicity of $\gamma$,  these arcs fall into only finitely many homotopy classes rel $m$, and  there is a universal upper bound $L=L(\gamma,m)$ on their lengths.  Let $D$ be  a  disc with boundary $m$ and let $\tilde D$  be the entire preimage of $D$ in $\tilde M$. Tightness means that the path $\tilde \gamma $  intersects infinitely many components of $\tilde D$, and intersects no single component more than once. 

In the notation of Example \ref{explicitmetric}, we have that $\tilde M \subset \BH^3 $ is obtained from $CH(\Gamma)$ by deleting an equivariant collection of horoballs. Each  component of $\tilde D$  separates $CH(\Gamma)$, so if $\gamma'$  is a segment of $\gamma$,  any geodesic in $\BH^3$ joining the endpoints of  $\gamma'$ must  intersect each of the discs that $\gamma'$ intersects. 
Hence, if $\epsilon>0$  is the minimum distance between any two components of $\tilde D$,  then $\tilde \gamma $ is a $(L/\epsilon,L)$-quasigeodesic. 
\end {proof}

\begin{figure*}
	\centering
\includegraphics{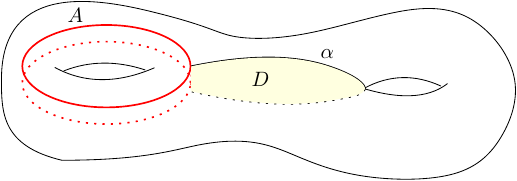}
\caption{Surgering $A$ along $\alpha $ gives a meridian $c$.}
\label{essmer}
\end{figure*}

We now describe how to create  systems of meridians with respect to which a given lamination is in tight position.

\begin {definition}[Surgery] Suppose that $\lambda $ is a  geodesic lamination on $\partial M$ and $m=\sqcup_{i=1}^n m_i$  is a  geodesic meridian multi-curve on $\partial M$, and $\beta$  is an $m$-wave in $\lambda $ whose interior is disjoint from $m$. Then the pair of points $\partial \beta$ separates some component $m_i$ of $m$ into two arcs $m_i^1,m_i^2$, both of which are homotopic to $\beta$ rel endpoints in $M$.  We perform a  \emph {$\lambda $-surgery} on $m$  by replacing $m_i^1$ (say) with $\beta$, thus constructing a new multicurve $m':= (\beta \cup m_i^2) \sqcup \bigsqcup_{j\neq i} m_j$. \end{definition}

This notion of surgery appears in many other references, e.g.\ \cite{Biringerautomorphisms,casson1985algorithmic,Kleineidamalgebraic,Lecuireplissage}. We summarize its elementary properties here:

\begin{fact}  \label{easy facts} Suppose that $\lambda $  is a geodesic lamination.
	\begin{enumerate}
		\item If $m$ is a meridian and $\lambda $ has an $m$-wave, it also has an $m$-wave whose interior is disjoint from $m$,  so a $\lambda $-surgery can be performed.
\item Any curve $m'$ obtained by $\lambda$-surgery on a meridian $m$ as above is a meridian.
\item If $m$ is a \emph{cut  system} for $M$, i.e.\ a multi-curve of meridians  bounding discs that cut $M$  into balls and $3$-manifolds with incompressible boundary, then some $\lambda $-surgery on a component of $m$ is another cut system. 
	\end{enumerate}
\end{fact}
\begin {proof}
	For (1),  suppose that $\lambda$ has an $m$-wave $\beta$. Let $\tilde m$ be the entire preimage of $m$  in the cover $\partial \tilde M$, and lift $\beta$  to an arc $\tilde \beta$  starting and ending on some fixed component $m_0 \subset \tilde m$.  Since each component of $\tilde m$ separates $ \partial \tilde M$,  there is some ``outermost'' subarc $\tilde \beta '$ that has endpoints on the same component of $\tilde m$, and that has interior disjoint from $\tilde m$. This $\tilde \beta '$  projects to an $m$-wave of $\lambda$  whose interior is disjoint from $m$.

For (2), note that if $m':=\beta \cup m_2$  is obtained by $\lambda$-surgery on $m$, as above, then $m,m'$ are homotopic in $M$, and hence $m'$ is nullhomotopic. Also, if $m'$ is inessential in $\partial M$, then $\beta$ is homotopic on $\partial M$ to $m_2$, implying that $\lambda$ and $m$ were not in minimal position, a contradiction since they are both geodesic. Hence $m'$ is a meridian.

For (3), consider an $m$-wave in $\lambda$ whose interior is disjoint from $m$ and say that $\partial\beta\subset m_1$. Then $\partial\beta$ separates $m_1$ into two arcs $m_1^1,m_1^2$. It is not difficult to see that either $(\beta \cup m_1^1) \sqcup \bigsqcup_{j\neq 1} m_j$ or $(\beta \cup m_1^2) \sqcup \bigsqcup_{j\neq 1} m_j$ is a cut system.
\end {proof}

The  following lemma is a  modification of  a result of Kleineidam--Souto \cite[Lemmas 7 and 8]{Kleineidamalgebraic}  that is essential  for everything below.

\begin {lemma}[No waves, or a sequence of meridians]\label {tightcuts}
 Suppose $\lambda$  is a geodesic lamination on  $S=\partial M$  and $m$  is a meridian multi-curve. Then either 
\begin {enumerate}
\item there exists a finite sequence of $\lambda $-surgeries on $m$ that terminates in some meridian multi-curve $m'$ where $\lambda $ has no $m'$-waves,
\item  $S(\lambda)$ contains a sequence of meridians $(\gamma_i)$ such that $i(\lambda,\gamma_i) \to 0$,  with respect to every transverse measure on $\lambda $. \end{enumerate}
\end {lemma}

Here, $(2)$ makes sense even when $\lambda$  admits no transverse measure of full support. Note that if $\lambda $ is a minimal lamination and $\partial S(\lambda) $ is incompressible, then $(2)$  implies that  $\lambda $ is an  intrinsic limit of meridians.

\begin {proof}[Proof of Lemma \ref{tightcuts}]
The two cases depend on whether $\lambda$ contains infinitely many  homotopy  classes of $m$-waves, or not. Here, our homotopies are through  arcs on $S$,  keeping their endpoints on $m$.

If there are only finitely many classes of $m$-waves in $\lambda$, then a finite sequence of  $\lambda $-surgeries converts $m$ into a  multi-curve $m'$ such that  $\lambda$ has no $m'$-waves,  as each surgery decreases the number of waves by at least one. 
If there are infinitely many homotopy classes  of $m$-waves in $\lambda$, then we can choose a sequence of parameterized $m$-waves $\alpha_i : [0,1] \longrightarrow \BR$ such that
\begin {enumerate}
\item the two sequences  of endpoints $(\alpha_i(0))$ and $(\alpha_i(1))$  both converge, and if either sequence converges into a simple closed curve $\gamma \subset \lambda$, then it approaches $\gamma$  from only one side,
\item no $\alpha_i $ and $\alpha_j $ are homotopic keeping their endpoints on $m$, for $i\neq j$.
\end {enumerate}
 To construct the desired sequence of meridians, let $\beta^0_i$ be the shortest geodesic on $S$ from $\alpha_i(0)$ to $\alpha_{i+1}(0)$,  and define $\beta^1_i$ similarly.  For large $i $, the union $\beta^0_i \cup \alpha_i \cup \beta^1_i \cup \alpha_{i+1}$  is an essential closed curve in $S(\lambda) $  that is nullhomotopic in $M$.  It may not be simple, since $\beta^0_i$ and $\beta^1_i$  may overlap, but it has at most one self intersection. So by the Loop Theorem \cite{Jacolectures}, one of the three simple closed curves obtained by surgery on it is a meridian $\gamma_i$. 

Now,  the fact that the endpoints can approach a simple closed curve in $\lambda$ only from one side  implies that for large $i $, the curves $\gamma_i$  do not intersect any simple closed curve contained in $\lambda$.  Since $\gamma_i$  only intersects $\lambda $ along the arcs $\beta^0_i$ and $\beta^1_i$, whose  hyperbolic lengths converge to zero,  it follows that $i(\gamma_i,\lambda) \to 0$  for any transverse measure on $\lambda$.
\end{proof}

 Here is an important application of Lemma \ref {tightcuts}. 

\begin{lemma}[Quasigeodesic or a sequence of meridians]\label {limitquasi}
 Suppose $\lambda \subset \partial_{\chi<0} M$ is a minimal geodesic lamination and that $\partial S(\lambda)$ is incompressible in $M$. Let  $h \subset S(\lambda)$  be a  simple geodesic ray or biinfinite geodesic that is disjoint from $\lambda$ or  contained in $\lambda$. Then either 
\begin{enumerate}
\item any lift $\tilde h \subset\partial \tilde M$ of $h$ is a quasi-geodesic in $\tilde M$.
\item $S(\lambda)$ contains a sequence of meridians $(\gamma_i)$ such that $i(\lambda,\gamma_i) \to 0$,  with respect to every transverse measure on $\lambda $. 
\end{enumerate}
In particular, if $h$ is homoclinic, then $\lambda$  satisfies (2).
\end{lemma}

\begin {proof}
Assume that (2) does not hold. Given a cut system $m$ for $M$, Lemma \ref{tightcuts} and Fact \ref{easy facts} (3) say that we can perform $\lambda $-surgeries until we obtain a new cut system $m$  such that $\lambda \cup h$  has no $m$-waves. If $m$ intersects $\lambda $, then $\lambda \cup h$  is in tight position with respect to $m$,   so (1) follows from Fact \ref{tightquasi}.   Therefore, we can assume $m$ does not intersect $\lambda $.  Up to isotopy,  we can also assume that $S(\lambda)$  does not intersect $m$. Since $\partial S(\lambda)$ is assumed to be a collection of incompressible curves, it follows that $S(\lambda)$ is itself incompressible, so (1) follows from Fact \ref{incompressible-qg}. \end {proof}

 We now come to the central definition of the section.

\begin {definition}
	A minimal geodesic lamination $\lambda \subset \partial_{\chi < 0} M$  is an \emph{intrinsic limit of meridians} if there is a transverse measure\footnote{It is currently unknown whether the particular  transverse measure  matters:  we  might suspect that a  measured lamination is a projective limit of meridians if and only if the same is true for any other measured lamination with the same support, but  there could also very well be a counterexample.} on $\lambda$ and a sequence of meridians $(\gamma_i)$ contained in $S(\lambda)$ such that $\gamma_i \to \lambda$ in $\PML(S (\lambda))$. 
\end {definition}

Using Lemma \ref{limitquasi},  we can prove the following proposition, which gives several equivalent characterizations  of intrinsic limits. 

\begin {prop}[Intrinsic limits]\label {intrinsiclimits}
 Suppose $\lambda\subset S=\partial M$ is a minimal geodesic lamination and $\partial S(\lambda)$  is incompressible.  The following are equivalent:
\begin {enumerate}
\item  $\lambda $ is an  {intrinsic limit of  meridians},
\item  given (some/any)  transverse measure on $\lambda $, there is a sequence of meridians $(\gamma_i)$  in $S(\lambda)$ such that $i(\gamma_i,\lambda)\to 0,$ 
\item there is a biinfinite homoclinic geodesic in $S(\lambda) $  that is either a leaf of $\lambda$, or is disjoint from $\lambda$,
\item  given any  transverse measure on $\lambda $, there is a sequence of  essential (possibly non-simple) closed curves $(\gamma_i)$  in $S(\lambda)$  such that each $\gamma_i $ is nullhomotopic  in $M$, and $i(\gamma_i,\lambda)\to 0.$
\end {enumerate}
\end {prop}

Note that when we say $\partial S(\lambda)$ is incompressible, we mean that no closed curve that is a {boundary component} of $S(\lambda)$ is nullhomotopic in $M$. This condition is mainly here to make statements and proofs easier. For instance, without this assumption our proof of (4) $\implies$ (2) may produce peripheral meridians, but peripheral meridians can't be used in (2) $\implies$ (1).

\begin {proof}
$(2) \implies (1)$.  Fix some transverse measure on $\lambda$. By $(2)$,  $$i(\gamma_i,\lambda) / \length(\gamma_i) \to 0,$$ so after passing to a subsequence  we can assume that $(\gamma_i)$  converges to a measured lamination $\mu$ in $S(\lambda)$ that  does not intersect $\lambda$  transversely. As $\lambda$ fills $S(\lambda)$, $\mu$ is supported on $\lambda $.

$(1) \implies (3)$.  After passing to a subsequence,  we can assume that $(\gamma_i)$ converges in the  Hausdorff topology to some lamination, which must then be  an extension of $\lambda $ by finitely many leaves. $(3)$ follows from an unpublished criterion of Casson, see Lecuire \cite[Th\'eor\`eme B.1]{Lecuireplissage} for a proof,  that states that any Hausdorff limit of meridians has a homoclinic leaf.

$(3) \implies (2)$. This is an immediate corollary of Lemma \ref{limitquasi}.

$(4) \iff (2)$.  The direction  $\Leftarrow$ is immediate, so  suppose $(\gamma_i)$  is a sequence of essential closed curves  in $S (\lambda) $  that are nullhomotopic  in $H $ and $i(\lambda,\gamma_i)\to 0$.  By Stallings' version of the Loop Theorem,  for each $i $ there is a meridian $\gamma_i'$ that is obtained from $\gamma_i$  by surgery at the self intersection points.  Such surgeries can only decrease the intersection number with $\lambda $, so $(2)$ follows.
\end {proof}

 We will also need the following criterion in the next section.

\begin{lemma}[Intrinsic limits of annuli]
 Suppose $\lambda \subset \partial_{\chi<0} M$ is a minimal lamination such that $S(\lambda)$ is compressible but $\partial S(\lambda)$ is incompressible, and that  there is a sequence $(A_i)$ of essential  embedded annuli  in $(M, S(\lambda))$  with $i (\partial A_i,\lambda)\to 0$.  Then $\lambda $ is an intrinsic limit of meridians.\label {annualilimit}
\end{lemma}

\begin {proof}
Pick a meridian $m  \subset S(\lambda)$. For each $i$, let $T_i : M \longrightarrow M$ be the Dehn twist along the annulus $A_i$. Then for any sequence $n_i \in \BZ$, the curves $T_i^{n_i}(m)$ are meridians, and if $n_i$ grows sufficiently fast, then $$i(T_i^{n_i}(m),\lambda)/\length(T_i^{n_i}(m)) \to 0.$$ Hence, after passing to a subsequence $T_i^{n_i}(m)$ converges to a lamination $\lambda'$ supported in $S(\lambda)$ with zero intersection number with $\lambda$, implying $\lambda'$ and $\lambda$ have the same support, so $\lambda$  is an  intrinsic limit of meridians.
\end {proof}

\section {Limits of homoclinic rays}

In this section we characterize the laminations onto which pairs of disjoint mutually homoclinic rays can accumulate.

\begin{theorem}[Mutually homoclinic rays]\label {travelers}
Let $M$ be a compact orientable hyperbolizable $3$-manifold and equip $\partial_{\chi<0} M$ with an arbitrary hyperbolic metric. Let $h_\pm$ be two disjoint, mutually homoclinic simple geodesic rays on  $\partial_{\chi<0} M$ that accumulate onto (possibly equal) minimal laminations $\lambda_\pm$, and that the multicurve $\partial S(\lambda_\pm)$ is incompressible in $M$. Then one of the following holds:
\begin{enumerate}
\item one of $\lambda_+$ or $\lambda_-$ is an intrinsic limit of meridians,
\item $ h_+$ and $h_-$ are asymptotic on $\partial_{\chi<0} M$, and either 
\begin{enumerate}
	\item {any} two mutually homoclinic lifts $\tilde h_\pm$ to $\partial \tilde M$ are asymptotic on $\partial \tilde M$, or
	\item $\lambda:=\lambda_-=\lambda_+$ is a simple closed curve that is homotopic in $M$ to a nontrivial power $\gamma^n, \ n>1$ of some closed curve $\gamma$ in $M$,
\end{enumerate}

\item $h_\pm$ are not asymptotic on $\partial_{\chi<0} M$, and there is an essential (possibly nontrivial) interval bundle $B \subset M$ such that $\lambda_\pm$ each fill a component of $\partial_H B$, and $\lambda_\pm$ are essentially homotopic through $B$, as in \S \ref{laminterval}.
\end{enumerate}
\end{theorem}

The proof of Theorem \ref{travelers} is given in \S \ref{travelerspf}.

\medskip

One can construct examples of mutually homoclinic rays in each of cases (1)--(3). For concreteness, suppose that $M$ is a handlebody.  For (1), pick two meridians $\lambda_-,\lambda_+$ on $M$ and let $h_\pm$ spiral onto $m_\pm$. One can also produce similar examples by letting $\lambda_\pm$ be arbitrary laminations in disjoint subsurfaces $S(\lambda_\pm)$ that are spheres with at least 4 boundary components, all of which are compressible in $M$, and letting $h_\pm$ accumulate onto $\lambda_\pm$. For (2) (a), take $\lambda$ to be any simple closed curve on $\partial M$ that is essential in $M$ but has no nontrivial roots in $\pi_1 M$, and let $h_\pm$ spiral around $\lambda$ in the same direction. We'll discuss 2 (b)  in Remark \ref{21remark}.  For (3), write $M = S \times [-1,1]$ where $S$ is a surface with boundary, let $\lambda$ be a lamination on $S$, and let $h_\pm$ be corresponding leaves of $\lambda_\pm := \lambda \times \{\pm 1\}$. Then $\tilde M \cong \tilde S \times [-1,1]$, so there are lifts of $h_\pm$ that are mutually homoclinic. One can also construct similar examples of (3) where the interval bundle $B$ is twisted. 

In case (1), we expect it is possible that $S(\lambda_-)$ is incompressible, say, while $\lambda_+$ is an intrinsic limit of meridians. For instance, suppose $C$ is a compression body with connected, genus-at-least-two interior boundary $\partial_- C$, and exterior boundary $\partial_+C$. Let $f : C \longrightarrow C$ be a homeomorphism such that $f|_{\partial_+ C}$ and $f|_{\partial_- C}$ are both pseudo-Anosov, with attracting laminations $\lambda_+,\lambda_-$, respectively. We expect that there are rays $\ell_\pm \subset \lambda_\pm$ that are mutually homoclinic. But $\lambda_+$ is an intrinsic limit of meridians, while $S(\lambda_-)$ is incompressible.

The assumption that $\partial S(\lambda_\pm)$ is incompressible is necessary in Theorem \ref{travelers}. For instance, suppose $M$ is a compression body with exterior boundary a genus $3$ surface $S$, where the only meridian on $S$ is a single separating curve $\gamma$. Let $T$ be the component of $S \setminus \gamma$ that is a punctured genus $2$ surface. Then there are distinct minimal geodesic laminations $\lambda,\lambda' \subset T$, each of which fills $T$, that are properly homotopic in $M$: just take distinct laminations that are identified when we cap off the puncture of $T$ to get a closed genus $2$ surface. Corresponding ends of corresponding leaves of $\lambda,\lambda'$ are mutually homoclinic rays that accumulate onto $\lambda,\lambda'$, respectively, but none of (2)--(3) hold. One could write down a version of Theorem \ref{travelers} that omits the assumption that $\partial S(\lambda_\pm)$ is incompressible, but the conclusion would be relative to capping off $S(\lambda_\pm)$, and the statement would be more complicated.

\begin{figure*}
\centering
\includegraphics{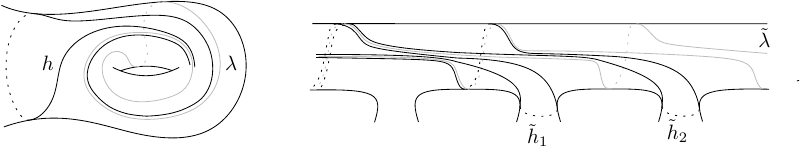}
\caption{An example of (2) (b) in Theorem \ref{travelers}, see Remark \ref{21remark}.}\label{case5}
\end{figure*}

\begin{remark} \label{21remark}  The reader may be wondering about (2) (b) in Theorem \ref{travelers}, and how it differs from (2) (a). A relevant example of two asymptotic rays that have mutually homoclinic nonasymptotic lifts to $\partial \tilde M$ is given in Figure~\ref{case5}. On the left we have a solid torus that is a boundary-connect-summand of $M$, which (say) is a handlebody. The biinfinite geodesic $h$ is homoclinic and its two ends are mutually homoclinic rays that both spiral onto a simple closed curve $\lambda$, the $(2,1)$-curve on the solid torus. Then $\lambda$ is homotopic to the square of the core curve of the solid torus. Although the two ends of $h$ are asymptotic on $\partial M$, any lift $\tilde h $ in $\partial \tilde M$ will have ends that are mutually homoclinic, but nonasymptotic. On the right, we have drawn the preimage $\tilde \lambda$ of $\lambda$, and two lifts $\tilde h_1,\tilde h_2$ of $h$. Note that one end of $\tilde h_1$ is asymptotic to an end of $\tilde h_2$. 

When $M$ is a compression body, Casson-Gordon prove in \cite[Theorem 4.1]{casson1987reducing} that any simple closed curve $\lambda\subset \partial M$ that has a nontrivial root in $\pi_1 M$ lies on the boundary of a solid torus that is a boundary connect summand of $M$, exactly as in Figure \ref{case5}. When $M$ has incompressible boundary, such $\lambda$ come from components of the characteristic submanifold of $M$, see \S \ref{sec:characteristic}, that are either solid tori or twisted interval bundles over nonorientable surfaces.
\end{remark}

Here is a slightly more refined version of Theorem \ref{travelers} that applies to homoclinic biinfinite geodesics on $\partial_{\chi < 0} M$.

\begin{corollary}[Homoclinic biinfinite geodesics]\label{travelers-biinfinite}
Suppose that $M$ is as in Theorem \ref{travelers}, that $h $ is a homoclinic biinfinite simple geodesic on some component $S \subset \partial_{\chi < 0} M$, that $h_\pm$ are the two ends of $h$, that $h_\pm$ limit onto $\lambda_\pm$, and that $\partial S(\lambda_\pm)$  is incompressible in $M$. 

Then one of (1)--(3) in Theorem \ref{travelers} holds. Moreover, in case (2), if $\lambda:=\lambda_\pm$ is not an intrinsic limit of meridians then either
\begin{enumerate}
\item[(i)] after reparametrizing $h$, there is some $s$ such that $h(-s)$ and $h(s)$ are joined by a geodesic segment $c$   with $h \cap int(c) =\emptyset$,  such that $c$, $ h |_{(-\infty,-s)} $ and $ h |_{[s,\infty)} $ bound an embedded geodesic triangle $\Delta \subset S$ with one ideal vertex, and  $c \cup h([-s,s])$ is a meridian in $M$, or
\item[(ii)] $\lambda$  is a simple closed curve on $S$, the two ends of $h$ spiral around $\lambda$ in the same direction, and any neighborhood of the union $h \cup \lambda \subset S$ contains a meridian. 
\end{enumerate}
And in case (3), we can choose the interval bundle $B$ such that $h$ contains a subarc $\alpha \subset h$ that is a compression arc for $B$.
\end{corollary}

\begin {figure}
\centering
\includegraphics{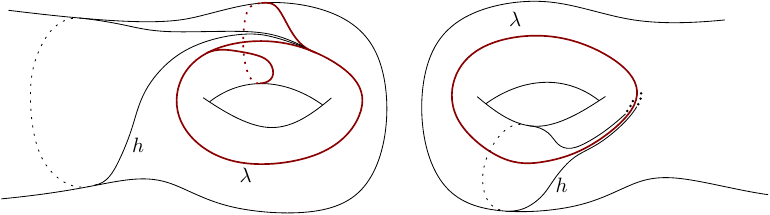}
\caption {Homoclinic geodesics $h$ as in cases (i) and (ii) in Corollary \ref{travelers-biinfinite}, respectively.}
\label {asymptoticfig}
\end {figure}

\begin{proof}
Let $\tilde h$ be a homoclinic lift of $h$ on $\partial \tilde M$. By Lemma \ref{limitquasi}, either one of $\lambda_\pm$ is an intrinsic limit of meridians in $M$, in which case we're in case (1) and are done, or both ends of $\tilde h$ are quasi-geodesic in $\tilde M$. Since $\tilde h$ is homoclinic, it follows that its two ends are  mutually homoclinic, so we're in the setting of Theorem~\ref{travelers} and one of (2)-(3) holds. 

\medskip

Assume we're in case (2) of Theorem~\ref{travelers}, and set $\lambda:=\lambda_\pm$. Assume first that $\lambda:=\lambda_\pm$ is a simple closed curve.  Since the two ends of $h$ are asymptotic, they spiral around $\lambda$ in the same direction. Let $U$ be a neighborhood of $h \cup \lambda $ on $\partial_{\chi<0} M$. Then $h$ is a homoclinic geodesic contained in $U$, so Lemma \ref{incompressible-qg} implies that $U$ is compressible as desired in (ii).

Now suppose that $\lambda$ is \emph{not} a simple closed curve, in which case we're in case (2) (a) of Theorem~\ref{travelers}. We show (i) holds. Let's start by constructing the desired geodesic triangle. Parametrize $h$, pick a universal covering map $\BH^2 \longrightarrow S$ and lift $h$ to a $\hat h$  in $\BH^2$, and let $$\xi = \lim_{t\to +\infty} \hat h(t) \in \partial_\infty \BH^2.$$ Since the two ends of $h$ are asymptotic on $S$, there is a deck transformation  $\gamma : \BH^2 \longrightarrow \BH^2$ such that 
$\xi = \lim_{t\to -\infty } \gamma \circ \hat h(t).$ It follows that if we use a particular arc-length parametrization of $h$, we may assume that for each $t \in \BR$, the points $\hat h(t),\gamma \circ \hat h(-t)$ lie on a common horocycle tangent to $\xi$. Fix some large $s$ such that the geodesic segment $\hat c$ joining $\hat h(s)$ and $\gamma  \circ \hat h(-s)$ is shorter than the injectivity radius of $S$, and therefore projects to a simple geodesic segment $c$ in $S$.

Let $\hat \Delta \subset \BH^2$ be the triangle bounded by $\hat c$ and the two rays $\hat h([s,\infty))$ and $\gamma \circ \hat h((-\infty,s])$.  Let $g : \BH^2 \longrightarrow \BH^2$ be a deck transformation. We claim that $g \circ \hat h(\BR) \cap int(\hat \Delta) = \emptyset$. If not, then since $\hat \Delta$ has geodesic sides, two of which are disjoint from $g \circ \hat h$, it follows that one of the two endpoints of $g \circ \hat h$ is $\xi$. If it's the positive endpoint, then $g$ fixes $\xi$, and the axis of $g$ projects to a (simple) closed curve on $S$, around which the two ends of $h$ spiral, contradicting that $\lambda$ isn't a simple closed curve. If the negative endpoint of $g \circ \hat h$ is $\xi$, then $g \circ \gamma^{-1}$ fixes $\xi$ and we get a similar contradiction.

Next, we claim that we have $g(int(\hat \Delta)) \cap int(\hat \Delta)=\emptyset$ as long as $g\neq id$. Suppose that for some $g\neq id$ the intersection is nonempty. Then  $g(\xi) \neq \xi$, since otherwise we have a contradiction as in the previous paragraph. The previous paragraph implies that the sides of the triangles $g(\hat \Delta) , \hat \Delta$ that are lifts of rays of $h$ do not intersect the interior of the other triangle. So, the only way the interiors of $g(\hat \Delta) , \hat \Delta$ can intersect is if $\hat c$ and $g(\hat c)$ intersect. However, this does not happen since we chose $s$ large enough so that $\hat c$ projects to a simple geodesic segment in $S$.

The previous two paragraphs imply that $\hat \Delta$ projects to an embedded geodesic triangle $\Delta$ in $S$ whose interior is disjoint from $h$, as desired in (i). By construction, $c$ and $h([-s,s])$ are simple geodesic segments and, since $g \circ \hat h(\BR) \cap int(\hat \Delta) = \emptyset$ for any $g\neq id$, they are disjoint. It follows that $c \cup h([-s,s])$ is an essential simple closed curve on $S$. Note that $c \cup h([-s,s])$ is an essential simple closed curve on $S$, since it is the concatenation of two geodesic segments with disjoint interiors. We need to show it is nullhomotopic in $M$. Now, if $\tilde h$ is a lift of $h$ to $\partial \tilde M$, its ends are mutually homoclinic, and hence are asymptotic on $\partial \tilde M$ by the assertion in case (2) of Theorem \ref{travelers}. Therefore, after choosing compatible lifts, the projection $\hat \Delta \longrightarrow
\Delta$ factors through a geodesic triangle $\tilde \Delta \subset \partial \tilde M$ bounded by $\tilde h([s,\infty))$, $\tilde h((-\infty,-s])$ and a geodesic segment $\tilde c$. The curve $c \cup h([-s,s])$ is the projection of the closed curve $\tilde c \cup \tilde h([-s,s]) \subset \tilde M$,  and therefore is nullhomotopic in $M$.

\medskip

Now assume we are in case (3). Let $S_\pm$ be the components of $\partial_H B$ containing $\lambda_\pm$. We may assume that $h$ is in minimal position with respect to $\partial S_\pm$. Since $h$ is simple and the ends of $h$ limit onto minimal laminations that fill $S_\pm$, we have that $h$  intersects $\partial S_- \cup \partial S_+$ at most twice. Furthermore, in the case that $S_-=S_+$, the homoclinic geodesic $h$ cannot be contained entirely in the incompressible surface $S_\pm$, by Fact \ref{incompressible-qg}. So, $h$ is the concatenation of two rays in $S_\pm$ and an arc $\alpha$ such that $int(\alpha)$ lies outside $S_\pm$. 

Let $X\subset S$ be the union of $S_\pm$ and a regular neighborhood of $\alpha$. Since $h $ is homoclinic, there is a meridian on $X$ by Fact~\ref{incompressible-qg}. If the two endpoints of $\alpha$ lie on different boundary components of $\partial_H B$, then $\alpha$ is a compression arc for $B$ by Fact~\ref{compressionfact}. So, we may assume that the two endpoints of $\alpha$ lie on the same boundary component $c $ of $ \partial_H B$. Fact~\ref{compressionfact} then says that $\alpha$ is homotopic rel endpoints in $M$ to an arc of $c$. So, if we make a new path $h' \subset \partial_H B$ from $h$ by replacing $\alpha$ with that arc of $c$, then $h'$ is still homoclinic, so it cannot be boundedly homotopic to a geodesic in $\partial_H B$ by Fact~\ref{incompressible-qg}, which implies that its ends $h_\pm$ are asymptotic, a contradiction to the assumption in (3).
\end{proof}

\subsection{Proof of Theorem \ref{travelers}}
\label{travelerspf}
The proof proceeds in a few cases. As in Example \ref{explicitmetric}, we identify $M$ with the convex core $CC(N)$ of a geometrically finite hyperbolic $3$-manifold $N=\BH^3 / \Gamma$, and we identify the universal cover $\tilde M$ with the preimage of $CC(N)$ in $\BH^3$, which is the convex hull of the limit set of $\Gamma$. Note that the closure of $\tilde M$ in $\BH^3 \cup \partial \BH^3$ is a ball.

 There are four cases to consider:

\begin{enumerate}
	\item[(A)] Both $\lambda_\pm$ are simple closed curves. We show that either (1) or (2) holds.
	\item[(B)] Both $\lambda_\pm$ are distinct, in which case the surfaces $S(\lambda_\pm)$ are disjoint, but one of these surfaces is compressible, say $ S (\lambda_+) $. We show (1).
	\item[(C)] At least one of $\lambda_\pm$ is not a simple closed curve, and both $S(\lambda_\pm)$ are incompressible. We show (2) (a)  or (3) holds.
	\item[(D)] $\lambda_- = \lambda_+$, which is not a simple closed curve, and $S(\lambda_\pm)$ is compressible. We show either (1) or (2) (a) holds.
\end{enumerate}
(A) and (B) above are the easiest. Our proof in case (C) involves a hyperbolic geometric interpretation of the characteristic submanifold of a pair,  as discussed in \S 3 of 
\cite{lecuire2004structures} and in Walsh~\cite{Walshbumping}; our argument is a bit more complicated than theirs, since we have to deal with accidental parabolics.  In case (D), our argument adapts and fills some gaps in a surgery argument of Kleineidam--Souto \cite{Kleineidamalgebraic} and Lecuire~\cite[Appendix C]{Lecuireplissage}. 

\medskip

\noindent \it Proof of (A). \rm  Assume that both $\lambda_\pm$ are simple closed curves.  If one of $\lambda_\pm$ is a meridian, we are in case (1) and are done. So, we may assume that both $\lambda_\pm$ are incompressible in $M$. If $\lambda_-\neq \lambda_+$, then we are in case (3) by the Annulus Theorem. So we may assume the two curves are the same, and write $\lambda:=\lambda_\pm$. 

We claim that $h_\pm$ spiral around $\lambda$ in the same direction, so that they are asymptotic on $\partial M$. Suppose not, and pick mutually homoclinic lifts $\tilde h_\pm$ in $\tilde M$. Then $\tilde h_-$ and $\tilde h_+$ are asymptotic to lifts $\tilde \lambda$ and $ \alpha(\tilde \lambda)$ of $\lambda$, where $\alpha \in \Gamma $ is a deck transformation. Any lift of $\lambda$ is a quasi-geodesic in $\tilde M$, and hence in $\BH^3$, so $\tilde h_\pm$ are quasi-geodesic rays, and therefore have well-defined endpoints in $\partial \BH^3$, which must be the same since the two rays are mutually homoclinic. Since $h_\pm$ spiral around $\lambda$ in opposite directions, this means that $\alpha \in \Gamma$ takes one endpoint of $\tilde \lambda$ in $\partial \BH^3$ to the other endpoint of $\tilde \lambda$. Since $\tilde \lambda$ is stabilized by a loxodromic isometry in $\Gamma$, and $\Gamma$ is torsion-free and discrete, this is impossible.

Suppose we are not in case (2) (a), so there are mutually homoclinic lifts $\tilde h_\pm$ that are not asymptotic on $\partial \tilde M$. As in the previous paragraph, we may assume that $\tilde h_-$ and $\tilde h_+$ are asymptotic to lifts $\tilde \lambda$ and $\alpha(\tilde \lambda)$ for some deck transformation $\alpha \in \Gamma$. Since $\tilde h_\pm$ are not asymptotic, $\tilde \lambda \neq \alpha(\tilde \lambda)$. As before, $\alpha $ fixes the common endpoint of $\tilde h_\pm$ in $\partial \BH^3$, which is a fixed point of the cyclic group $\langle \beta \rangle \subset \Gamma$ of loxodromic isometries fixing $\tilde \lambda$. As $\Gamma$ is discrete and torsion-free, and $\alpha \not \in \langle \beta \rangle $, we have that $\alpha$ is a root of $\beta$ or $\beta^{-1}$ in $\Gamma$, and (2) (b) follows.

\medskip

\noindent \it Proof of (B). \rm  Suppose that $\lambda_\pm$ are distinct, in which case the surfaces $S(\lambda_\pm)$ are disjoint, but that one of these surfaces is compressible, say $ S (\lambda_+) $. We claim that $\lambda_+$ is an intrinsic limit of meridians, in which case (1) holds and we are done. If not, take a meridian $m \subset S(\lambda_+) $ and apply  Lemma \ref {tightcuts}. We obtain a new meridian  $m' \subset S(\lambda_+) $  such that $\lambda_+$ has no $m$-waves. Since $\lambda_+$ fills $S(\lambda_+)$ and the boundary components $\partial S(\lambda_\pm)$  are incompressible, it follows that $\lambda_+$ is in tight position with respect to $m$. So after possibly restricting the domains, $h_+$  is in tight position with respect to $m'$, while $h_-$  never intersects $m'$.  This  contradicts the fact that $h_\pm$  are mutually homoclinic, since if $\tilde h_\pm$ are homoclinic lifts in $\tilde M$, for large $t$ the point $\tilde h_+(t)$ is  separated from  the image of $\tilde h_-$ by  arbitrarily many lifts of $m'$. 

\medskip

\noindent \it Proof of (C). \rm  Assume that at least one of $\lambda_\pm$ is not a simple closed curve, and that $S_\pm := S(\lambda_\pm)$ are incompressible.  Note that $S_\pm$ are equal or have disjoint interiors. We want to prove that we're in case (2) or (3).  Lift $h_\pm$ to mutually homoclinic rays $\tilde h_\pm \subset \partial \tilde M$. Fact~\ref{incompressible-qg} implies that each inclusion $\tilde S_\pm \hookrightarrow \tilde M$ is a quasi-isometric embedding, so if $\tilde S_- = \tilde S_+$, then the two mutually homoclinic rays $\tilde h_\pm$ are actually asymptotic on $\partial \tilde M$. If this is true for all lifts $\tilde h_\pm$, we are in case (2) (a) and are done. So, we can assume below that $\tilde S_- \neq \tilde S_+$. Note that it may still be that $\lambda_-=\lambda_+$ and $S_-=S_+$.

Let $\Gamma_\pm \subset \Gamma$ be the stabilizer of $\tilde S_\pm$ and let $\Lambda_\pm \subset \partial\BH^3$ be the limit set of $\Gamma_\pm$. Since $\Gamma_\pm$ acts cocompactly on $\tilde S_\pm$, the inclusion $\tilde S_\pm \hookrightarrow \BH^3$ extends continously to a map $\partial_\infty \tilde S_\pm \longrightarrow \Lambda_\pm \subset \partial \BH^3$, by the main result of \cite{mj2017cannon}. In particular, $\tilde h_\pm$ have well defined endpoints in $\partial \BH^3$, and since they're mutually homoclinic, they have the same endpoint $\xi \in \Lambda_- \cap \Lambda_+ \subset \partial \BH^3$. 

We now apply Theorem \ref{windowsthm}. Since $\xi \in \Lambda_- \cap \Lambda_+$, using the notation of Theorem \ref{windowsthm}, the rays $\tilde h_\pm$ are either eventually contained in the convex hulls $\tilde C_\pm \subset \tilde S_\pm$, or are asymptotic onto their boundaries. But $\tilde C_\pm$ project to (possibly degenerate) generalized subsurfaces $C_\pm$ with geodesic boundary in $S_\pm$, and the rays $h_\pm$ limit onto filling laminations in $S_\pm$, so it follows that actually $C_\pm=S_\pm$, and that there is a homotopy from $S_-$ to $S_+$ in $M$ that is the projection of a homotopy from $\tilde S_-$ to $\tilde S_+$. Since one of $\lambda_\pm$ is not a simple closed curve, this means they are \emph{both} not simple closed curves and the (a priori degenerate) subsurfaces with geodesic boundary $S_\pm$ are not simple closed geodesics.

Let $S_\pm' \subset int(S_\pm)$ be obtained by deleting small collar neighborhoods of $\partial S_\pm$, so that $S_\pm'$ are both actually embedded, still contain $\lambda_\pm$, and are either disjoint or equal. Since $S_\pm'$ are incompressible and homotopic in $M$, Theorem~\ref{charsubthm} implies that they bound an essential interval bundle $B\subset M$. Moreover, the fact that the homotopy from $S_-$ to $S_+$ is the projection of a homotopy from $\tilde S_-$ to $\tilde S_+$ means that we can assume that there is a component $\tilde B \subset \tilde M$ of the preimage of $B$ that intersects $\partial \tilde M$ in $\tilde S_\pm'$. Note that $\tilde B$ is invariant under $\Delta =\Gamma_- \cap \Gamma_+$, since any element of $\Delta$ preserves $\tilde S_\pm'$, and hence $\tilde B$.

We claim that $\lambda_\pm$ are essentially homotopic through $B$. By Fact~\ref{homotopy involutoin}, it suffices to show that if $\sigma$ is the canonical involution of $B$, as described in  \S \ref{intervalbundle}, then $\sigma(\lambda_\pm) $ is isotopic to $\lambda_\mp$ on $S_\mp'$. Well, $\sigma$ lifts to a $\Delta$-equivariant involution $\tilde \sigma$ of $\tilde B$ that exchanges $\tilde S_-'$ and $\tilde S_+'$, where here $\Delta = \Gamma_-\cap \Gamma_+$. By equivariance, $\tilde \sigma$ extends continuously to the identity on $\Lambda_\Delta$, so in particular its extension fixes $\xi$, and hence $\tilde \sigma(h_\pm)$ is properly homotopic to $h_\mp$ on $S_\mp$, which implies $\tilde \sigma(\lambda_\pm)$ is isotopic to $\lambda_\mp$ as desired.

If $h_\pm$ are not asymptotic on $\partial M$, then we are in case (3) and are done. So, assume $h_\pm$ are asymptotic. Then there is some $\gamma \in \Gamma$ such that $\gamma(\tilde h_-) \subset \tilde S_+'$ and is asymptotic to $\tilde h_+$. This $\gamma$ fixes the endpoint $\xi\in\partial \BH^3$ of $\tilde h_\pm$. Moreover, $\gamma(\tilde B) $ is a component of the preimage of $B$ that contains $S_+'$, and therefore equals $\tilde B$. So, $\gamma$ exchanges $\tilde S_\pm'$, and therefore $\gamma^2 \in \Delta$. But then $\tilde h_\pm$ are asymptotic to the axes of $\gamma^2 \actson \tilde S_\pm$, implying that $ h_\pm$ accumulate onto simple closed curves in $\partial M$, contradicting our assumption in (C).

\medskip

\noindent \it Proof of (D). \rm Assume that $\lambda_-=\lambda_+$, write $\lambda=\lambda_\pm$ for brevity, assume that $\lambda$ is not a simple closed curve, and that $S(\lambda)$ is compressible.  We want to prove that  either $\lambda $ is an intrinsic limit of meridians, or $h_\pm$ are asymptotic, as are any pair of mutually homoclinic lifts $\tilde h_\pm$.

 If $\lambda$ is an intrinsic limit of meridians, we are done, so since $S(\lambda)$ is compressible with incompressible boundary, by Lemma \ref{tightcuts}  we can choose a meridian $m \subset S(\lambda)$ with respect to which $\lambda $ is in tight position.  Let $\tilde m$ be its  full preimage in $\partial \tilde M$, and let $\tilde h_\pm$ be \emph{any} pair of mutually homoclinic lifts in $\partial \tilde M$. Truncating if necessary, we can assume that $h_\pm$ are in tight position with respect to $m$, and hence the lifts $\tilde h_\pm$ are quasigeodesic rays in $\BH^3$, by Fact~\ref{tightquasi}.  Since they are mutually homoclinic, $\tilde h_-$ and $\tilde h_+$ converge to the same point $\xi \in \partial_\infty \BH^3$, and tightness further implies that after restricting to appropriate subrays, $\tilde h_-$ and $\tilde h_+$ intersect exactly the same components of $\tilde m$, in the same order. Reparametrizing, we have
$$\tilde h_\pm : [0,\infty) \longrightarrow  \partial \tilde M, \ \ \tilde h_+ (i), \tilde h_-(i) \in \tilde m_i, \ \forall i\in \BN,$$
where each $\tilde m_i$ is a component of $\tilde m$, and where $\tilde h_\pm(t) \not \in \tilde m$ when $t\not \in \BN$. Let $$d_i := d_{\tilde m} \big ( \tilde h_+ (i) , \tilde h_- (i) \big )$$ be the distance along $\tilde m$  between $\tilde h_+ (i)$ and $ \tilde h_-(i)$.

\medskip

\begin{claim}\label{case1}
There is some  uniform $\epsilon>0$,  independent of the particular chosen lifts $\tilde h_\pm$, such that either
\begin{enumerate}
	\item $\tilde h_\pm$ are asymptotic on $\partial \tilde M$, and hence $h_\pm$ are asymptotic on $\partial M$, or
\item $\liminf_i d_i \geq \epsilon$.
\end{enumerate}
\end{claim}
\begin{proof}
	Let's  assume that $\tilde h_+$ and $\tilde h_-$ are not asymptotic on $\partial \tilde M$, and write $d = \liminf_i d_i $. Fix some transverse measure on $\lambda$. If $d$ is small, we will construct meridians $\gamma \subset S(\lambda)$ with very small intersection number with $\lambda$. Since $\lambda$ is not an intrinsic limit of meridians, there is some fixed lower bound for such intersection numbers, which will give a contradiction for small $d$.

Suppose $d$ is small and pick $0<<i<j$  such that $$d_i,d_j < 2d,$$
 let $b_i$ be the (unique) shortest path on $\tilde m$ from $\tilde h_-(i)$ to $\tilde h_+(i)$, and define $b_j$ similarly. Let $\tilde  \gamma_{ij}$ be the loop on $\partial \tilde M$ obtained by concatenating  the four segments $\tilde h_+([i,j]), \tilde h_-([i,j]), b_i $ and $b_j$ in the obvious way. 

 We first claim that after fixing $i$, it is possible to choose $j$ such that $\tilde \gamma_{ij}$ is homotopically essential on $\partial \tilde M$.  Assume not, let $\tilde S \subset \partial \tilde M$ be the component containing $\tilde h_\pm$, fix a universal covering map $$\BH^2 \longrightarrow \tilde S,$$ and lift the rays $\tilde h_\pm |_{[i,\infty)}$  to rays $$\mathfrak h_\pm : [i,\infty) \to \BH^2$$ in such a way that $b_i$ lifts to a segment connecting $\mathfrak h_-(i)$ to $\mathfrak h_+(i)$. Now, there are infinitely many $j>i$ with $d_j<2d$. For each such $j$, we know that $\tilde \gamma_{ij}$ is homotopically inessential on $\partial \tilde M$, so the points $\mathfrak h_-(j)$ to $\mathfrak h_+(j)$ are at most $2d$ away from each other in $\BH^2$. This gives a sequences of points exiting the rays $\mathfrak h_\pm$ that are always at most $2d$ apart, so $\mathfrak h_-$ is asymptotic to $\mathfrak h_+$. Hence, $\tilde h_-$ is asymptotic to $\tilde h_+$, contrary to our assumption.

We now fix large $i,j$ such that $d_i,d_j<2d$ and $\tilde \gamma:=\tilde \gamma_{ij}$ is homotopically essential on $\partial \tilde M$. Then $\tilde \gamma$ projects to a homotopially essential loop $\gamma \subset \partial M$ that is homotopically trivial in $M$.  Note that if $i,j$ are chosen large enough and $d$ is small, then $\gamma \subset S(\lambda)$.  Furthermore, since the segments $b_i,b_j$ are the only parts of $\tilde \gamma$ that intersect $\lambda$, and these segments have  hyperbolic length less than $2d$, the intersection number $i(\gamma,\lambda)$ is small when $d$ is small. (Recall that $\lambda $ is a minimal lamination that is not a simple closed curve, so no leaves have positive weight, and hence hyperbolic length can be compared to intersection number.) But $\lambda$  is not an intrinsic limit of meridians, so Proposition \ref{intrinsiclimits} (4) says that there is some positive lower bound for the intersection numbers of $\lambda$ with essential curves that are nullhomotopic in $M$. Hence, we get a contradiction if $d$ is small.
\end{proof}

\smallskip

Suppose we have two pairs $\{a,b\}$ and $\{c,d\}$ of points in $m$, all four of which are distinct. We say the two pairs are \emph{unlinked} in $m$ if in the induced cyclic ordering on $\{a,b,c,d\}\subset m$, $a$ is adjacent to $b$ and $c$ is adjacent to $d$, and we say that the two pairs are \emph {linked} otherwise.

\begin{claim}\label {unlinked}
If $i,j\in \BN$, $i< j$, then the pairs $\{h_+(i),h_-(i)\}$ and $\{h_+(j),h_-(j)\}$ are unlinked in $m$. \end{claim}

For an example where the pairs are \emph{linked}, see Figure \ref{spiral}.  The proof below works  in general whenever $h_\pm$ are simple geodesic rays on $\partial M$ in tight position with respect to $m$, where neither $h_+$ nor $h_-$ spirals onto a simple closed curve.

\begin{figure*}
	\centering
\includegraphics{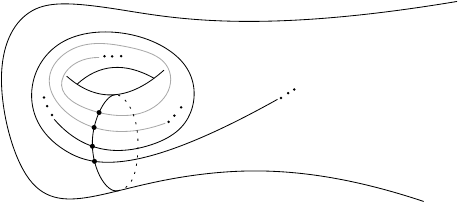}
\caption{Two rays spiraling onto a simple closed curve (which is not allowed below), where the points in Claim \ref{unlinked} are linked.}\label{spiral}
\end{figure*}

\begin {proof}
The essential observation used in the proof is that the closure $$cl( \partial \tilde M) \subset \BH^3 \cup \partial \BH^3$$ is homeomorphic to a sphere: indeed, the closure of $\tilde M$ in $\BH^3\cup \partial \BH^3$ is a ball, since $\tilde M \subset \BH^3$ is convex with nonempty interior, and the closure of the boundary is the boundary of the closure. We obtain the unlinking property above by exploiting separation properties of arcs and curves on $cl(\partial \tilde M)$.

Since $h_\pm$ are in tight position with respect to $m$, both lifts $\tilde h_\pm$ cross  $\tilde m_i$ exactly once. Since $\tilde h_\pm$ limit to the same point in $\partial \BH^3$, they must then cross $\tilde m_i$ in the same direction. In other words, the tangent vectors $h_+(i)',h_-(i)'$ point to the same side of $m$. The same statement holds for $j$. This allows us to break into the following two cases:
\begin{itemize}
\item[(a)] the arcs $h_\pm |_{[i,j]}$ start and end on the same side of $m$, i.e.\ the vectors $h_\pm(i)'$ point to the opposite side of $m$ as the vectors $h_\pm(j)'$, or  
\item[(b)] the arcs $h_\pm |_{[i,j]}$ start and end on different sides of $m$, i.e.\ all four velocity vectors $h_\pm(i)',h_\pm(j)'$ point to the same side of $m$,

\end{itemize}
see Figure \ref{sidesfig}.

\begin{figure}
\centering
\includegraphics{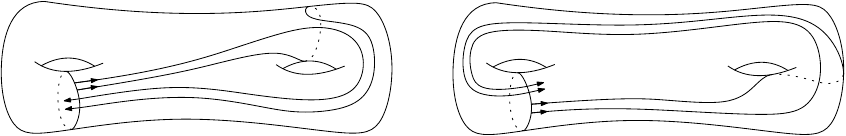}
\caption{The cases (a) and (b) in the proof of Claim \ref{unlinked}.}
\label{sidesfig}
\end{figure}
First, assume we're in case (a). Let 
$$\alpha_\pm := \tilde h_\pm |_{[i,j]},$$ which we regard as oriented arcs in $\partial \tilde M$ {starting} on $\tilde m_i$ and ending on $\tilde m_j$. Let $\gamma : \tilde M \longrightarrow \tilde M$ be the deck transformation taking $\tilde m_j$ to $\tilde m_i$. Then the arcs $$\beta_\pm :=\gamma \circ \tilde h_\pm |_{[i,j]}$$ start on $\gamma(\tilde m_i)$ and end on $\gamma(\tilde m_j)=\tilde m_i$, and since we're in case (a) they end on the \emph{same side} of $\tilde m_i$ as the arcs $\alpha_\pm$ start. Note that $\gamma(\tilde m_i) $ is not $ \tilde m_i$ or $\tilde m_j$. Indeed, if $\gamma(\tilde m_i)=\tilde m_i$ then we'd have $\gamma=id$, contradicting that $\gamma(\tilde m_j)= \tilde m_i$. And if $\gamma(\tilde m_i)=\tilde m_j$, then $\gamma^2$ would leave $\tilde m_i$ invariant, implying that $\gamma^2=id$, which is impossible since $\pi_1 M$ has no torsion.

We claim that \emph{the {interiors} of the arcs $\beta_\pm$ do not intersect $\tilde m_i$ or $\tilde m_j$, and the arcs $\alpha_\pm$ do not intersect $\gamma(\tilde m_i)$}. Indeed, the interiors of $\beta_\pm$ don't intersect $\tilde m_i$ because the arcs $\beta_\pm$ end on $\tilde m_i$ and intersect each component of $\tilde m$ at most once, by tight position. The interiors of $\beta_\pm$  don't intersect $\tilde m_j$ because any arc from $\tilde m_i$ to $\tilde m_j$ intersect at least $j-i+1$ components of $\tilde m$ (counting $\tilde m_i$ and $\tilde m_j$), while any proper subarc of $\beta_\pm$ intersects at most $j-i$ components of $\tilde m$. Here, for the $j-i+1$ bound we are using tight position of $h_\pm$, the definitions of $\tilde m_i,\tilde m_j$, and the fact that each component of $\tilde m$ separates $\partial \tilde M$. The fact that the arcs $\alpha_\pm$ don't intersect $\gamma(\tilde m_i)$ is similar: any arc from $m_i$ to $\gamma(\tilde m_i)$ must pass through at least $j-i+1$ components of $\tilde m$, while any proper subarc of $\alpha_\pm$ intersects at most $j-i$ components, and $\alpha_\pm$ do not end on $\gamma(\tilde m_i)\neq \tilde m_j$.

Let $A \subset cl(\partial \tilde M)\cong S^2$ be the annulus that is the closure of the component of $cl(\partial \tilde M) \setminus (\tilde m_i \cup \tilde m_j)$ that contains the side of $\tilde m_i$ on which the arcs $\alpha_\pm$ start and the arcs $\beta_\pm$ end. Then $\alpha_\pm$ are two disjoint arcs in $A$  that join $\tilde m_i $ to $\tilde m_j$, and therefore $\alpha_\pm$ separate $A$ into two rectangles.  The component $\gamma(\tilde m_i)$ on which the arcs $\beta_\pm$ start is contained in the interior of one of these two rectangles. Therefore, the two arcs $\beta_\pm$ must lie in the same component of $A \setminus (\alpha_+ \cup \alpha_-).$ Looking at endpoints, this means the pairs $\{\tilde h_+(i),\tilde h_-(i)\} $ and $\{\gamma \circ \tilde h_+(j),\gamma \circ \tilde h_-(j)\}$ are unlinked in $\tilde m_i$, and the claim follows.

\medskip

Now assume that we're in case (b). The curve $\tilde m_i$ separates $\partial\tilde M$, and we let $X\subset \partial\tilde M$ be the closure of the component of $\partial \tilde M \setminus \tilde m_i$ into which the velocity vectors $\tilde h_\pm'(i)$ and $(\gamma \circ \tilde h_\pm)'(j)$ all point. The closure $$ cl(X)  \subset \BH^3 \cup \partial \BH^3$$ is homeomorphic to a disk, since $cl(\partial \tilde M)$ is a sphere. As before, we let $\gamma : \tilde M \longrightarrow \tilde M$ be the deck transformation taking $\tilde m_j$ to $\tilde m_i$.  Then the rays $$\alpha_\pm := \tilde h_\pm([i,\infty)), \ \ \ \beta_\pm:=\gamma \circ \tilde h_\pm([j,\infty))$$  are all contained in $X$. Note that $\alpha_\pm$ both limit to a point $\xi\in \partial \BH^3$, while $\beta_\pm$ limit to $\gamma(\xi) \in \partial \BH^3$. 

The union $\alpha_-\cup \alpha_+$  compactifies to an arc in $cl(X)$, since the two rays limit to the same point in $\BH^3$. The same is true for $\beta_-\cup \beta_+$. \emph{Hoping for a contradiction, suppose that the points in the statement of the claim are linked.} Then the pairs of endpoints of $\alpha_- \cup \alpha_+$ and $\beta_- \cup \beta_+$ are also linked on $\tilde m_i = \partial cl(X)$. We now have two arcs on the disk $cl(X)$ with linked endpoints on $\partial cl(X)$, so the two arcs must intersect. As $\alpha_\pm,\beta_\pm$ are all disjoint, the only intersection can be on $\partial \BH^3$, so their endpoints at infinity must all agree, i.e.\ $\gamma(\xi)=\xi$.

 Since $\gamma(\xi)=\xi$, all the rays $ \gamma^k \circ \tilde h_+$ limit to $\xi$, where $k\in \BZ$. Hence, all these (quasi-geodesic) rays are pairwise mutually homoclinic, and for each pair $k,l$, the rays $ \gamma^k \circ \tilde h_+$ and $ \gamma^l \circ \tilde h_+$ eventually intersect the same components of $\tilde m$, in the same order, although their initial behavior may be different.  In analogy with the setup of Claim \ref{case1}, let $d_{k,l}$ be the liminf of the distances from $\gamma^k \circ \tilde h_+$ to $\gamma^l \circ \tilde h_+$ along the components of $\tilde m$ that they both intersect. 

We claim that there are $k,l$  such that $d_{k,l} < \epsilon$, where $\epsilon$ is the constant from Claim \ref{case1}. Indeed, for $N> \length(m)/\epsilon$, it is impossible to pack $N$ points at least $\epsilon$ apart in any component of $\tilde m$. So if we let $k$ range over a set $F \subset \BZ$ of size $N$, whenever a component of $\tilde m$ intersects all $\gamma^k \circ \tilde h_+, \ k\in S$, two such intersections must be within $\epsilon$ of each other. There are infinitely many such components of $\tilde m$ and $F$ is finite, so we can pick $k,l \in S$ such that $\gamma^k \circ \tilde h_+$ and $\gamma^l \circ \tilde h_+$ are within $\epsilon$ on infinitely many such components.

Finally, $ \gamma^k \circ \tilde h_+$ and $ \gamma^l \circ \tilde h_+$ are mutually homoclinic lifts of $\tilde h_+$, and $d_{k,l} < \epsilon$, so the exact same argument as in Claim \ref{case1}  shows that $ \gamma^k \circ \tilde h_+$ and $ \gamma^l \circ \tilde h_+$ are asymptotic on $\partial \tilde M$. It follows that $ h_+$ spirals onto a (simple) closed curve in $\partial M$ in the homotopy class of (a primitive root of) $\gamma^{l-k}$. (Indeed, $\gamma^{l-k}$ lifts to a deck transformation of the universal cover $\BH^2 \longrightarrow \partial M$, and the axis of this deck transformation is asymptotic to suitably chosen lifts of both $ \gamma^k \circ \tilde h_+$ and $ \gamma^l \circ \tilde h_+$.) This is a contradiction, though, since  $ h_+$ limits onto $\lambda$, which is not a simple closed curve.
\end {proof}

Assume now that our mutually homoclinic rays $\tilde h_\pm$ are not asymptotic on $ \partial \tilde M$, in which case we're in case (2) of the theorem and are done. By Claim~\ref{case1}, there is some $\epsilon>0$ such that $d_{\tilde m_i}( \tilde h_+(i),\tilde h_-(i)) \geq \epsilon$ for all $i$. We will show that  $\lambda $ is an intrinsic limit of annuli, in the sense of  Lemma \ref {annualilimit}, which says that then $\lambda$ is an intrinsic limit of meridians.

The proof is an adaptation and correction of a surgery argument of Lecuire \cite[Affirmation C.3]{Lecuireplissage}. As there are two gaps\footnote{The first gap is that the sentence \emph{``Quitte \`a extraire, la suite $(gh^{-1})^{2n} g(\tilde l^1)$ converge vers une g\'eod\'esique $\tilde  \gamma \subset p^{-1}(\alpha_1)$ dont la projection $l \subset \partial M$ est une courbe ferm\'e.''}\  at the end of the proof of Affirmation C.3 isn't adequately justified; this is fixed in Claim \ref{case1}. The second  is that the assumption $d(l_+^2(y_i),l_+^2(y_j))<\varepsilon'$  in the statement of Affirmation C.3 is never actually verified, and does not come trivially from a compactness argument.  This is fixed in Claim \ref{affclaim}.} in Lecuire's earlier argument, we give the proof in full detail below, without many citations of \cite{Lecuireplissage}.

\begin{claim} \label{affclaim}
Given $0<\delta<\epsilon$, there are choices of $i<j$  such that  either
\begin{enumerate}
\item[(I)] The points $h_+(i)$ and $h_+(j)$ bound a segment $I_+ \subset m$ of length less than $\delta$, and similarly with $-$ instead of $+$. The four velocity vectors $h_+'(i), h_+'(j), h_-'(i), h_-'(j)$ all point to the same side of $m$, and the four segments $h_+([i,j])$, $h_-([i,j])$, $I_+$ and $I_-$ have disjoint interiors. So, the curves $\gamma_\pm:= h_\pm([i,j]) \cup I_\pm \subset \partial M$ are simple and disjoint.
\item[(II)] The points $h_+(i)$ and $h_-(j)$ bound a segment $I_+ \subset m$ of length less than $\delta$, and  similarly the points $h_-(i)$ and $h_+(j)$ bound a segment $I_- \subset m$ of length less than $\delta$. The four velocity vectors $h_+'(i), h_+'(j), h_-'(i), h_-'(j)$ all point to the same side of $m$, and the four segments $h_+([i,j])$, $h_-([i,j])$, $I_-$ and $I_+$ have disjoint interiors. So, the curve $\gamma\subset \partial M$ obtained by concatenating all four segments is simple.

\end{enumerate}
\end{claim}

See Figure \ref{surgery} for a very useful picture. Note that in the picture, the velocity vectors of all paths intersecting $m$ point to the same side of $m$, i.e.\ `up', and all 4-tuples of points are unlinked as in Claim \ref{unlinked}.

\begin{figure}[t]
	\centering
\includegraphics {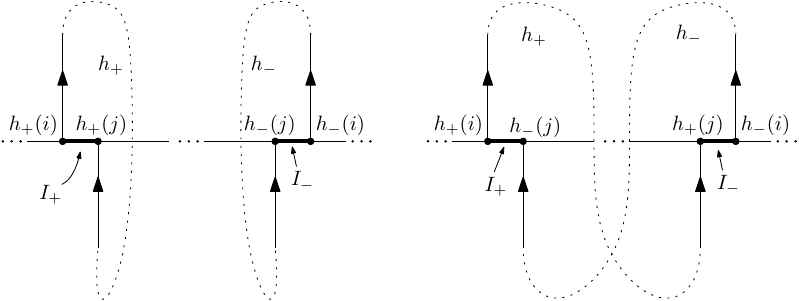}
\caption {Above, the horizontal curve is always $m$.}
\label{surgery}
\end{figure}

\begin{proof}

Start by fixing a circular order on $m$. Define `the right' to be the direction in $m$ that  is increasing with respect to the circular order, and define `the left' similarly.
Since $\lambda$ is minimal and not a simple closed curve, it has infinitely many leaves $\ell$ that are not  boundary leaves. Fix some such $\ell$, making sure that $h_\pm \not \subset \ell$ if the given rays happen to lie inside the lamination $\lambda$. The ray $h_+$  accumulates onto both sides of $\ell$, so if we fix $p \in \ell \cap m$, the set $h_+(\BN)$ accumulates onto $p$ from both sides, and similarly with $-$ instead of $+$. Fix an interval $J \subset m$ of length $\delta$ centered at $p$, and write $J=J_l \cup J_r$ as the union of the closed subintervals to the `left' and to the `right' of $p$. Note that $p \not \in h_\pm(\BN)$, so each intersection of $h_\pm$ with $J$ lies in exactly one of $J_l$ or $J_r$.

Let's call an index $i$ \emph{left-closest} if either $h_+(i)$  or $h_-(i)$ lies in $J_l$ and is closer to $p$ than any previous $h_\pm(k), \ k<i$, that lies in $J_l$. \emph{Right-closest} is defined similarly using $J_r$, and we call an index $i$ \emph{closest} if it is either left or right closest. Note that since $\delta<\epsilon$ we can never have both $h_+(i),h_-(i)$ in $J$ simultaneously, so no $i$ is both right-closest and left-closest at the same time. Since there are infinitely many $i$ of both types, at some point there will be a transition where some $i_l$ is left-closest, some $i_r>i_l$ is right-closest, and there are no closest indices in between. 

Let $i_c$ be the smallest closest index that is bigger than $i_r$. (Here, $c$ stands for `center', since the corresponding point on $J$ will lie between the points we get from  the indices $i_l$ and $i_r$.) We now have \emph{three} points in $J$, so two of the corresponding velocity vectors point to the same side of $m$. Let $i,j \in \{i_l,i_r,i_c\}$ be the two corresponding indices, and for concreteness, \emph{let's assume for the moment that $h_+(i)$ and $h_+(j)$ are the corresponding points in $J$,} deferring a discussion of the other cases to the end of the proof. Note that since the rays $h_\pm$ are mutually homoclinic and are in tight position with respect to $m$, the velocity vectors $h_-'(i)$ and $h_-'(j)$ point to the same sides of $m$ as $h_+'(i)$ and $h_+'(j)$, respectively, and so all four vectors point to the same side. That is, \begin{enumerate}
\item[(a)]  $h_+'(i), h_+'(j), h_-'(i), h_-'(j)$ all point to the same side of $m$.
	\item[(b)] the segment $I_+ \subset J$ bounded by $h_+(i)$ and $h_+(j)$ contains no element $h_+(k)$ or $h_-(k)$ where $k$ is between $i$ and $j$.
\end{enumerate}

Let $I_-\subset m \setminus J$ be the segment that is bounded by the points $h_-(i)$ and $h_-(j)$. \emph{Suppose for a moment that we knew that $I_-$ had length less than $\delta$.}  Then for each $k$, it is impossible that \emph{both} $h_+(k)$ or $h_-(k)$ lie in $I_-$, as we're assuming that corresponding intersections of $h_\pm$ with $m$ stay at least $\epsilon>\delta$ apart. In particular, if $k$ is between $i,j$ and we apply the unlinking condition of Claim \ref{unlinked} twice, once to $i,k$ and once to $j,k$, we get from this and (b) above that \emph{neither} element $h_+(k)$ or $h_-(k)$ is contained in $I_-$. So, the four segments $h_+([i,j])$, $h_-([i,j])$, $I_+$ and $I_-$ have disjoint interiors, and we're in the situation of case (I) in the claim, as desired.

As constructed above, however, there is unfortunately no reason to believe that the interval $I_-$ has length less than $\delta$. To rectify this, recall that $\lambda$ actually has infinitely many non-boundary leaves $\ell^n$.  For each such $\ell^n$ and $p^n\in \ell^n \cap m$, we can repeat the above construction using constants $\delta^n\to 0$, producing points (say) $h_+(i^n),h_+(j^n)$ that lie within the length $\delta^n$-interval $J^n \ni p^n$  and that satisfy properties (a) and (b) above.  Choose the sequence $p^n \in \ell^n \cap m$ so that it is monotonic in the circular order induced on $m$, and let $\delta^n \to 0$ fast enough so that the associated intervals $I_+^n$ are all disjoint, so that in the circular order on $m$ we have
$$h_+(i^1) < h_+(j^1) < h_+(i^2) < h_+(j^2) < \cdots < h_+(i^n),$$
and where each $I_+^n$ is the interval $[h_+(i^n),h_+(j^n)]$, rather than the complementary arc on $m$ that has endpoints $h_+(i^n),h_+(j^n)$. Then Claim \ref{unlinked} implies that
$$h_-(i^1) > h_-(j^1) > h_-(i^2) > h_-(j^2) > \cdots > h_-(i^n).$$
Discarding finitely many $n$, we can assume all the point $h_+(i^n),h_+(j^n)$ lie in an interval $U\subset m$ of length less than $\delta$. Since the points $h_-(i^n),h_-(j^n)$ are at least $\epsilon>\delta$ away from the corresponding $+$ points, they all lie in $m\setminus U$. Then since the interval $I_-^n$ is defined to be disjoint from $I_+^n \subset U$,  we must have $I_-^n=[h_-(j^n), h_-(i^n)]$, rather than the other interval with those endpoints. It follows that at least for large $n$, all the intervals $I_-^n$ are disjoint. Since $m$ as compact, we can then pick some $n$ where $I_-^n$  has length less than $\delta$, as desired.  Therefore, we are in case (I) in the statement of the claim, and are done.

\medskip

In the argument above we have simplified the notation by assuming that we have points $h_+(i^n),h_+(j^n) \in J^n$ satisfying conditions (a) and (b), which put us in case (I) at the end. Up to exchanging $+,-$, the only other relevant case is when, our chosen points are $h_-(i^n),h_+(j^n) \in J^n$.  After passing to a subsequence in $n$, if we are not in the case already addressed, then we may assume that our chosen points are $h_-(i^n),h_+(j^n) \in J^n$ for all $n$. And after exchanging $+$ with $-$ and passing to a further subsequence, we may assume $$h_-(i^1) < h_+(j^1) < h_-(i^2) < h_+(j^2) < \cdots < h_-(i^1)$$
 in the circular order on $m$, and that the interval $I_+^n = [h_-(i^n), h_+(j^n)]$. Everything from then on works exactly as above: if we set $I_-^n $ to be the interval bounded by $h_+(i^n), h_-(j^n)$  that is disjoint from $ I_+^n$, then for some $n$ we have that the length of $I_-^n$ is less than $\delta$, and it is easy to verify that we are in case (II) of the claim.
\end{proof}

We now finish the proof of Theorem \ref{travelers}. Suppose we are in case (I) of Claim~\ref{affclaim}. Then the two simple closed curves $\gamma_\pm$ drawn on the left in Figure~\ref{surgery} are the projections to $M$ of the paths in $\tilde M$ obtained by concatenating the arcs $\tilde h_\pm|_{[i,j]}$ with lifts $\tilde I_\pm \subset \tilde m_j$ of the intervals $I_\pm \subset m$, see Figure \ref{liftsurgery}. We can homotope one path to the other in $\tilde M$ while preserving the fact that the endpoints are points on $\tilde m_i,\tilde m_j$ that differ by the unique deck transformation taking $\tilde m_i$ to $\tilde m_j$. So projecting down, the simple closed curves $\gamma_\pm$ are freely homotopic in $M$, and hence bound an annulus $A \hookrightarrow M$. See the left part of Figure \ref{liftsurgery}.

There is a uniform lower bound (depending on $\lambda,m$) for the angle at any intersection point of any leaf of $\lambda$ with $m$, and the points $h_\pm(i)$ are at least $\epsilon$ away from each other in $m$.  This implies that there is a uniform lower bound for the Hausdorff distance between $h_\pm|_{[i,j]}$ on $\partial M$. As long as the bound $\delta$ on the lengths of $I_\pm$ is small enough, the geodesics in the homotopy classes of $\gamma_\pm$ stay very close to $h_\pm|_{[i,j]}$, and are therefore distinct. So, the curves $\gamma_\pm$ are not homotopic in $\partial M$, and hence bound an \emph{essential} annulus $A \hookrightarrow M$.

\begin{figure}
\centering
\includegraphics{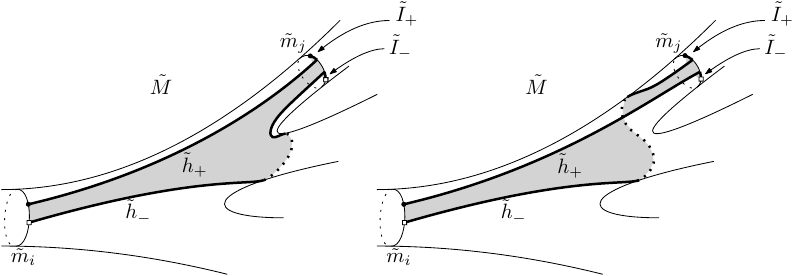}
\caption{On the left, the two paths drawn in heavy ink project to the two simple closed curves $\gamma_\pm$ in Claim \ref{affclaim} (I), shown on the left in Figure \ref{surgery}. The shaded region is a rectangle embedded in $\tilde M$ that projects to an embedded annulus $A \hookrightarrow M$ with boundary $\gamma_-\cup \gamma_+$. On the right, the union of the two paths projects to the simple closed curve $\gamma$ on $M$ of Claim \ref{affclaim} (II), and the shaded region projects to a M\"obius band $B \hookrightarrow M$ with boundary $\gamma$. }
\label{liftsurgery}
\end{figure}

Choosing $i,j$ to be large and $\delta$ to be very small, $\partial A$  is contained in $S(\lambda)$ and its intersection number with $\lambda $ is small.  Hence, $\lambda $ is an intrinsic limit of annuli, in the sense of Lemma \ref{annualilimit}, so we're done.

Case (II) is similar. Here, the single simple closed curve $\gamma$ described in Claim \ref{affclaim} (II)  bounds a M\"obius band $B \hookrightarrow M$, see the right side of Figure~\ref{liftsurgery}. Since $\partial M$ is orientable, $B$ is not boundary parallel, and hence by JSJ theory the   boundary of a regular neighborhood of $B$ is an essential annulus $A \hookrightarrow M$ whose boundary consists of two disjoint curves that are both homotopic to $\gamma$ on $\partial M$. As in case (I), we can make the intersection number of $\partial A$ with $\lambda$ arbitrarily small, so $\lambda $ is an intrinsic limit of annuli, and we are done.



 \begin {remark}
The proof (D) above is quite delicate.   Most of this delicacy comes from Claims \ref{case1} and \ref{affclaim}, which are needed to ensure that the annuli approximating $\lambda$  that are produced immediately afterward are \emph{embedded}.  But while we are able to prove these claims using arguments involving the planarity of the closure of $\partial \tilde M$ in $\BH^3 \cup \partial_\infty\BH^3$, one would not have to worry about these annuli being embedded if there was a strong `Annulus Theorem' guaranteeing that any essential singular annulus in an irreducible $3$-manifold $M$ can be surgered to give an essential embedded annulus. If this were true, Claims \ref{case1} and \ref{affclaim} could be replaced by a one paragraph compactness argument. Here, a \emph{singular annulus} is a map $f: (A,\partial A) \longrightarrow (M,\partial M)$ where $A = S^1 \times [0,1]$. We say $f$ is \emph{essential} if it is not  homotopic rel $\partial A$  into $\partial M$.

\begin {figure}[t]
\centering
\includegraphics{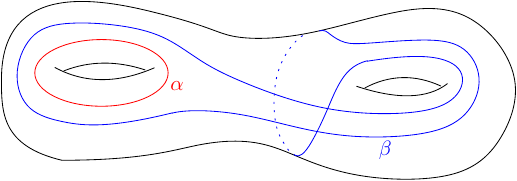}
\caption{The two curves $\alpha $ and $\beta$ are homotopic through the handlebody pictured, and therefore bound an essential singular  annulus.  The only annuli  one can produce from surgery are inessential, but one can surger to obtain  the  `obvious' disc  in the picture that separates the handlebody into two solid tori.}
\label {disc}	
\end {figure}

Such an annulus theorem follows from the JSJ decomposition when $M$ has incompressible boundary.  When $M$ has compressible boundary, there is a similar theorem as long as the original singular annulus has a spanning arc that is not homotopic rel $\partial$ into $\partial M$, see Cannon-Feustel \cite{cannon1976essential}. However, our proof above does not provide such annuli, and indeed such annuli do not exist in compression bodies (the $M$ of most interest to us), since \emph{any} proper arc in a compressionbody $M$ is homotopic rel $\partial$ into $\partial M$. 

In fact, in a general $M$, one \emph{cannot} always surger essential singular annuli to produce embedded essential annuli. For instance, the two curves in Figure~\ref{disc} bound an essential singular annulus that cannot be surgered to give an embedded essential annulus. However, in that example, one \emph{can} surger to get a meridian in the handlebody, so maybe an essential singular annulus can always be surgered to give either a meridian or an essential embedded annulus? This also turns out not to be true. Suppose $P$  is a pair of pants and let $M=P\times [0,1]$, which is homeomorphic to a genus two handlebody. If $\gamma \longrightarrow P$ is an immersed figure-$8$  whose image forms a spine of $P$, then  the singular annulus $\gamma \times [0,1] \longrightarrow P \times [0,1]$  is essential, but the three embedded annuli that one can obtain from it  by surgery are all inessential, and no meridian can be created by surgery either. However,  we expect this is the \emph{only} counterexample. The first author of this paper has spent considerable time trying to prove this with a tower argument,  but pushing  down the tower is very  subtle, since if the obvious constructions  fail, one has to characterize the figure $8$ example.
\end {remark}

\section{Hausdorff limits of meridians}
\label {cassonsec}
Let $M$ be an orientable hyperbolizable compact $3$-manifold, equip $\partial_{\chi<0}M $ with a hyperbolic metric. In \cite[Theorem B.1]{Lecuireplissage}, Lecuire showed that every lamination $\lambda$ on $\partial_{\chi<0}M$ that is a Hausdorff limit of meridians contains a homoclinic leaf that is a homoclinic geodesic. The converse is not true:  certainly in order to be a Hausdorff limit of meridians $\mu$ needs to be connected, and as explained in Figure \ref{hlim} in the introduction there are even connected laminations that contain homoclinic leaves but are not  Hausdorff limits of meridians.

We say that two laminations $\mu_1,\mu_2$ are \emph {commensurable} if they contain a common sublamination $\nu$ such that  for both $i$,  the difference $\mu_i \setminus \nu$  is the union of finitely many leaves.   $\mu_1$ and $\mu_2$ are \emph {strongly commensurable} if they contain a common $\nu$ such that for both $i$, the difference $\mu_i \setminus \nu$  is the union of finitely many leaves, none of which are simple closed curves.

\begin{theorem}[Hausdorff limits of meridians]\label{hlimits}
Suppose that $S \subset \partial_{\chi<0} M$ is a connected subsurface with geodesic boundary, such that $\partial S$ is  incompressible, and that the disc set $\mathcal D(S,M)$ is `large', i.e.\ it has infinite diameter in the curve graph $C(S)$.  Let $\lambda$ be a geodesic lamination in $int(S)$ that is a finite union of minimal laminations, and assume that the following does \emph{not} hold:
\begin{itemize}
\item[$(\star)$] $S$ is a closed, genus two surface,  there is a separating meridian $\mu$ that does not intersect $\lambda$ transversely, and $\lambda$ intersects transversely the two nonseparating meridians disjoint from $\mu$.
\end{itemize} Then $\lambda$ is strongly commensurable to a Hausdorff limit of meridians in $S$ if and only if $\lambda$ is strongly commensurable to a lamination containing a homoclinic leaf, and this happens if and only if one of the following holds:
\begin{enumerate}
\item $\lambda$ is disjoint from a meridian on $S$,
\item some component of $\lambda $ is an intrinsic limit of meridians, or
\item there is an essential (possibly nontrivial) interval bundle $B \subset M$ over a compact surface $Y$ that is not an annulus or M\"obius band, and there are components $\lambda_\pm \subset \lambda$ that each fill a component of $\partial_h B$ (possibly the same component, if $\partial_h B$ is connected), such that $\lambda_\pm$ are essentially homotopic through $B$, as in \S \ref{laminterval}, and there is a compression arc $\alpha$ for $B$ that is disjoint from $\lambda $.
\end{enumerate}
\end{theorem}

Recall from Proposition \ref{discsetcurve} that when $\mathcal D(S,M)$ does not have infinite diameter in $C(S)$, it is either empty, consists of a single separating meridian, or consists of a single non-separating meridian $m$ and all separating curves that are band sums of $m$. In these cases, it is obvious what the Hausdorff limits of meridians are. For instance, in the last case a finite union $\lambda$ of minimal laminations in $S$ is strongly commensurable to a Hausdorff limit of meridians if and only if either $\lambda=m$ or $\lambda \subset S \setminus m$. For the `if'  direction, note that if $\lambda \subset S \setminus m$ then it can be approximated by an arc with endpoints on opposite sides of $m$, and doing a band sum with $m$ gives a curve that approximates $\lambda$. For the `only if' direction, just note that all meridians are either equal to $m$ or are contained in $S \setminus m$.

The case $(\star)$ above really is exceptional. In that case, (1) holds, and at least when $\mu \subset \lambda$ we have that $\lambda$ contains a homoclinic leaf, but  $\lambda$ is \emph{not} commensurable to a Hausdorff limit of meridians. Indeed, let $T_\pm\subset S \setminus \mu$  be the two components of $S \setminus \mu$ and hoping for a contradiction, take a sequence of meridians $(m_n)$ that Hausdorff converges to $\lambda'\supset \lambda$. We can assume after passing to a subsequence that $m_n$ has an $\mu$-wave in $T_+$ (say) for all $n$. Since $T_+$ is a compressible punctured torus, there is a \emph{unique} homotopy class rel $\mu$ of $\mu$-wave in $T_+$, so $\lambda'$ contains a leaf $\ell$ that either intersects $T_+$ in an arc in this homotopy class, or is contained in $T_+$ and is obtained by spinning an arc in this homotopy class around $\mu$. But then $\ell$ intersects nontrivially every nonperipheral minimal lamination in $T_+$ other than the unique nonseparating meridian $\mu_+$ of $T_+$, so $\lambda$ is disjoint from $\mu_+$, contrary to assumption.

\subsection{The proof of Theorem \ref{hlimits}}

Most of the proof of Theorem \ref{hlimits} is contained in the following results. Assume that $S\subset \partial_{\chi<0} M$ is a connected subsurface with geodesic boundary, $\partial S$ is incompressible, and $\mathcal D(S,M)$ is large.

\begin{lemma}\label{promotinghlimitslem}
Suppose $\lambda \subset S$ is a lamination, there is a meridian $\mu$ that does not intersect $\lambda$ transversely, and that if $S$ is a closed surface of genus $2$ then $\mu$ is nonseparating. Then $\lambda$ is strongly commensurable to a Hausdorff limits of meridians on $S$.
\end{lemma}

The proof of Lemma \ref{promotinghlimitslem} uses some ideas that the first author developed with Sebastian Hensel, whom we thank for his contribution.

 \begin{proof}
 We may assume that $\lambda$ is a finite union of minimal components. It suffices to assume $\mu$ is not a leaf of $\lambda$, as long as we prove the conclusion both for such a $\lambda$ and for $\lambda\cup \mu$.

\medskip

Assume first that $\mu$ is nonseparating in $S$. Let $(c_i)$ be a sequence of simple closed curves on $S$ that Hausdorff-converges to $\lambda$. One can do this by constructing for each component $\lambda_0 \subset \lambda$ a simple closed geodesic approximating $\lambda_0$, by taking an arc that runs along a leaf of $\lambda_0$ for a long time, and then closing it up the next time it passes closest to its initial endpoint in the correct direction. Let $\alpha$ be a simple closed curve on $S$ that intersects $\mu$ once, and intersects all the components of $\lambda$. For each $k$, let $\gamma_i^k$ be the geodesic homotopic to the `band sum' of $\mu$ and $T_{c_i}^k(\alpha)$, where $T_{c_i}$ is the twist around $c_i$ and a band sum of two curves intersecting once is the boundary of a regular neighborhood of their union. Note that $\gamma_i^k$ is a meridian for all $i,k$. If $(k_i)$ is a sequence that increases quickly enough, $(\gamma_i^{k_i})$ converges to a lamination strongly commensurable to $\lambda$. And if we pick a meridian $\beta$ on $S$ that intersects both $\mu$ and $\lambda$, then $T_\mu^i \circ T_{\gamma_i^{k_i}}^i(\beta)$ Hausdorff converges to a lamination strongly commensurable to $\lambda\cup \mu$. 

Now suppose $\mu$ is separating. We claim that there is another separating meridian in $S$ that is disjoint from $\mu$. Let $m$ be a maximal multicurve of meridians in $S$ that contains $\mu$ as a component. Since $\mathcal D(S,M)$ is large, $m\neq \mu$. If $m$ has a separating component distinct from $\mu$, we are done. So, suppose we have a nonseparating component $m_0 \subset m$. We can make a (separating) band sum of $m_0$ that is disjoint from $\mu$ unless $m_0$ lies in a punctured  torus component of $S \setminus \mu$. So, we assume the latter is true. Since $\mathcal D(S,M)$ is large, it cannot be that $m=\mu \cup m_0$, since then all meridians are disjoint from $m_0$. So, there is another component $m_1$ of $m$, which we can assume is also nonseparating. This $m_1$ must lie on the opposite side of $\mu$ from $m_0$, and as before we're done unless the component of $S\setminus \mu$ containing $m_1$ is also a punctured torus. But in this case, $S$ is a genus two surface contrary to assumption.

Let $T \subset S \setminus \mu$ be a component that contains a nonperipheral separating meridian, which we  call $\mu'$. Let $V$ be the other component. Write $\lambda = \lambda_T \cup \lambda_{V}$, where $\lambda_T\subset T$ and $\lambda_{V} \subset S \setminus T$. Let $C$ be the compression body with exterior boundary equal to the component of $\partial M$ that contains $S$, that one obtains by compressing the meridian $\mu$. 

We claim that there are sequences of simple closed curves $(\alpha_i)$, $(\beta_i)$ in $T$ such that the following two properties hold:
\begin{itemize}
\item $(\alpha_i)$ and $(\beta_i)$ both Hausdorff converge to { a geodesic lamination strongly commensurable to} $\lambda_T$,
\item for all $i$, $\alpha_i$ and $\beta_i$ bound an essential annulus in $C$.
\end{itemize}
To construct these sequences, start by picking a simple closed curve $\alpha$ in $T$ such that $\alpha$ and each component of $\lambda_T$ together fill $T$. Let $\beta$ be a simple closed curve on  $T$ such that $\alpha, \beta,\mu$ bound a pair of pants in $T$.  In $C$, we can compress the boundary component $\mu$ of this pair of pants, so $\alpha,\beta$ bound an annulus in $C$. Moreover, this annulus is essential, since otherwise $\alpha,\beta$ bound an annulus in $T$, implying $T$ is torus with the one boundary component $\mu$, contradicting the fact that there is a separating nonperipheral meridian in $T$. Then find a sequence $(c_i)$ of simple closed curves in $T$ that Hausdorff converge to { a geodesic lamination strongly commensurable to} $\lambda_T$, take $k_i$ to be a fast increasing sequence and set $\alpha_i=T_{c_i}^{k_i}(\alpha)$ and $\beta_i:=T_{c_i}^{k_i}(\beta)$. Since $\alpha $ fills with every component of $\lambda_T$, the curve $\beta$ intersects every component of $\lambda_T$. It follows that  $(\alpha_i)$ and $(\beta_i)$ Hausdorff converge to { a geodesic lamination strongly commensurable to} $\lambda_T$. And since each $c_i$ is nonperipheral in $T$, each component of $c_i$ bounds an annulus in $C$ with a curve on the interior boundary of $C$, so the twist $T_{c_i}$ extends to $C$, implying that $\alpha_i,\beta_i$ bound an annulus in $C$ as desired above.

Now let $C'$ be the compression body obtained by compressing both $\mu$ and $\mu'$, so $C \subset C' \subset M$. Note that since both curves are separating and are disjoint, Proposition \ref{discsetcurve} says that $\mathcal D(S,C')$ is large, so we can pick a meridian $m\in \mathcal D(S,C')$ that intersects $\mu$ and every component of $\lambda$. Fix a sequence of geodesic muticurves $(d_i)$ in $V$ that Hausdorff converges to $\lambda_{V}$. As with the twists $T_{c_i}$ in $C$, the twists $T_{d_i}$ extend to $C'$. And the compositions $T_{\alpha_i} \circ T_{\beta_i}^{-1}$ extend to $C'$ because the curves bound annuli in $C \subset C'$. We then define $\gamma_i:=(T_{\alpha_i}\circ T_{\beta_i}^{-1})^{k_i} \circ T_{d_i}^{k_i}(m)$ for some fast increasing $k_i\to \infty$. These $\gamma_i$ are all meridians and Hausdorff converge to a lamination strongly commensurable to $\lambda$. To obtain $\lambda\cup \mu$ instead, hit $\gamma_i$ with high powers of twists around $\mu$.
 \end{proof}

Here is a more powerful version of Lemma \ref{promotinghlimitslem}. The idea of the proof is more or less the same, but more complicated.

\begin{prop}[Promoting Hausdorff limits] \label{promotinghlimits}
 Suppose that $\nu,\eta$ are disjoint geodesic laminations on $S$ that are finite unions of minimal components, and set $\lambda=\nu\cup \eta$. Suppose also that no component of $\nu$ is a meridian.

Let $X$ be the union of the subsurfaces with geodesic boundary that are filled by the components of $\nu$. Suppose that there are disjoint, nonhomotopic meridians $\mu,\mu'$ on $S$ that are disjoint from $\eta$, and a sequence of homeomorphisms $$f_i : S \longrightarrow S, \ \ f_i|_{S \setminus int(X)}=id$$ such that $\mu_i:=f_i(\mu)$ and $\mu_i':=f_i(\mu')$ are both sequences of meridians that Hausdorff converge to laminations strongly commensurable to $\nu$. Then $\lambda=\nu \cup \eta$ is strongly commensurable to a Hausdorff limit of meridians in $S$.
\end{prop}
 
Before proving the proposition, we record the following application.

\begin{corollary}\label{hauslimitcor}
Suppose that $\lambda$ is a geodesic lamination on $S$ that is a finite union of minimal components. If either
\begin{itemize}
	\item some component $\nu \subset \lambda$ that is not a simple closed curve is an intrinsic limit of meridians,
	\item there are (possibly equal) components $\lambda_\pm \subset \lambda$, neither of which is a simple closed curve, and where each fills a component of the horizontal boundary (possibly the same component if $\partial_h B$ is connected) of an essential interval bundle $$(B,\partial_H B) \hookrightarrow (M,S),$$ where $\lambda_\pm$ are essentially homotopic through $B$, and where there is a compression arc $\alpha$ for $B$ that is disjoint from $\lambda$, 
\end{itemize}
then $\lambda$ is strongly commensurable to a Hausdorff limit of meridians.
\end{corollary}
\begin{proof}
Suppose some component $\nu \subset \lambda$ that is not a simple closed curve is an intrinsic limit of meridians. Setting $X:=S(\lambda)$ we can take $(\mu_i)$ to be any sequence of meridians in $X$ that Hausdorff converges to a lamination strongly commensurable to $\nu$. Moreover, since $\nu$ fills $X$ and is a limit of meridians, the disc set $\mathcal D(X,M)$ is large, so for each $i$ there is some meridian $\mu_i'$ disjoint from $\mu_i$. Since there are only finitely many topological types of pairs of disjoint curves in $X$ up to the pure mapping class group of $X$, after passing to a subsequence we can assume that all $\mu_i,\mu_i'$ are of the form in the proposition. The desired conclusion follows.

In the second case, we let $X$ be the subsurface with geodesic boundary obtained by tightening $\partial_H B$ and set $\nu=\lambda_-\cup \lambda_+$. Write the interval bundle as $\pi: B \longrightarrow Y$, where $Y$ is a compact surface with boundary. We can assume without loss of generality that $\alpha$ is a strict compression arc, i.e.\ that it is homotopic rel endpoints to a fiber $\pi^{-1}(y),y\in \partial Y$. Note that since $\lambda_\pm$ are not simple closed curves, $Y$ is not an annulus or M\"obius band. 

Since $\lambda_\pm$ are essentially homotopic through $B$, Fact \ref{homotopy involutoin} says that if our reference hyperbolic metrics are chosen appropriately, we have that $\lambda_-\cup \lambda_+= (\pi|_{\partial_HB})^{-1}(\bar \lambda)$ for some geodesic lamination $\bar \lambda$ on $Y$. Since $\lambda_\pm$ together fill $\partial_HB$, the lamination $\bar \lambda$ is minimal and fills $Y$. So in particular, it has no closed, one-sided leaves, and therefore if we pick a nonzero transverse measure on $\bar \lambda$, we have that $\bar \lambda$ is the projective limit of a sequence of two-sided nonperipheral simple closed curves $(c_i)$ in $Y$, by Theorem 1.2 of \cite{erlandsson2021mapping}. Homotope the $c_i$ on $Y$ to based simple loops at $y \in \partial Y$, let $m(c_i)$ be the associated compressible curves on $S$ constructed in Claim \ref{mc}, and let $\mu_i$ be the geodesic meridians on $S$ in their homotopy classes. Then $(\mu_i)$ Hausdorff converges to a lamination strongly commensurable to $\lambda_-\cup\lambda_+$. After passing to a subsequence, we can assume that all the $c_i$ differ by pure homeomorphisms of $Y$, in which case the meridians $\mu_i$ are as required in Proposition \ref{promotinghlimits}, for some $\mu,f_i$.  Note that since our compression arc is assumed to be disjoint from $\lambda$, all the $\mu_i$ are disjoint from $\eta := \lambda \setminus \lambda_\pm$, and hence so is our $\mu$. We create disjoint meridians $\mu_i'$ similarly, by taking some $c_i'$ on $Y$ disjoint from $c_i$, and letting $\mu_i'$ be the geodesic meridian homotopic to $m(c_i')$. It then follows from Proposition \ref{promotinghlimits} that $\lambda$  is strongly commensurable to a Hausdorff limit of meridians as desired.
\end{proof}

We now prove the proposition.

\begin{proof}[Proof of Proposition \ref{promotinghlimits}]
Assume that $\mu,\mu'$ are disjoint meridian on $S$ that are disjoint from $\eta$, that $f_i : S \longrightarrow S$ are homeomorphisms that are the identity outside of $X$, and that $\mu_i:=f_i(\mu)$ and $\mu_i':=f_i(\mu')$ are sequences of meridians that Hausdorff converge to laminations strongly commensurable to $\nu$.

We want to show that $\nu \cup \eta$ is strongly commensurable to a Hausdorff limit of meridians on $S$. We now basically repeat the argument in Lemma \ref{promotinghlimitslem}, so the reader should make sure that they understand that argument before continuing here.

\medskip

Suppose $\mu$ (say) is nonseparating in $S$.
Choose a simple closed curve $\alpha$ on $S$ that intersects $\mu$ once and intersects essentially each component of $\eta$.  Let $\alpha_i := f_i(\alpha)$, and note that $\alpha_i$ intersects $\mu_i$ once, and also intersects essentially each component of $\eta$. Let $(c_i)$ be a geodesic multi-curve that Hausdorff converges to $\eta$, and let $\gamma_i^k$ be the geodesic homotopic to the `band sum' 
\begin{equation}
B(\mu_i,T_{c_{i}}^k(\alpha_i)) = T_{c_i}^k(B(\mu_i,\alpha_i)) = T_{c_i}^k \circ f_i( B(\mu,\alpha)),\label{bands!}
\end{equation} where here $B(\cdot,\cdot)$ takes in two simple closed curves that intersect once and returns the boundary of the regular neighborhood of their union. If one of the inputs in a band sum is a meridian, then so is the output, so $\gamma_i^k$ is a meridian for all $i,k$. The given equalities  are true at least for large $i$. The first equality holds because $\mu$ is disjoint from $\eta$, $f_i=id $ on the subsurfaces filled by the components of $\eta$, and therefore $\mu_i$ is disjoint from $c_i$ for large $i$. The second equality is obvious from the definitions of $\mu_i,\alpha_i$.

Let $(k_i)$ be a fast increasing sequence. After passing to a subsequence, we can assume that $(\gamma_i^{k_i})$ Hausdorff converges to a lamination $\lambda$. \it We claim that $\lambda$ is strongly commensurable to $\nu \cup \eta$. \rm

First, using the second term in \eqref{bands!}, if $k_i$ is huge with respect to $i$, then $c_i$ is contained in a small neighborhood of $\gamma_i^{k_i}$, and so since $(c_i)$ Hausdorff converges to $\eta$, \it we have $\lambda \supset \eta$.\rm

We claim that each $\gamma_i^{k_i}$ essentially intersects each component $X_0 \subset X$. If not, then from the third term in \eqref{bands!} it follows that $B(\mu,\alpha)$ is disjoint from $X_0$. But $\mu$ essentially intersects $X_0$, since otherwise the Hausdorff limit of the $\mu_i$ will not contain the associated component $\nu_0 \subset \nu$. So, $\mu,\alpha$ and $X_0$ all lie in the punctured torus $T\subset S$ bounded by $B(\mu,\alpha)$. But since $\alpha$ intersects every component of $\eta$, we have that $\eta$ intersects $T$ as well, in a collection of arcs disjoint from $\mu$. Since $X_0$ is disjoint from $\eta$, $X_0=\mu$, so $\nu_0=\mu$ is a meridian, contrary to our standing assumption.

It now follows that the Hausdorff limit $\lambda $ essentially intersects each component of $X$. Since $\gamma_i^{k_i}$ is disjoint from $\mu_i$ and $(\mu_i)$ Hausdorff converges to a lamination containing $\nu$, The laminations $\lambda,\nu$ cannot intersect transversely. Since  each component of $X$ is filled by a component of $\nu$, \it we have $\lambda \supset \nu$.\rm

Finally, if $Y$ is the union of all the subsurfaces with geodesic boundary that are filled by the components of $\eta$, then as $f_i=id$ outside $X$ and $c_i \subset Y$ for large $i$, the intersection of $\gamma_i^{k_i}$ with $S \setminus (X \cup Y)$ is properly homotopic to the intersection of  $B(\mu,\alpha)$, which is independent of $i$. It follows that \it $\lambda \setminus (\nu \cup \eta)$ is a finite collection of non-closed leaves, \rm so we are done.

\medskip 

We can now assume that both $\mu,\mu'$ are separating, so that $\mu_i,\mu_i'$ are also separating for all $i$. Let $T_i \subset S \setminus \mu_i$ be the component containing $\mu_i'$, and let $V_i $ be the other component. Note that $T_i$ is not a punctured torus, since it contains a nonperipheral separating curve. Since $\partial T_i \cap \eta =\emptyset$, we have $$\eta = \eta_{T} \sqcup \eta_{V},$$ 
where the first term is the intersection of $\eta$ with $T_i$, and the second term is defined similarly. Note that since $f_i=id$ on $S\setminus X$, all the $\mu_i$ induce the same two-element partition of the components of $S \setminus X$, so at least after passing to a subsequence the decomposition of $\eta$ above is actually independent of $i$, which is why we have omitted the $i$ in the notation. 

Let $C$ be the compression body whose exterior boundary is the component of $\partial M$ containing $S$, and which is obtained by compressing the curve $\mu$. Let $C'$ be similarly obtained by compressing both $\mu$ and $\mu'$, so that $C \subset C' \subset M$. Since $C'$ admits two nonhomotopic disjoint separating meridians, the disc set $\mathcal D(S,C)$ is large by Proposition \ref{discsetcurve}, so we can pick a meridian $m\in \mathcal D(S,C)$ that intersects every component of $\nu\cup \eta$, as well as $\mu,\mu'$. Let $C_i\subset C_i' \subset M$ be the compression bodies obtained by compressing $\mu_i,\mu_i'$. Then $f_i$ extends to a map $C' \longrightarrow C_i'$, implying that $m_i:=f_i(m)$ is a meridian in $C_i'$.

As in the proof of Lemma \ref{promotinghlimitslem}, we can pick sequences $(\alpha_i),(\beta_i)$ of simple closed curves in $T_i$ such that $(\alpha_i)$ and $(\beta_i)$ both Hausdorff converge to $\eta_T$, and where $\alpha_i,\beta_i$ bound an essential annulus in $C_i$ for all $i$. As in the lemma, $T_{\alpha_i} \circ T_{\beta_i}^{-1}$ extends to $C_i'$. Let $(c_i)$ be a sequence of multicurves in $V_i$ that Hausdorff converges to $\eta_V$. Each component of $c_i$ bounds an annulus in $C_i'$ with a curve on the interior boundary of $C_i'$, and hence the multitwist $T_{c_i}$ extends to a homeomorphism of $C_i'$.  For any given $k$, set $$\gamma_i^{k} := (T_{\alpha_{k}} \circ T_{\beta_{k}}^{-1})^k \circ T_{c_{k}}^k(m_i).$$

We claim that for fast increasing $k_i$, the curves $\gamma_i^{k_i}$ Hausdorff converge to a lamination that is strongly commensurable to $\nu \cup \eta$ as desired. This is proved using the same types of arguments we employed in the nonseparating case above. In particular, recall that $m$ was selected to intersect all components of $\nu \cup \eta$. Since $f_i$ is supported on subsurfaces filled by components of $\nu$, all the $m_i=f_i(m)$ intersect all components of $\nu \cup \eta$, and hence for large $k_i$ they intersect $\alpha_{k_i},\beta_{k_i}$. So, $\gamma_i^{k_i}$ is twisted many times around $\alpha_{k_i},\beta_{k_i}$, and hence its Hausdorff limit contains $\nu$. Similarly, the $m_i$ intersect $c_{k_i}$ for large $i$. Since $c_k$ lies in $V_k$, it is disjoint from $\alpha_k\subset T_k$ and $\beta_k\subset T_k$,  and thus the Hausdorff limit of $\gamma_i^{k_i}$ contains $\eta$. Finally, the Hausdorff limit has no other minimal components because $\alpha_k,\beta_k,c_k$ are contained in subsurfaces filled by components of $\nu\cup \eta$, and $m_i=f_i(m)$ is constant outside this subsurfaces.
 \end{proof}

We can now start the proof of the theorem.

\begin{proof}[Proof of Theorem \ref{hlimits}]
Suppose that $\lambda \subset S$ is a lamination and $(\star)$ does not hold, so that it is not the case that $S$ is a genus two surface and $\lambda$ is disjoint from a separating meridian $\mu$, but intersects the two nonseparating meridians disjoint from $\mu$. We want to show that $\lambda$ is strongly commensurable to a Hausdorff limit of meridians if and only if it is strongly commensurable to a lamination containing a homoclinic leaf, which happens if and only if either

\begin{enumerate}
\item $\lambda$ is disjoint from a meridian,
\item some component of $\lambda$ is an intrinsic limit of meridians, or 
\item there is an essential (possibly nontrivial) interval bundle $B \subset M$ over a compact surface $Y$ that is not an annulus or M\"obius band, and there are components $\lambda_\pm \subset \lambda$ that each fill a component of $\partial_H B$, such that $\lambda_\pm$ are essentially homotopic through $B$, as in \S \ref{laminterval}, and there is a compression arc $\alpha$ for $B$ that is disjoint from $\lambda $.
\end{enumerate}

\medskip

\noindent \bf Hausdorff limit $\implies$ homoclinic leaf. \rm  Suppose first that $\lambda$ is strongly commensurable to a Hausdorff limit of meridians $\lambda'$. Then by \cite[Theorem B.1]{Lecuireplissage}, there is a homoclinic leaf $h\subset \lambda'$, so $\lambda$ is strongly commensurable to a lamination with a homoclinic leaf as desired. 

\medskip

\noindent \bf Homoclinic leaf $\implies $ (1), (2) or (3). \rm We now assume we have a homoclinic leaf $h$ in some lamination strongly commensurable to $\lambda$.

The two ends of $h$ limit onto (possibly equal) components $\lambda_\pm\subset \lambda$. If one of $S(\lambda_\pm)$ has compressible boundary, there is a meridian disjoint from $\lambda$, so we are in case (1) and are done. So, $\partial S(\lambda_\pm)$ is incompressible, and we're in the situation of Theorem \ref{travelers} and Corollary \ref{travelers-biinfinite}. We now break into cases. 

If one of $\lambda_\pm$ is an intrinsic limit of meridians, we're in case (2) and are done. If we're in case (3) of Theorem \ref{travelers} and Corollary \ref{travelers-biinfinite}, we're in case (3) of the theorem and are done, unless the given interval bundle $B\longrightarrow Y$ is over an annulus or M\"obius band. But in that case, letting $c$ be a boundary component of $Y$, we can consider the geodesic meridian $\mu$ on $S$ homotopic to the $m(c)$ constructed in Claim \ref{mc}, using the compressing arc given by Corollary \ref{travelers-biinfinite}. This $\mu$ is disjoint from $\lambda$, so we're in case (1) of the theorem.

Finally suppose that the two ends of $h$ are asymptotic on $S$, so that $\lambda_-=\lambda_+$. Let's separate further into the cases (i) and (ii) in Corollary \ref{travelers-biinfinite}. In case (i), using the notation of the corollary, the curve $c\cup h([-s.s])$ is a meridian disjoint from $\lambda$. So, we're in case (1) of the theorem. In case (ii), let  $T$ be a neighborhood of $h\cup \lambda_\pm$ that is either a punctured torus or a pair of pants, depending on whether the two ends of $h$ limit onto opposite sides of $\lambda_\pm$, or onto the same side. Because we're in case (ii), there is a meridian in $T$. Hence, one of the boundary components of $T$ is a meridian, and is disjoint from $\lambda$ so we're done.

\medskip

\noindent \bf (1), (2) or (3) $\implies$ Hausdorff limit. \rm Suppose (1), (2) or (3) holds. We want to show $\lambda$ is strongly commensurable to a Hausdorff limit of meridians. If $\lambda$ is disjoint from a meridian, then we're done by Lemma~\ref{promotinghlimitslem}. If a component of $\lambda$ is an intrinsic limit of meridians, we're done by the first part of Corollary~\ref{hauslimitcor}. In case (3) above, we're done by the second part of Corollary~\ref{hauslimitcor}.
 \end{proof}

\section{Extending partial pseudo-Anosovs to compression bodies}

Let $M $ be a compression body with exterior boundary $\Sigma$. Let $S\subset \Sigma$ be an essential subsurface such that $\partial S$ is incompressible. In this section, we prove:

\begin{theorem}[Extending partial pseudo-Anosovs]\label{BJM}
	Suppose that $f : \Sigma \longrightarrow \Sigma$ is a partial pseudo-Anosov supported on $S$. Then $f$ has a power that extends to a nontrivial subcompressionbody of $(M,S)$ if and only if the attracting lamination of $f$ is a projective limit of meridians that lie in $S$.
\end{theorem}

When $S=\Sigma$, this is a theorem of Biringer-Johnson-Minsky \cite{Biringerextending}. The proof of Theorem \ref{BJM} is basically the same as their proof, but we need to go through it anyway, to note the places that parabolics appear, and to deal with the fact that we are looking at subcompression bodies of $(M,S)$ rather than of $M$. Also, before starting on the bulk of the proof in \S \ref{BJMpf}, we isolate part of the argument into a separate purely topological subsection, \S \ref{sec:topologicaliteration}. This separation of the argument into distinct topological and geometric parts makes it more understandable than the original version, we think.

\subsection{Dynamics on the space of marked compression bodies}

\label {sec:topologicaliteration}
Let $\Sigma$ be a closed, orientable surface, and let $S\subset \Sigma$ be an essential subsurface. The \emph {space of marked $S$-compression bodies} is defined to be 
$$\mathrm {CBod}(S) = \{  (C,h:\Sigma \to \partial_+ C) \}/\sim, $$
where here $C$ is a compression body, $h$ is a homeomorphism, and
\begin{itemize}
	\item the multicurve $h(\partial S)\subset \partial_+C$ is incompressible,
	\item there is a multicurve on $S$ such that $h(S)$ is a cut system for $C$, i.e.\ $h(S)$ bounds a collection of disks that cut $C$ into balls and trivial interval bundles over the interior boundary components.
\end{itemize} We declare $(C_i,h_i:S \to \partial_+ C_i)$, $i=1,2$ to be equivalent (written $\sim$ as above) if there is a homeomorphism $\phi : C_1\to C_2$ that respects the boundary markings: that is, $\phi \circ h_1$ and $h_2$ are homotopic maps $S \to \partial_+ C_2$.

We write $(C_1,h_1) \subset (C_2,h_2)$ if there is an embedding $\phi: (C_1,\partial_+ C_1) \into (C_2,\partial_+ C_2)$ that respects the boundary markings. This gives a partial ordering on $\mathrm {CBod}(S)$.  We often identify $\Sigma$ with $\partial_+ C$ instead of specifying the boundary marking, and simply write $C$ for an element of $\mathrm {CBod}(S)$.  So $\mathrm{CBod}(S)$ is the set of all compression bodies with exterior boundary $\Sigma$ that one obtains by compressing curves in $S$ (without compressing boundary curves) up to the obvious equivalence.

A marked $S$-compression body $(C,h)$ has a \emph{disk set} $\mathcal D(C) \subset \mathcal C(S)$, where a simple closed curve $\gamma\in \mathcal C(S)$ lies in the disc set if $h(\gamma)$ is a meridian in $C$. In fact, the disk set $\mathcal D(C)$ determines $(C,h)$ up to equivalence, say by an argument similar to the last paragraph of the proof of Fact \ref{charcomp}. The set $\mathrm {CBod}(S) $ can then be identified with the `set of all disk sets' in $\mathcal C (S)$.  It then inherits a topology as a subset of the power set $\mathcal P( \mathcal C (S))$, wherein $D_n \to D$ if and only if for every $c\in \mathcal C(S)$, we have either $c\in D$ and $c\in D_n$ for all large $n$, or $c\not \in D$ and $c\not \in D_n$ for all large $n$.

\begin {lemma}\label {convergecontain}
If $C_n \to  C $ in $\mathrm {CBod}(S)$, then $C \subset C_n$ for large $n $.
\end {lemma}
\begin {proof}
Suppose that $C $ is obtained by compressing a finite set $\Gamma$ of disjoint simple closed curves on $S $.  For large $n $, we have $\Gamma \subset \mathcal D (C_n) $, so $\mathcal C \subset C_n $.
\end {proof}

\begin {lemma}
$\mathrm {CBod}(S) $ is compact.
\end {lemma}
\begin {proof}
As $\mathcal P( \mathcal C (S))$ is compact, we want to show that $\mathrm {CBod}(S)$ is closed. Suppose $C_n$ is a sequence of marked compression bodies with disk sets $$D_n = \mathcal D(C_n) \subset \mathcal C(S),$$ and that $D_n \to D \subset  \mathcal C(S)$. Let $\Gamma$ be a maximal set of disjoint, pairwise non-homotopic elements of $D$. Compressing $\Gamma$  yields a marked compression body $C$. Since $\Gamma$ is finite,  $\Gamma \subset D_n$ for large $n$, so $\mathcal D (C)  \subset D_n$.  Thus, $\mathcal D (C)  \subset D$.  

It therefore suffices to show $D \subset \mathcal D(C)$. Suppose this is not the case, and pick $\beta \in D\setminus \mathcal D(C)$ such that the intersection number of $\beta$ and $\Gamma$ is minimal. By maximality of $\Gamma$, this intersection number cannot be zero. Since $\beta\in D$, if $n $ is large we have $\beta\in \mathcal D(C_n)$. By an outermost disk argument, if $\gamma\in \Gamma$ is a component that intersects $\beta$, there is an arc $c\subset \gamma$ with endpoints on $\beta$ and interior disjoint from $\beta$, that is homotopic rel endpoints in $C_n$ to the two arcs $b_1,b_2\subset \beta$ into which $\beta$ is cut by $\partial c$. Passing to a subsequence, we can assume that $c,b_1,b_2$ are independent of $n$. Then $c \cup b_1$ and $c \cup b_2$ are both meridians in $C_n$ for all large $n$, and hence lie in $D$. Since they intersect $\Gamma$ fewer times than $\beta$ does, they lie in $\mathcal D(C)$. But then $\beta$ (which is a band sum of the two curves) also lies in $\mathcal D(C)$, a contradiction.
\end {proof}

Let $f : \Sigma \longrightarrow \Sigma$ be a homeomorphism with $f=id$ on $\Sigma\setminus S$. Then $f$ acts  on $\mathrm {CBod}(S)$ by $f \cdot (C,h) = (C, h \circ f^{-1})$.  When we regard marked $S$-compression bodies as compression bodies with exterior boundary \emph{equal} to $\Sigma$, we'll just write $C$ and $f(C)$ for a marked compression body and its image. Note that $f(C)=C$ if and only if $f$ extends to a homeomorphism of $C$.

Fixing $M \in \mathrm {CBod}(S)$ and $f$ as above, let
$\mathcal A$ 
be the set of accumulation points in $\mathrm {CBod}(S)$ of the $f$-orbit of $M$, and let 
$$\mathcal A_{min}=\{C\in \CA \ | \ \nexists D \in\CA \text { such that } D \subsetneq C\} $$ be the subset consisting of all minimal elements of $\mathcal A$.  
\begin {theorem}[Existence of maximal subcompression body] \label {maxcompression}
$\mathcal A_{min}$ is a finite $f$-orbit that contains a single element $C_f$ such that $C_f \subset M $.  

Moreover, $C_f$ is the unique maximal element of $\mathrm{CBod}(S)$ such that $C_f \subset M$ and a power of $f $ extends to $C_f$. 
\end {theorem}

This result was proved in \cite {Biringerextending} when $S=\Sigma$.  Our proof follows the same general lines, but is topological instead of hyperbolic geometric.

We proceed with a series of lemmas.

\begin {lemma}\label {lemmaone}
$\mathcal A_{min}$ is nonempty, finite and $f $-invariant.
\end {lemma}
\begin {proof}
$\mathcal A$ is nonempty, since $\mathrm {CBod}(S)$ is compact.  This implies that $\mathcal A_{min}$ is nonempty, for example since the `height' of a compression body is nonnegative and decreases under strict containment, see \S 3 of \cite{Biringerautomorphisms}.

By Claim \ref {convergecontain}, $\mathcal A_{min}$ is discrete.  But $\mathcal A_{min}$ is closed in $\mathcal A, $ which is closed in $\mathrm {CBod}(S) $, which is compact. So, $\mathcal A_{min}$ is compact, and must be finite. Finally, $\mathcal A_{min}$ is $f$-invariant since $\mathcal A$ is and the $f $-action respects containment. 
\end {proof}

\begin {lemma}\label {lemmatwo}
Suppose that for $i=1,2$ we have $C_i\in \mathrm{CBod}(S) $ with $C_i\subset M $, and that $f^i(C_1)=C_1$ while $f^j(C_2)=C_2$. Then there is an element $C \in \mathcal A_{min} $ such that $C_1, C_2 \subset C \subset M  $.
\end{lemma}
\begin{proof}
Every element of $\CA$ is the image under a power of $f $ of an accumulation point of the sequence $f^{nij}(M)$, so since $\mathcal A_{min}$ is $f$-invariant there is some $C' \in \mathcal A_{min}$ to which a subsequence of $f^{nij}(M)$ limits.  As $C_1,C_2 \subset f^{nij}(M)$ for all $n $, we must have $C_1, C_2  \in C'$.

By Claim \ref {convergecontain}, there is some $n $ such that $f^{nij}(M) \supset C'$.  Then $C:=f^{-nij}(C') \in \mathcal A_{min} $ is contained in $M $ and must contain $C_1, C_2$ as well.
\end {proof}

\begin {lemma}\label {lemmathree}
There is a unique element $C_f \in \mathcal A_{min} $ that is contained in $M $, and $A_{min}$ is an $f$-orbit.
\end {lemma}
\begin {proof}
Applying the previous lemma to two copies of the trivial compression body $\Sigma \times I$ shows that $A_{min} $ has an element that is contained in $M $. 

So, suppose that $C,D \in A_{min} $ are both contained in $M $.  By the previous lemma, there is another element of $A_{min}$ that contains them both, which contradicts the minimality assumption unless $C = D $.  Therefore, there is a unique element $C_f \in \mathcal A_{min} $ that is contained in $M $.

To show that $A_{min}$ is an $f$-orbit, suppose that $C \in A_{min} $.  Since $C$ is an accumulation point, there is some $n $ such that $  f^n(M) \supset C$.  Then $f^{-n}(C) \subset M$, implying that $f^{-n}(C)=C_f $ by uniqueness.
\end {proof}

This finishes the proof of Theorem \ref{maxcompression}, since Lemma \ref {lemmaone} shows that a power of $f $ extends to $C_f $ and Lemmas \ref {lemmatwo} and \ref {lemmathree} imply that any subcompression body of $M$ to which a power of $f $ extends is contained in $C_f$.

\subsection{The proof of Theorem \ref{BJM}}\label{BJMsec}\label{BJMpf}
Let $S \subset \Sigma=\partial M_+$ be a compact essential subsurface, with $\partial S$ incompressible in $M$, and let $f: \Sigma \longrightarrow \Sigma$ be a pseudo-Anosov map on $S$. 

The `only if' direction of the theorem is trivial. Namely, suppose that some power $f^k$ extends to a nontrivial subcompressionbody $C$ of $(M,S)$. Pick a meridian $m\subset S$ for $C$. Then $(f^k(m))$ is a sequence of meridians in $M$ that lie in $S$, and converge to the attracting lamination of $f$.

 For the `if' direction of the theorem, assume that no nonzero power of $f$ extends to a nontrivial subcompression body of $(M,S)$. We must show that the attracting lamination $\lambda^+$ is not in the limit set $\Lambda(S,M)$. The argument is  similar to the proof of the main theorem in \cite {Biringerextending}.  As such, we will sketch the argument in places and refer to \cite {Biringerextending} for details.
 
 Consider the sequence  $M_n = f^{-n}(M)$ of marked $S$-compression bodies, where we consider the exterior boundary of each $M_n$ as identified with the surface $\Sigma$. Fix a base point $[X] \in \CT (\Sigma) $ and give the interior of each $M_n $ a geometrically finite hyperbolic metric such that  the end adjacent to the exterior boundary $\Sigma = \partial_+M$ is convex cocompact, and when its conformal boundary is identified with $\Sigma$, the conformal structure is $[X]$. 
Let $$\rho_n : \pi_1 \Sigma \longrightarrow \PSL_2 \BC, \ \ N_n:=\BH^3/\rho_n (\pi_1 \Sigma) $$ be a representation (unique up to conjugacy) uniformizing the interior of $M_n$ and compatible with our markings, in the sense that $\rho_n$ is the composition of the map $\pi_1 \Sigma \longrightarrow \pi_1 M_n\cong \pi_1 N_n$ induced by inclusion and a faithful uniformizing representation of $\pi_1 N_n$. Note that  the kernel of $\rho_n $ is $$\ker (\rho_n) = f^ { - n }_*(\ker( \pi_1 \Sigma \to \pi_1 M)). $$

By Theorem \ref {maxcompression} and the assumption that no power of $f $  extends to a nontrivial subcompression body of $M$, the only minimal accumulation point of $(M_n)$ in $\mathrm{CBod}(S)$ is the trivial compression body. So in particular, we can choose a subsequence $(M_{n_j})$ that converges to the trivial compression body.  By the compactness of generalized Bers slices (see \cite[Theorem 4.3]{Biringerextending}), we may assume after appropriate conjugations and passing to a further subsequence that $(\rho_{n_j}) $ converges algebraically to a representation $$\rho_\infty:\pi_1 \Sigma\longrightarrow\PSL_2\BC, \ \ N_{\infty} :=\BH^ 3/\rho_\infty(\pi_1 \Sigma) $$
and that $N_\infty $ can be identified with the interior of a compression body $M_\infty$ with exterior boundary $\Sigma$ in such a way that the end of $N_\infty$ adjacent to $\Sigma$ is convex cocompact with conformal structure $[X]$ and the representation $\rho_\infty$ is compatible with the marking in the same way as before.

The disk set $\mathcal D(S,M_\infty)$ consists of all simple closed curves on $S $ represented by elements $\gamma \in \pi_1 \Sigma$ with $\rho_\infty (\gamma)=1$.  By Chuckrow's Theorem (see \cite [Lemma 2.11]{Biringerextending}), $\rho_\infty(\gamma)=1$ if and only if $\rho_{n_j}(\gamma)=1$ for all sufficiently large $i $. Since $(M_{n_j})$ converges to the trivial compression body in the topology of $\mathrm{CBod}(S)$, it follows that the surface $S\subset \Sigma=\partial_+ M_\infty $ is incompressible in $M_\infty$.

\begin {claim}
The repelling lamination $\lambda  ^ - $ of $f$ is unrealizable in $N_\infty$.
\end {claim}

\begin {proof}
The proof is almost identical to that of \cite [Lemma 6.2] {Biringerextending}, so we offer a sketch and we refer the reader to their paper for details.  

Fixing an $M$-meridian $\gamma \subset S$, the sequence $f^{-n_j}(\gamma)$ converges in the Hausdorff topology to a lamination $\lambda_M$ that is the union of $\lambda  ^ - $ and finitely many leaves spiraling onto it.  It suffices by \cite [Theorem 2.3] {Brockcontinuity} to show that $\lambda_M $ is unrealizable.  So, hoping for a contradiction, assume $\lambda_M $ \emph {is} realizable; then $\lambda_M $ is carried by a train track $\tau $ that maps nearly straightly into $N_\infty $ (see \cite{Biringerextending}).  

By algebraic convergence, $\tau $ also maps nearly straightly into $N_{n_j}$ when $i $ is large.   Since $f^{-n_j}(\gamma) \to\lambda_M $, for large $i $ the curve $f^{-n_j}(\gamma)$ is carried by $\tau $.  This implies that $f^{-n_j}(\gamma)$ is geodesically realizable in $N_{n_j}$ for large $i $, contradicting the fact that it is homotopically trivial.
\end {proof}

By work of Thurston \cite [Prop 9.7.1] {Thurstongeometry}, the $\pi_1$-injective surface $S\subset \partial_+M_\infty$ is isotopic into a degenerate end of $N_\infty$ with ending lamination $\lambda  ^ -$.  In particular, the peripheral curves of  $S  $ represent cusps in $N_\infty$ and every non-peripheral curve on $S $ has hyperbolic type in $N_\infty $.  Any pair of disjoint non-peripheral simple closed curves on $S  $ can then be realized geodesically by a pleated surfaces $S  \longrightarrow N_\infty $ in the given homotopy class, and Thurston's compactness of pleated surfaces (see \cite [Lemma 6.13] {Matsuzakihyperbolic}) implies the following.

\begin{lemma}[compare with Lemma 6.3, \cite{Biringerextending}]
	\label{distanceboundslem}
Let $\alpha \subset S  $ be a simple closed curve.  Then for every $k $, there is some $K $ such that for any other simple closed curve $\beta $ in $ S  $, we have $$\dist_{\mathcal C (S )}(\alpha, \beta) \leq k \ \Longrightarrow \ \dist_{ N_\infty }(\alpha_\infty, \beta_\infty) \leq K,$$
	where $\alpha$ and $\beta_\infty $ are the geodesics in $N_\infty $ in the homotopy classes of $\alpha$ and $\beta$.
\end{lemma}

Hoping for a contradiction, suppose now that $\lambda^+ \in \Lambda (S , M) $.  When regarded as an element of $\partial_\infty \mathcal C (S ) $, the support of $\lambda^+ $ is then an accumulation point of $\mathcal D(S ,M) \subset \mathcal C (S ) $.  If $\alpha \in  \mathcal C (S ) $, then for $n=1,2,\ldots $ the sequence $(f ^ n(\alpha))$ is a quasi-geodesic path that limits to $\lambda^+  \in\partial_\infty \mathcal C (S ) $, see \cite{Masurgeometry1}.   Since $\mathcal D(S ,M)$ is a quasi-convex subset (see Theorem \ref {quasi-convexity}, due to Masur-Minsky), there is a constant $C $ and for each $n $ a meridian $\gamma_i \in \mathcal D(S ,M) $ with
$$d_{\mathcal C (S ) }(f ^ n(\alpha),\gamma_i) \leq C.$$
Translating the points $f^n( \alpha)$ and $\gamma_i$ by $f^{-n}$, this becomes:
$$d_{\mathcal C (S ) }(\alpha,f^{-n}(\mathcal D(S ,M)) \leq C.$$
By Lemma \ref {distanceboundslem}, an element $\gamma_{n_j} \in f^{-n_j}(\mathcal D(S ,M))$ at distance at most $C$ from $\alpha $ in $\mathcal C (S ) $ can be geodesically realized in some fixed compact subset $A \subset N_\infty$.  Algebraic convergence implies that for sufficiently large $j $ this geodesic can be pulled back and tightened to a geodesic in $N_{n_j}$.  But by construction, $\gamma_{n_j}$ is a meridian in $M_{n_j}$, so it cannot possibly be realized geodesically in $N_{n_j}$, which is a contradiction.

\section{Extending reducible maps to compression bodies}

\label {sec:extension}
We present here a generalization of \cite [Theorem 1.1] {Biringerextending} that characterizes which (possibly reducible) mapping classes of the boundary of a $3$-manifold $M$ have powers that extend to sub-compression bodies.  

In what follows, let $M $ be a compression body with exterior boundary $\partial_+M$. Let $S\subset \partial_+M$ be an essential subsurface such that $\partial S$ is incompressible. Let $f: \partial_+ M \longrightarrow \partial_+ M$ be a homeomorphism that is `supported' in $S$, meaning that $f=id$ on $\partial_+ M\setminus S$.

\begin{definition}
We say that $f$ is \emph{pure} if there are disjoint, compact, essential $f$-invariant subsurfaces $S_i \subset S$, $i=1,\ldots, n$, such that $f=id$ on $S_{id}:=S \setminus \cup_i S_i$, and where for each $i$, if we set  $f_i: = f | _{S_i}$, then either
\begin {enumerate}
\item $S_i $ is an annulus and $f_i $ is a power of a Dehn twist, or
\item $f_i $ is a pseudo-Anosov map on $S_i$.
\end {enumerate}
\end{definition}
It follows from the Nielsen-Thurston classification, see \cite{Farbprimer}, that every $f$ has a power that is isotopic to a pure homeomorphism.

\medskip

When $f$ is pure, with associated restrictions $f_i : S_i \longrightarrow S_i$ as above, we define a geodesic lamination $\lambda=\cup_i\lambda_i$ on $S$, where $\lambda_i \subset S_i$ as follows. If $f_i$ is pseudo-Anosov, we let $\lambda_i$ be the support of the attracting lamination of $f_i$. If $f_i$ is a Dehn twist, we let $\lambda_i$ be the core curve of the annulus $S_i$. So defined, $\lambda$ is called the \emph{attracting lamination} of the pure homeomorphism $f$.

\begin{theorem}\label{extension}
Suppose that $S \subset \partial_+ M$ is  an essential subsurface such that the multicurve $\partial S$ is incompressible. Let $f : \partial_+ M \longrightarrow \partial_+ M$ be a pure homeomorphism supported in $S$. Then $f$ has a power that extends to a nontrivial subcompressionbody of $(M,S)$ if and only if either: 
	\begin{enumerate}
		\item there is a meridian in $S_{id}$, 
		\item for some $i$, the map $f_i$ has a power that extends to a nontrivial subcompressionbody of $(M,S_i)$, or
		\item there are (possibly equal) indices $i,j$ such that $S_i,S_j$ bound an essential interval bundle $B$ in $M$, such that some power of $f|_{S_i \cup S_j}$ extends to $B$, and  there is a compression arc $\alpha$ for $B$ whose interior lies in $S_{id}$. \end{enumerate}
\end{theorem}

Recall from \S \ref{compressionsec} that a `subcompressionbody of $(M,S)$ is a compression body obtained from $\partial_+ M$ by compressing some meridian multicurve in $S$. 

Theorem \ref{BJM} says that (2) is equivalent to asking that the attracting lamination (say) of $f$ is a projective limit of meridians in $S_i$. In (3), the condition that a power of $f|_{S_i\cup S_j}$ extends to $B$ is easier to check. Indeed, if $\sigma : \partial_H B \longrightarrow \partial_HB $ is the canonical involution, as defined in \S \ref{sec:ibundle}, then by Fact \ref{extendfact} we have that $f^k|_{\partial_HB}$ extends to $B$ exactly when $\sigma \circ f_i^k $ is isotopic to $f_j^k$. When $B$ is a twisted interval bundle, $f_i,f_j$ are both pseudo-Anosovs on $\partial_HB$ and this means that as mapping classes we have $f_j = g \circ f_i$ for some finite order $g$ commuting with both $f_i,f_j$, see e.g.\ McCarthy's thesis \cite{mccarthy1982normalizers}. When $B$ is a trivial interval bundle, $\sigma$ indentifies $S_i$ and $S_j$, and we have similarly that $f_j=g \circ \sigma(f_i)$ for some $g$ commuting with both.

\begin{proof}[Proof of Theorem \ref{extension}]
	Let's start with the `if' direction, since that's easier. If there is a meridian in $S_{id}$, then $f$ extends to the compression body obtained by compressing that meridian. Suppose (2) holds, so that some power $f_i^k$ extends to a nontrivial subcompression body $C$ of $(M,S_i)$. Then $f$ also extends to $C$, since all the $S_j$, where $j\neq i$, bound trivial interval bundles with subsurfaces of the interior boundary of $C$. So we're done. 
	
	The only interesting case is if (3) holds, so that some $S_i,S_j$ bound an essential interval bundle $B$ in $M$ such that some power of $f^k|_{S_i \cup S_j}$ extends to $B$, and  there is a compression arc $\alpha$ for $B$ whose interior lies in $S_{id}$. Here, let $S' \subset S$ be the smallest essential subsurface containing $S_i,S_j$ and $\alpha$; so, $S'$ is obtained from a regular neighborhood of the union of these three subsets of $S$ by capping off any inessential boundary components with discs. Let $C$ be the characteristic compression body of the pair $(M,S')$, as defined in Fact \ref{charcomp}. 
	
	We claim that $f^k$ extends to $C$. To see this, note that we can construct $C$ as follows. For concreteness, first assume that the boundary components of $S_i,S_j$ that contain the endpoints of $\alpha$ bound an annulus $A \supset \alpha$ on $S$. Then $S'=S_i\cup S_j \cup A$, the annulus $A$ is parallel in $M$ to component $A'\subset \mathrm{Fr}(B)$ that is an annulus with the same boundary curves as $A$, and $C$ is the union of the interval bundle $B$, the solid torus bounded by $A,A'$, and a trivial interval bundle over $\partial_+ M \setminus S'$. We can then extend $f^k$ to $C$ by letting it be the given extension of $f^k|_{S_i\cup S_j}$ on $B$, the identity on the solid torus, and the obvious fiber preserving extension of $f^k |_{\partial_+ M \setminus S'}$ to the adjacent interval bundle. The case that the boundary components of $S_i,S_j$ that contain the endpoints of $\alpha$ do not bound an annulus on $S$ is similar, except that instead of the solid torus above we take a thickened disk bounded by a rectangular neighborhood of $\alpha$ on $S\setminus (S_i\cup S_j)$, and a rectangular neighborhood of the homotopic arc on the frontier of $B$. 
	
	\medskip

We now work on the `only if' direction. Passing to a power, suppose that $f$ extends to a nontrivial subcompressionbody $C $ of $(M,S)$. We may assume that there is no proper, $f$-invariant essential subsurface $S' \subset S$ such that $f|_{S'}$ extends to a nontrivial subcompression body of $(M,S')$. If there were such a subsurface $S'$, we could replace $S$ by a minimal such $S'$, therefore reducing the argument to the minimal case we are assuming we are in above.

If $f=id $ on $S$, the fact that there is a nontrivial subcompressionbody of $(M,S)$ means there is a meridian in $S=S_{id}$, so we're in case (1) and are done. This case may seems silly, but observe that if $f$ is some complicated pure homeomorphism where there's a meridian in $S_{id}$, the associated `minimal' case that we pass to in the previous paragraph is where $S$ is an annular neighborhood of some such meridian, and $f=id $ on $S$. 

\smallskip

Assume from now on that $f$ is not the identity map of $S$.

\smallskip

We claim that every meridian $m\in \mathcal D(C,S)$ intersects every component of $\lambda$. Indeed, suppose some $\lambda_i$ is disjoint from some such $m$ and let $S'$ be the component of $S \setminus S_i$ containing $m$. Since $f$ extends to $C$ and $S' \subset \partial C_+$ is $f^k$-invariant, $f$ extends to the characteristic subcompressionbody $C'$ of the pair $(C,S')$, defined by starting with $\partial_+ M$ and compressing all meridians of $C$ that lie in $S'$, see Fact \ref{charcomp}. Since $m$ is a meridian in $C'$, this $C'$ is nontrivial, which contradicts the mimimality assumption in the first paragraph.

Pick a meridian $m\in \mathcal D(C,S)$. Since $m$ intersects all components of $\lambda$,  the sequence of meridians $m_i:=f^{i}(m)$ Hausdorff converges to a lamination $\lambda'$ strongly commensurable to $\lambda$. Applying Theorem \ref{hlimits} to the pair $(C,S)$ and using that all meridians in $C$ intersect all components of $\lambda$, we have that either:
\begin{itemize}
\item some component $\lambda_a$ is an intrinsic limit of meridians {lying in $S(\lambda_a)$}, in which case Proposition~\ref{BJM} (applied to $f_a : S_a \longrightarrow S_a$) says we're in case (2), or 
\item there are indices $a,b$ such that $S_a,S_b$ bound an essential interval bundle $B \subset C$, where $\lambda_a,\lambda_b$ are essentially homotopic in $B$, and where there is a compression arc $\alpha \subset S$ for $B$ that is disjoint from $\lambda$, and hence can be isotoped so that its interior lies in $S_{id}$.
\end{itemize} 
Let's assume we're in the last case, since otherwise we're done. We want to show that some power of  $f_a \cup f_b : \partial_H B \longrightarrow \partial_HB$ extends to $B$.

 First, suppose that $B $ is a twisted interval bundle, so that $S_a=S_b, \ f_a=f_b, \ \lambda_a=\lambda_b$. Using just the index $a$ from now on, if $\sigma$ is the canonical involution of $B$, then Fact \ref{homotopy involutoin} implies that $\sigma(\lambda_a)$ is isotopic to $\lambda_a$. Let $A \subset T(S)$ be the axis of $f_a$ on the Teichm\"uller space $T(S)$. By Theorem 12.1 of \cite{Fathitravaux} and Theorem 2 of \cite{masur1980uniquely}, we have that $A,\sigma(A)$ are asymptotic, so as they are both pseudo-Anosov axes they must be equal by discreteness of the action of the mapping class group. Since $\sigma$ has finite order, it then fixes $A$ pointwise. Now the group $\Gamma= \langle \sigma,f_a\rangle \subset \Mod(\partial_HB)$ is isomorphic to the direct product of a finite group fixing $A$ pointwise and a cyclic group of pseudo-Anosovs, so for some positive $k$ we have $\sigma \circ f_a^k = f_a^k$ in $\Mod(\partial_H B)$. By Theorem 3 of \cite{birman1973isotopies} we may isotope $f_a^k,\sigma$ so that they commute, while preserving that $\sigma^2=id$; we can then alter the bundle map $\pi : B \longrightarrow Y$ so that the new $\sigma$ is still the canonical involution. It follows that $f_a^k $ is a lift to $\partial_HB$ of a pseudo-Anosov map $g: Y \longrightarrow Y$, and hence $f_a^k $ extends to $B$ as desired, see Fact \ref{extendfact}.

Next, assume $B$ is a trivial interval bundle, with canonical involution $\sigma$ that switches $S_a,S_b$. As in the previous paragraph, we have that $\sigma(\lambda_b)$ is isotopic to $\lambda_a$, so $\Gamma = \langle f_a, \sigma \circ f_b \circ \sigma^{-1}\rangle  \subset \Mod(S_a)$ is a direct product $\Gamma = F \times \langle \phi \rangle$ of a finite group $F$ and a cyclic group generated by a pseudo-Anosov $\phi$, where if we quotient by $F$ then $f_a$ and $ \sigma \circ f_b \circ \sigma^{-1}$ both project to positive powers of $\phi$. It suffices to show that they project to the \emph{same} positive power of $\phi$, for then we are done by the same argument as in the previous paragraph. For this, recall that all meridians of $C$ intersect $S_a,S_b$, so these surfaces are `holes' for the disk set of $C$, as discussed in \cite{masur2013geometry}. So with $m_i=f^i(m)$ the sequence of meridians in $C$ constructed above, Lemma 12.20 of \cite{masur2013geometry} says that for each $i$, the distance in the arc complex of $S_a$ between $$m_i \cap S_a = f_a^i(m \cap S_a) \ \ \text{and} \ \ \sigma(m_i\cap S_b) = (\sigma \circ f_b \circ \sigma^{-1})^i \sigma(m\cap S_b) $$ is at most 6. However, if $f_a$ and $ \sigma \circ f_b \circ \sigma^{-1}$ project to different positive powers of $\phi$, their stable translation lengths on the arc complex of $S_a$ are different, which is a contradiction.
\end{proof}

\bibliographystyle{amsplain}
\bibliography{../total}

\end{document}